\documentclass[12pt]{amsart}
\usepackage{mathrsfs}
\usepackage{amsfonts}
\usepackage[centertags]{amsmath}
\usepackage{amssymb}
\usepackage{amsthm}
\usepackage{graphicx}
\usepackage{float}
\usepackage[all]{xy}
\usepackage{tikz-cd}
\usepackage{caption}
\usepackage[colorlinks]{hyperref}
\hypersetup{
bookmarksnumbered,
pdfstartview={FitH},
breaklinks=true,
linkcolor=blue,
urlcolor=blue,
citecolor=blue,
bookmarksdepth=2
}

\usepackage{enumerate}
\usepackage[a4paper, portrait, margin=0.9in]{geometry}
\usepackage{rotating}
\usetikzlibrary{matrix,arrows,decorations.pathmorphing}
\usepackage{listing}
\usepackage{color}
\usepackage{adjustbox}
\usepackage{array}
\usepackage{combelow}

\theoremstyle{plain}
   
   \newtheorem{theorem}{Theorem}[section]
   \newtheorem{proposition}[theorem]{Proposition}
   
   \newtheorem{lemma}[theorem]{Lemma}
   \newtheorem{corollary}[theorem]{Corollary}

\theoremstyle{definition}
   \newtheorem{definition}[theorem]{Definition}

   \newtheorem{remark}[theorem]{Remark}

\numberwithin{equation}{section}

\DeclareMathOperator{\id}{id}

\DeclareMathOperator{\End}{End}

\newcommand\scalemath[2]{\scalebox{#1}{\mbox{\ensuremath{\displaystyle #2}}}}

\usetikzlibrary{arrows,automata,positioning}

\title{Shifted quantum affine superalgebras} 
\author{Piao-Juan Hu, Jian-Rong Li, Yan-Feng Luo$^\dag$}\thanks{$\dag$ Corresponding author}
\address{Piao-Juan Hu: School of Mathematics and Statistics, Lanzhou University, Lanzhou 730000, P. R. China.}
\email{hupj21@lzu.edu.cn}
\address{Jian-Rong Li: Faculty of Mathematics, University of Seville, Calle Tarfia s/n, 41012 Seville, Spain}
\email{lijr07@gmail.com}
\address{Yan-Feng Luo: School of Mathematics and Statistics, Lanzhou University, Lanzhou 730000, P. R. China.}
\email{luoyf@lzu.edu.cn}
\date{}

\begin{document}
\allowdisplaybreaks[3]
\raggedbottom
\emergencystretch=2em

\begin{abstract}
For a quantum affine superalgebra with a fixed Drinfeld current presentation, we study shifted quantum affine superalgebras through a two-sided shifted Drinfeld current presentation, in which the leading degrees of the Cartan--Drinfeld currents are modified while the remaining current and Serre relations are retained.

We then focus on the standard-parity type A algebras
$\mathcal U_q^{\mu_+,\mu_-}(\widehat{\mathfrak{sl}}_{m|n})$.
For arbitrary shifts and $m,n\ge1$, $m\ne n$, we prove a triangular decomposition. When $m,n>1$, we obtain PBW--Drinfeld bases for arbitrary shifts and a Levendorskii-type finite presentation of the one-sided algebra $\mathcal U_q^{0,\mu}$ for every antidominant $\mu$. For antidominant two-sided shifts, we construct shifted RTT realizations and prove the corresponding Ding--Frenkel isomorphisms.

We construct a shifted formal Drinfeld coproduct for arbitrary two-sided shifts, coassociative in a completed tensor product. Shifted RTT realizations yield non-completed one-sided coproducts in the antidominant region, made explicit on finitely many generators by the finite presentation; injective shift homomorphisms extend them to arbitrary one-sided shifts. Their ternary coassociativity holds whenever the middle shift is antidominant, including boundary cases. Finally, for two-sided antidominant shifts $\mu_+,\mu_-\in\Lambda^-$, a shifted full general-linear Ding--Frenkel isomorphism identifies the special-linear simple-current algebra as the fixed algebra under common scalar rescalings and yields coassociative two-sided non-completed coproducts.
\end{abstract}


\maketitle

\setcounter{tocdepth}{1}
\tableofcontents

\section{Introduction}

\subsection{Background and motivation}

Quantum affine superalgebras are quantum deformations of affine Lie superalgebras and occur naturally in the study of graded Yang--Baxter equations and quantum integrable systems \cite{BGZ1990,FLV1991,ZGB1991}.  Their structure is described through several complementary presentations.  In untwisted type A, Drinfeld--Jimbo and Drinfeld current realizations are available for arbitrary parity sequences \cite{Yam1999,BM2021b}, while RTT realizations of the general-linear quantum affine superalgebra and their relation with Drinfeld currents were developed in \cite{FHS1997,GZ1998,Zyz1997,Zh2016}.  Since the shifted construction studied here is formulated at the level of Drinfeld currents, the choice and availability of a current presentation are essential.  Section~\ref{section: quantum affine superalgebra} recalls the precise status needed in this paper, including the corresponding results for several other affine super types.

Shifted quantum affine algebras in the non-super setting were introduced by Finkelberg and Tsymbaliuk in connection with quantized K-theoretic Coulomb branches of three-dimensional gauge theories \cite{FT2019}.  Their representation theory and truncations were subsequently developed by Hernandez \cite{Her2023}, and the relation between Borel and shifted category $\mathcal O$ was further studied in \cite{HN2026}.  Cluster-algebra structures on Grothendieck rings were established in \cite{GHL2024}; further links with $R$-matrices, Hall algebras, and Lax-matrix constructions appear in \cite{Pin2025,GS2025,FPT2022}.  A twisted counterpart of the shifted quantum affine theory was recently introduced in \cite{LLL2026}. Related developments in the quantum symmetric-pair direction include shifted affine $\imath$-quantum groups of split ADE and their GKLO-type representations \cite{LiPrzezdziecki2026}, as well as shifted affine $\imath$-quantum groups of quasi-split ADE type, together with their GKLO-type representations and truncations, in \cite{LWW2026}.  Negu\cb{t} and Wang \cite{NW2026} develop folded and twisted quantum loop/toroidal algebras in a shuffle framework, including shifted folded quantum loop algebras; in finite type this recovers the shifted twisted quantum affine setting. Related shifted toroidal directions include shifted toroidal and elliptic quiver BPS algebras \cite{GLY2022}, shifted quiver quantum toroidal algebras \cite{NW2022}, shifted quantum toroidal $\mathfrak{gl}_{1}$ and quantum affine $\mathfrak{sl}_{2}$ algebras \cite{Bou2023}, and the appearance of a shifted quantum toroidal algebra in the cyclotomic DAHA setting \cite{Wen2024}.  A recent super example involving a shifted quantum toroidal algebra of type $\mathfrak{gl}_{1|1}$ appears in \cite{KOS2026}.  These developments show that shifted quantum affine and toroidal algebras connect quantum group representation theory with geometry, integrable systems, cluster algebras, Hall algebras, and related gauge-theoretic constructions.

Shifted quantum loop and toroidal superalgebras were defined by Negu\cb{t} and Tsymbaliuk \cite[Remarks~2.6 and~2.7]{NT2026} using shifted Drinfeld doubles.  From the current-presentation point of view, the shifted construction is well suited to the super setting: the shift is implemented by changing the leading degrees of the Cartan--Drinfeld currents while leaving the root-current and Serre relations unchanged. For the purposes of this paper, in Definition~\ref{def:general-shifted-qas} we formulate shifted quantum affine superalgebras in terms of a two-sided shifted current presentation associated with a fixed Drinfeld current presentation of the underlying quantum affine superalgebra. Proposition~\ref{prop:comparison-NT-general} shows that, whenever the corresponding zero-shift current presentation fits the Drinfeld super-double framework of \cite{NT2026}, the two shifted constructions agree after the corresponding reindexing of the currents and adjoining inverses to the leading Cartan generators.

The main part of the paper concerns the standard-parity shifted quantum affine superalgebras
$\mathcal U_q^{\mu_+,\mu_-}(\widehat{\mathfrak{sl}}_{m|n})$.
Here genuinely super phenomena enter the structural arguments. The distinguished simple root $\alpha_m$ is isotropic, so $(\alpha_m,\alpha_m)=0$ and the usual mode recursion using $h_{m,\pm1}$ degenerates; in the finite presentation we instead propagate the odd current through a neighboring even Cartan direction.  The degree-four odd Serre relation must be controlled simultaneously in several loop indices.  In the RTT realization, graded signs enter the Gauss decomposition and the finite-generator formulas separate naturally into the three regimes $i<m$, $i=m$, and $i>m$.  Thus the results below are not obtained by merely inserting parity signs into the non-super theory.

\subsection{Main results}

We now summarize the main results.

\medskip\noindent 
\textbf{Definition A (two-sided shifted current presentation).}
For a quantum affine superalgebra equipped with a fixed Drinfeld current presentation and integral shift data $\mathbf b^\pm$, Definition~\ref{def:general-shifted-qas} gives a two-sided shifted current presentation of the shifted quantum affine superalgebra associated with these data.

For the remainder of this subsection, we write
$\mathcal U_q^{\mu_+,\mu_-}:=
\mathcal U_q^{\mu_+,\mu_-}(\widehat{\mathfrak{sl}}_{m|n})$
for the standard-parity type A shifted quantum affine superalgebras.

\medskip\noindent
\textbf{Theorem B (triangular decomposition and PBWD bases).}
For $m,n\geq1$ with $m\ne n$ and arbitrary $\mu_+,\mu_-\in\Lambda$, multiplication gives a vector-superspace isomorphism
\[
(\mathcal U_q^{\mu_+,\mu_-})^{<}\otimes
(\mathcal U_q^{\mu_+,\mu_-})^{0}\otimes
(\mathcal U_q^{\mu_+,\mu_-})^{>}
\xrightarrow{\ \sim\ }
\mathcal U_q^{\mu_+,\mu_-}.
\]
The positive and negative factors are independent of the shift and have the same current presentations as the corresponding halves of the unshifted quantum affine superalgebra
$\mathcal U_q(\widehat{\mathfrak{sl}}_{m|n})$.
When $m,n>1$, the PBW--Drinfeld bases of \cite{Tsy2021} therefore give a PBWD basis of $\mathcal U_q^{\mu_+,\mu_-}$; see Theorem~\ref{thm:shifted-triangular} and Corollary~\ref{cor:shifted-pbw}.  

The approach of Negu\cb{t} and Tsymbaliuk \cite[Remarks~2.6 and~2.7]{NT2026} via shifted Drinfeld super-doubles and shuffle realizations provides alternative proofs of the triangular decomposition and the PBWD basis statement after matching the shift parameters and current conventions.

\medskip\noindent
\textbf{Theorem C (Levendorskii-type finite presentation).}
Assume $m,n>1$ and $m\ne n$.  For every antidominant $\mu\in\Lambda^{-}$ and every decomposition $\mu=\mu_1+\mu_2$ with $\mu_1,\mu_2\in\Lambda^{-}$, Theorem~\ref{thm:finite iso} gives a finite presentation of
$\mathcal U_q^{0,\mu}(\widehat{\mathfrak{sl}}_{m|n})$ in terms of finitely many root-current modes, the leading Cartan coefficients, and the modes $h_{i,\pm1}$.  All other current modes and Cartan coefficients are recovered recursively.  The isotropic node requires a separate recursion through a neighboring even Cartan direction.

\medskip\noindent
\textbf{Theorem D (shifted formal Drinfeld coproduct).}
For $m,n\geq1$ with $m\ne n$, the formal Drinfeld current formulas extend to arbitrary two-sided shifts and define a coassociative coproduct in a completed tensor product of shifted quantum affine superalgebras; see Theorem~\ref{thm:formal-shifted-coproduct} and Proposition~\ref{prop:formal-shifted-coassociativity}.  This construction is independent of the Levendorskii-type finite presentation of Theorem~\ref{thm:finite iso}; the latter re-enters only for the non-completed one-sided coproducts of Section~\ref{section: shifted RTT and coproducts}, where it makes the transported RTT coproduct explicit on a finite generating set.

\medskip\noindent
\textbf{Theorem E (injective shift homomorphisms).}
Assume $m,n>1$ and $m\ne n$.  For $\mu\in\Lambda$ and $\nu_1,\nu_2\in\Lambda^{-}$, the current assignments
\[
x_i^+(z)\mapsto(1-z^{-1})^{-\alpha_i(\nu_1)}x_i^+(z),\qquad
x_i^-(z)\mapsto(1-z^{-1})^{-\alpha_i(\nu_2)}x_i^-(z),
\]
together with the corresponding assignment for $\phi_i^{\pm}(z)$, define an injective homomorphism
\[
\iota_{\mu,\nu_1,\nu_2}:\mathcal U_q^{0,\mu}\longrightarrow
\mathcal U_q^{0,\mu+\nu_1+\nu_2}.
\]
These homomorphisms satisfy the natural composition law; see Theorem~\ref{th:iota mu nu1 nu2 is injective}.  Their injectivity is proved using the triangular decomposition together with the shuffle realization of the shift-independent positive and negative halves from \cite{Tsy2021}.

\medskip\noindent
\textbf{Theorem F (non-completed one-sided coproducts).}
Assume $m,n>1$ and $m\ne n$.  For arbitrary $\mu_1,\mu_2\in\Lambda$, there is a unique superalgebra homomorphism
\[
\Delta_{\mu_1,\mu_2}:\mathcal U_q^{0,\mu_1+\mu_2}
\longrightarrow
\mathcal U_q^{0,\mu_1}\otimes\mathcal U_q^{0,\mu_2}
\]
characterized by compatibility with the injective shift homomorphisms of Theorem~\ref{th:iota mu nu1 nu2 is injective}; see Theorem~\ref{Th: commutative diagram}.  For strictly antidominant $\mu_1,\mu_2$, the Levendorskii-type finite presentation of Theorem~\ref{thm:finite iso} reduces the coproduct to finitely many boundary generators, whose images are given explicitly in Proposition~\ref{prop:explicit-strict-antidominant-one-sided-coproduct}.  The associated ternary identity holds whenever the middle shift is antidominant, including boundary points of the antidominant chamber, i.e. allowing $\alpha_i(\mu_2)=0$ for some $i$; see Theorem~\ref{thm:boundary-rtt-current-coassociativity}.  At zero shift, after imposing the central reduction $\phi_{i,0}^+\phi_{i,0}^-=1$, the homomorphism $\Delta_{0,0}$ descends to the ordinary coproduct on $\mathcal U_q(\widehat{\mathfrak{sl}}_{m|n})$; consequently it induces the usual tensor-product action on modules.

\medskip\noindent
\textbf{Theorem G (shifted RTT realizations and two-sided antidominant coproducts).}
Assume $m,n>1$ and $m\ne n$.  For every $\mu_+,\mu_-\in\Lambda^{-}$, choose weakly decreasing lifts $\mathbf d^+,\mathbf d^-$ as in Section~\ref{subsec:shifted-super-rtt}.  The simple Gauss currents of the shifted RTT superalgebra $\mathscr U^{\mathrm{rtt}}_{\mathbf d^+,\mathbf d^-}$ associated with the Perk--Schultz $R$-matrix \cite{PS1981} recalled in Section~\ref{subsec:rtt-gauss} define an isomorphism
\[
\mathcal U_q^{\mu_+,\mu_-}(\widehat{\mathfrak{sl}}_{m|n})
\xrightarrow{\ \sim\ }
\mathscr U^{\mathrm{rtt},\mathrm{sl}}_{\mathbf d^+,\mathbf d^-};
\]
see Theorem~\ref{thm:shifted-super-df-isomorphism}.  Injectivity is proved by a three-factor RTT coefficient argument with two zero-shift outer factors.  Theorem~\ref{thm:shifted-full-gl-df-isomorphism} gives the corresponding shifted full general-linear quantum affine superalgebra Ding--Frenkel isomorphism, and Proposition~\ref{prop:shifted-rtt-scalar-fixed-algebra} identifies the special-linear simple-current subalgebra as the fixed algebra under common scalar rescalings.  The RTT matrix coproduct therefore restricts to the simple-current subalgebras and transports to the coassociative two-sided antidominant coproducts of Theorem~\ref{thm:two-sided-antidominant-current-coproduct}.

\subsection{Proof strategy and organization}

The main results above use three complementary mechanisms.  The Levendorskii-type finite presentation reconstructs all modes from finitely many generators, with a separate propagation argument at the isotropic root.  Independently, the formal Drinfeld current formulas give the shifted formal coproduct in a completed tensor product, while the injective shift homomorphisms are obtained using the structural properties of the shifted current presentation together with the shuffle realization.  On the RTT side, the shifted RTT algebras and the three-factor coefficient argument give the shifted Ding--Frenkel isomorphisms.  Transporting the RTT matrix coproduct through these isomorphisms produces the antidominant one- and two-sided non-completed coproducts.  For one-sided antidominant shifts, the finite presentation then makes this transported coproduct explicit on finitely many generators, using the zero-shift RTT coefficient calculations of Appendix~\ref{app:rtt-coproduct-proof} as a template; the injective shift homomorphisms finally extend the one-sided coproduct to arbitrary shifts.

Section~\ref{section: quantum affine superalgebra} fixes the Drinfeld--Jimbo, Drinfeld, and RTT conventions for the unshifted quantum affine superalgebra.  Section~\ref{section: Shifted quantum affine superalgebra} presents the two-sided shifted current presentation used in this paper and establishes the structural results needed in the subsequent sections.  Section~\ref{section: Levendorskii type presentation of shifted quantum affine superalgebra} establishes the Levendorskii-type finite presentation for the one-sided antidominant algebra $\mathcal U_q^{0,\mu}$.  Section~\ref{section: Coproduct} constructs the shifted formal Drinfeld coproduct in a completed tensor product and proves the injectivity of the shift homomorphisms.  Section~\ref{section: shifted RTT and coproducts} develops the shifted RTT realizations, the shifted Ding--Frenkel isomorphisms, the antidominant one- and two-sided non-completed coproducts, the finite-generator formulas in the one-sided antidominant case, and the extension of the one-sided coproduct to arbitrary shifts.  The appendices contain the longer finite-presentation proofs, the zero-shift RTT coefficient calculations used as a template for the shifted formulas, and the finite-generator transport calculations.

\section{Quantum affine superalgebras}\label{section: quantum affine superalgebra}
We review the Drinfeld current presentations available for quantum affine superalgebras and then specialize to the Drinfeld--Jimbo, Drinfeld, and RTT realizations of $U_q(\widehat{\mathfrak{sl}}_{m|n})$ used below.

\subsection{Quantum affine superalgebras and Drinfeld realizations}
Let $\widehat{\mathfrak g}$ be an affine contragredient Lie superalgebra of finite growth, and let $\Pi$ be a choice of simple roots. In the sense of Yamane \cite{Yam1999}, the corresponding quantum affine superalgebra $U_q^{\mathrm{DJ}}(\widehat{\mathfrak g},\Pi)$ is defined by Chevalley generators attached to $\Pi$, together with the Cartan, Chevalley, and parity-dependent Serre and higher-order Serre relations. Different choices of $\Pi$ may give different Cartan data for the same Lie superalgebra, so the chosen simple system is part of the presentation. Yamane established this Drinfeld--Jimbo framework for the affine Lie superalgebras of finite growth and obtained the Drinfeld new realization for untwisted type A.

Unlike the purely even case, there is not yet a single Drinfeld current presentation known to be equivalent to the Drinfeld--Jimbo presentation for every affine super type and every choice of simple roots.  The principal cases relevant to the definition in Section~\ref{section: Shifted quantum affine superalgebra} are as follows.  For the quantum affine superalgebra $U_q(\widehat{\mathfrak{sl}}_{m|n})$, the new Drinfeld presentation is available for arbitrary parity sequences and is isomorphic to the corresponding Drinfeld--Jimbo presentation; see \cite{Yam1999} and the formulation in \cite[Section~2 and Proposition~3.3]{BM2021b}.  Closely related current and RTT realizations of $U_q(\widehat{\mathfrak{gl}}_{m|n})$ were studied in \cite{FHS1997,GZ1998,Zyz1997}, and a complete type A $R$-matrix--Drinfeld isomorphism was recently established in \cite{XLZ2026}.  For the quantum affine superalgebra of type $D^{(1)}(2,1;x)$, a Drinfeld second realization and its isomorphism with the Drinfeld--Jimbo realization were established in \cite{HSTY2008}.  Drinfeld realizations of the quantum affine superalgebras of types $\mathfrak{osp}(1|2n)^{(1)}$, $\mathfrak{sl}(1|2n)^{(2)}$, and $\mathfrak{osp}(2|2n)^{(2)}$ were constructed in \cite{XZ2016}; the rank-one case $U_q(\widehat{\mathfrak{osp}}(1|2))$ had appeared earlier in \cite{Ding1999,DF2000}.  For $U_q(\mathfrak{osp}(2m+1|2n)^{(1)})$ with the standard parity sequence, Wu--Lin--Zhang \cite{WLZ2025} proved the isomorphism between the Drinfeld and Drinfeld--Jimbo presentations.  For orthosymplectic quantum affine superalgebras of types $B,C,D$ with arbitrary parity, Bezerra--Futorny--Kashuba \cite{BFK2026} formally introduced new Drinfeld current algebras and proved that the corresponding Drinfeld--Jimbo quantum affine superalgebras are surjective homomorphic images; they conjecture that these maps are isomorphisms.

\subsection{Cartan data}
We use the standard $\mathbb Z_2$-graded conventions for Lie superalgebras; see \cite{Kac1977}.  The parity of a homogeneous element $x$ is denoted by $|x|\in\mathbb Z_2$, and all brackets below are supercommutators unless a $q$-parameter is displayed explicitly.

For $m,n\in\mathbb{Z}_{>0}$ with $m\neq n$ and $M=m+n\geq 3$, consider the Lie superalgebra $\mathfrak{sl}_{m|n}$.  We use three index sets throughout:
\[
I:=\{1,\dots,M-1\},\qquad
\widehat I:=\mathbb Z/M\mathbb Z,\qquad
J:=\{1,\dots,M\}.
\]
Thus $I$ indexes the finite simple roots, $\widehat I$ indexes the affine simple roots, and $J$ indexes the standard basis of the vector representation and the rows and columns in the RTT realization.  We represent the affine vertex by $0\in\widehat I$.  Let $\widehat{\mathfrak{sl}}_{m|n}$ be the corresponding affine Lie superalgebra.  Since different choices of the root system can lead to different Cartan matrices and Dynkin diagrams, these choices are parametrized by $M$-tuples of $\pm1$ having exactly $m$ entries equal to $1$. Set
\begin{align*}
\mathcal{S}_{m|n}=\{(s_{1},\cdots,s_{M}) \mid s_{a}\in\{1,-1\},\#\{a \mid s_{a}=1\}=m\}.
\end{align*}
An element $\mathbf{s}=(s_{1},\cdots,s_{M})\in\mathcal{S}_{m|n}$ is called a parity sequence.  We extend its entries cyclically by $s_{a+M}=s_a$ for all $a\in\mathbb Z$; in particular, $s_0=s_M$.  All affine Dynkin indices are understood modulo $M$.  In particular, $\mathbf{s}=(1,\cdots,1,-1,\cdots,-1)$ is called the standard parity sequence.

Given a parity sequence $\mathbf{s}\in\mathcal{S}_{m|n}$, the corresponding Cartan matrix $A^{\mathbf{s}}=(A^{\mathbf{s}}_{i,j})_{i,j\in I}$ and the affine Cartan matrix $\widehat{A}^{\mathbf{s}}=(A^{\mathbf{s}}_{i,j})_{i,j\in \widehat I}$ are given by
\begin{align*}
A_{i,j}^{\mathbf{s}}=(s_{i}+s_{i+1})\delta_{i,j}-s_{i}\delta_{i,j+1}-s_{i+1}\delta_{i+1,j}\quad \text{for $i,j\in \widehat I$.}
\end{align*}
Note that $|\det A^{\mathbf{s}}|=|m-n|$ and $\det \widehat{A}^{\mathbf{s}}=0$. Denote by $\mathfrak{sl}_{\mathbf{s}}$ (resp. $\widehat{\mathfrak{sl}}_{\mathbf{s}}$) the superalgebra corresponding to the Cartan matrix $A^{\mathbf{s}}$ (resp. the affine Cartan matrix $\widehat{A}^{\mathbf{s}}$).

In this paper, we consider the Cartan matrix $A=(A_{i,j})_{i,j\in I}$ associated with the standard parity sequence. Then $A$ is of the following form:
\vspace{0.7cm}
\[
A = \begin{bmatrix}
2 & -1 &   &   &   &   & \\
-1 & \ddots & \ddots &   &   &   &  \\
& \ddots & 2 & -1 &   &   &   \\
&   & -1 & 0 & 1 &   &  \\
&   &   & 1 & -2 & \ddots & \\
&   &   &    & \ddots & \ddots & 1  \\
&   &   &   &   & 1 & -2
\end{bmatrix}.
\begin{tikzpicture}[overlay,remember picture]
\draw[<-,dashed] (-3.4,2.2) to (-3.4,2.7) node[above] {$m$};
\draw[<-,dashed] (0,0.1) to (0.5,0.1) node[right] {$m$};
\end{tikzpicture}
\]

Fix $q\in\mathbb{C}^{*}$ that is not a root of unity. For $k\in\mathbb{Z}$, define $[k]_{q}=\frac{q^{k}-q^{-k}}{q-q^{-1}}$. We also use the notation $[x,y]_{a}=xy-(-1)^{|x||y|}ayx$. For simplicity, we write $[x,y]_{1}=[x,y]$. We will use the identities:
\begin{align}\label{eqs:three terms}
\begin{split}
[[x,y]_{a},z]_{b}=[x,[y,z]_{c}]_{abc^{-1}}+(-1)^{|y||z|}c[[x,z]_{bc^{-1}},y]_{ac^{-1}},\\
[x,[y,z]_{a}]_{b}=[[x,y]_{c},z]_{abc^{-1}}+(-1)^{|x||y|}c[y,[x,z]_{bc^{-1}}]_{ac^{-1}}.
\end{split}
\end{align}
The multiplication on the tensor product is given by $(x\otimes y)(z\otimes w)=(-1)^{|y||z|}xz\otimes yw$. More generally, $(a_{1}\otimes\cdots\otimes a_{l})(b_{1}\otimes\cdots\otimes b_{l})=(-1)^{\sum_{i>j}|a_{i}||b_{j}|}a_{1}b_{1}\otimes\cdots\otimes a_{l}b_{l}$.

Let $\overline{\alpha}_{i}$ (resp. $\overline{\Lambda}_{i}$), $i\in I$, denote the simple roots (resp. fundamental weights) of $\mathfrak{sl}_{m|n}$. Let $\Lambda_{0}$ denote the affine fundamental weight, and $\delta$ be the null root of $\widehat{\mathfrak{sl}}_{m|n}$. Set $\Lambda_{i}=\overline{\Lambda}_{i}+\Lambda_{0}$ and $\alpha_{i}=\overline{\alpha}_{i}$ for $i\in I$, and set $\alpha_{0}=\delta-\sum_{i\in I}\alpha_{i}$. Then $\alpha_{i}$ (resp. $\Lambda_{i}$), $i\in \widehat I$, are the simple roots (resp. fundamental weights) of $\widehat{\mathfrak{sl}}_{m|n}$. Let $(\cdot\mid\cdot)$ be the nondegenerate supersymmetric invariant bilinear form given by $(\overline{\alpha}_{i}\mid\overline{\alpha}_{j})=A_{i,j}$, $(\overline{\alpha}_{i}\mid\overline{\Lambda}_{j})=\delta_{ij}$ for $i,j\in I$, $(\Lambda_{0}\mid\Lambda_{0})=(\delta\mid\delta)=(\delta\mid\overline{\alpha}_{i})=(\delta\mid\overline{\Lambda}_{i})=0$ for $i\in I$, and $(\Lambda_{0}\mid\delta)=1$. For $i\in\widehat I$, the parity of the simple root is
\[
|\alpha_i|=\frac{1-s_i s_{i+1}}{2},
\]
where the cyclic convention above is used.  For every $a\in\mathbb Z$, set $q_a:=q^{s_a}$; thus $q_{a+M}=q_a$ and, in particular, $q_0=q_M$.  Separately, for the vector/matrix index set $J$, define
\[
|a|:=\frac{1-s_a}{2}\in\mathbb Z_2\qquad(a\in J).
\]
For the standard parity sequence, $|a|=\overline 0$ for $1\le a\le m$ and $|a|=\overline 1$ for $m+1\le a\le M$.  Thus $|a|$ always denotes vector parity, whereas $|\alpha_i|$ denotes the parity of a simple root.

\subsection{Quantum affine superalgebra}
For the remainder of this section we specialize to untwisted type A. We work with the level-zero quotient, so the affine central element is specialized to $1$; in the Drinfeld--Jimbo presentation this is the relation $t_0t_1\cdots t_{M-1}=1$.
The following is the level-zero quotient of Yamane's Drinfeld--Jimbo presentation \cite[Proposition~6.7.1]{Yam1999}.
\begin{definition}
Assume $m,n\geq1$ with $m\ne n$, and use the standard parity sequence.
The Drinfeld--Jimbo realization $\mathcal{U}^{\mathrm{DJ}}_{q}(\widehat{\mathfrak{sl}}_{m|n})$ of the quantum affine superalgebra is the superalgebra generated by $e_{i}$, $f_{i}$, $t_{i}^{\pm1}$, $i\in\widehat I$, subject to the following relations:
\begin{gather}
t_{i}t_{j}=t_{j}t_{i},\quad t_{i}^{\pm1}t_{i}^{\mp1}=1,\quad t_0t_1\cdots t_{M-1}=1,\label{re:1}\\
t_{i}e_{j}t_{i}^{-1}=q^{ A_{i,j}}e_{j},\quad t_{i}f_{j}t_{i}^{-1}=q^{ -A_{i,j}}f_{j},\label{re:2}\\
[e_{i},f_{j}]=\frac{\delta_{ij}}{q_{i}-q_{i}^{-1}}(t_{i}-t_{i}^{-1}),\label{re:3}\\
[e_{i},e_{j}]=[f_{i},f_{j}]=0 \quad\text{for $A_{i,j}=0$},\label{re:4}\\
[e_{i},[e_{i},e_{i\pm1}]_{q^{-A_{i,i\pm1}}}]_{q^{A_{i,i\mp1}}}=[f_{i},[f_{i},f_{i\pm1}]_{q^{-A_{i,i\pm1}}}]_{q^{A_{i,i\mp1}}}=0 \quad\text{for $i\notin\{0,m\} $},\label{re:5}\\
[e_{i},[e_{i+1},[e_{i},e_{i-1}]_{q^{s_{i}}}]_{q^{s_{i+1}}}]=[f_{i},[f_{i+1},[f_{i},f_{i-1}]_{q^{s_{i}}}]_{q^{s_{i+1}}}]=0\quad\text{for $i\in\{0,m\}$ and $mn\neq 2$},\label{re:6}
\end{gather}
\begin{gather}
\begin{split}
&[e_{i+1},[e_{i-1},[e_{i+1},[e_{i-1},e_{i}]_{q^{-A_{i-1,i}}}]_{q^{A_{i+1,i+1}}}]]_{q^{-A_{i+1,i-1}}}\\&=[e_{i-1},[e_{i+1},[e_{i-1},[e_{i+1},e_{i}]_{q^{-A_{i+1,i}}}]_{q^{A_{i-1,i-1}}}]]_{q^{-A_{i-1,i+1}}} \quad \text{for $A_{i,i}\neq0$ and $mn=2$},\label{re:7}
\end{split}
\end{gather}
\begin{gather}
\begin{split}
&[f_{i+1},[f_{i-1},[f_{i+1},[f_{i-1},f_{i}]_{q^{-A_{i-1,i}}}]_{q^{A_{i+1,i+1}}}]]_{q^{-A_{i+1,i-1}}}\\&=[f_{i-1},[f_{i+1},[f_{i-1},[f_{i+1},f_{i}]_{q^{-A_{i+1,i}}}]_{q^{A_{i-1,i-1}}}]]_{q^{-A_{i-1,i+1}}} \quad \text{for $A_{i,i}\neq0$ and $mn=2$}\label{re:8}
\end{split}
\end{gather}
The parity of the generators is given by
$|e_{i}|=|f_{i}|=|\alpha_{i}|$ and $|t_{i}|=0$.
\end{definition}

The following is Yamane's new Drinfeld realization \cite[Theorem~8.5.1]{Yam1999} after passing to the level-zero quotient.
\begin{definition}\label{Def:Drinfeld realization for quantum affine superalgebra}
Assume $m,n\geq1$ with $m\ne n$, and use the standard parity sequence.
The quantum affine superalgebra
$\mathcal{U}_{q}(\widehat{\mathfrak{sl}}_{m|n}):=\mathcal{U}_{q}^{\mathrm{Dr}}(\widehat{\mathfrak{sl}}_{m|n})$
is the superalgebra generated by $x_{i,s}^{\pm}$, $h_{i,r}$, and $k_{i}^{\pm1}$, where $i\in I$, $s\in\mathbb{Z}$, and $r\in\mathbb{Z}\setminus\{0\}$, subject to the following relations.  Whenever a relation below contains an index $i\pm1$, it is imposed only for those signs for which $i\pm1\in I$:
\begin{gather}
k_{i}k_{j}=k_{j}k_{i},\quad k_{i}^{\pm1}k_{i}^{\mp1}=1,\label{QTSArelation1}\\
k_{i}x_{j,s}^{\pm}k_{i}^{-1}=q^{\pm A_{i,j}}x_{j,s}^{\pm},\label{QTSArelation2}\\
[k_{i},h_{j,r}]=[h_{i,r_{1}},h_{j,r_{2}}]=0,\label{QTSArelation3}\\
[h_{i,r},x_{j,s}^{\pm}]=\pm\frac{[rs_{i}A_{i,j}]_{q}}{r}x_{j,s+r}^{\pm},\label{QTSArelation4}\\
[x_{i,s}^{+},x_{j,l}^{-}]=\frac{\delta_{ij}}{q_{i}-q_{i}^{-1}}(\phi_{i,s+l}^{+}-\phi_{i,s+l}^{-}),\label{QTSArelation5}\\
[x_{i,s}^{\pm},x_{j,l}^{\pm}]=0 \quad\text{for $A_{i,j}=0$},\label{QTSArelation6}\\
x_{i,s+1}^{\pm}x_{j,l}^{\pm}-q^{\pm A_{i,j}}x_{i,s}^{\pm}x_{j,l+1}^{\pm}=q^{\pm A_{i,j}}x_{j,l}^{\pm}x_{i,s+1}^{\pm}-x_{j,l+1}^{\pm}x_{i,s}^{\pm}\quad\text{for $A_{i,j}\neq0$},\label{QTSArelation7}\\
\mathrm{Sym}_{k_1,k_2}[x_{i,k_{1}}^{\pm},[x_{i,k_{2}}^{\pm},x_{i\pm1,l}^{\pm}]_{q}]_{q^{-1}}=0 \quad\text{for $i\in I\setminus\{m\}$},\label{QTSArelation8}\\
\mathrm{Sym}_{k_1,k_2}[x_{m,k_{1}}^{\pm},[x_{m+1,l_{1}}^{\pm},[x_{m,k_{2}}^{\pm},x_{m-1,l_{2}}^{\pm}]_{q}]_{q^{-1}}]=0\quad\text{if $m,n>1$}.\label{QTSArelation9}
\end{gather}
Here
\begin{gather*}
\phi_{i}^{\pm}(z)=\sum_{r\in\mathbb{Z}_{\geq 0}}\phi_{i,\pm r}^{\pm}z^{\mp r}=k_{i}^{\pm1}\exp(\pm(q_{i}-q_{i}^{-1})\sum_{r=1}^{\infty}h_{i,\pm r}z^{\mp r}),
\end{gather*}
and the parity of the generators is given by
$|x_{i,r}^{\pm}|=|\alpha_{i}|$ and $|h_{i,r}|=|k_{i}|=0$.
\end{definition}

The relation with the Drinfeld--Jimbo presentation is made precise in \cite[Theorem~5.2]{Zh2014}: the homomorphism from the affine Drinfeld--Jimbo presentation is surjective and its kernel is generated by the affine central element minus \(1\).  Since we work throughout with the level-zero quotient, this yields a unique superalgebra isomorphism
\[\mathcal{U}^{\mathrm{DJ}}_{q}(\widehat{\mathfrak{sl}}_{m|n})\xrightarrow{\ \sim\ } \mathcal{U}_{q}(\widehat{\mathfrak{sl}}_{m|n}),
\]
such that
\begin{align}\label{iso:DJ to D}
\begin{split}
& e_{i} \mapsto x_{i,0}^{+},\quad f_{i} \mapsto x_{i,0}^{-},\quad t_{i}^{\pm1}\mapsto k_{i}^{\pm1},\quad \text{for $i\in I$}, \quad t_{0}^{\pm1} \mapsto \phi_{\theta}^{\mp}, \\
& e_{0} \mapsto (-1)^{m+n-1}q^{n-m}[\cdots[x_{1,1}^{-},x_{2,0}^{-}]_{q_{2}},\dots,x_{m+n-1,0}^{-}]_{q_{m+n-1}}\phi_{\theta}^{-},\\
& f_{0} \mapsto \phi_{\theta}^{+}[\cdots[x_{1,-1}^{+},x_{2,0}^{+}]_{q_{2}},\dots,x_{m+n-1,0}^{+}]_{q_{m+n-1}},
\end{split}
\end{align}
where $\phi_{\theta}^{\pm}:=k_{1}^{\pm1}\cdots k_{m+n-1}^{\pm1}$.  The affine Cartan relation $t_{0}\cdots t_{M-1}=1$ is therefore preserved.

\subsection{RTT realization and recursive Gauss-generator formulas}\label{subsec:rtt-gauss}

We recall the RTT presentation of the type $A$ quantum affine general (and special) linear superalgebra and its relation with the Drinfeld presentation. Drinfeld-current realizations for $U_q(\widehat{\mathfrak{gl}}_{m|n})$ were constructed in \cite{FHS1997,Zyz1997}, and the RTT realization together with the Ding--Frenkel homomorphism was further used in \cite{Zh2016}.  As noted in \cite{XLZ2026}, the earlier constructions rely on the super Ding--Frenkel correspondence, while a proof of injectivity of the loop-case RTT--Drinfeld map was not supplied.  A complete proof of the $R$-matrix--Drinfeld isomorphism for $U_q(\widehat{\mathfrak{gl}}_{m|n})$, together with the corresponding special-linear result for $U_q(\widehat{\mathfrak{sl}}_{m|n})$, was recently given in \cite{XLZ2026}.

We continue to use $J=\{1,\dots,M\}$, $M=m+n$, for the matrix indices fixed above. Throughout this subsection we assume $m,n>1$.  For later use, for $i,j,k\in J$ set
\[
\epsilon_{ijk}:=(-1)^{(|i|+|k|)(|k|+|j|)}.
\]
For $i\in I$, set
\begin{equation}\label{formula for tau i}
\tau_i:=q^{\sum_{a=1}^{i}s_a}=
\begin{cases}
q^i,&1\le i\le m,\\
q^{2m-i},&m+1\le i\le m+n-1.
\end{cases}
\end{equation}

Let $V=\oplus_{i\in J}\mathbb{C}v_{i}$ be the vector superspace with $|v_{i}|=|i|$. Define the elementary matrices $\mathsf E_{ij}\in\End(V)$ by $\mathsf E_{ij}v_{k}=\delta_{jk}v_{i}$ for $i,j,k\in J$. They form a basis of the vector superspace $\End(V)$, with $|\mathsf E_{ij}|=|i|+|j|$. The Perk--Schultz $R$-matrix (cf. \cite{PS1981}) is
\begin{align}\label{eq:Perk-Schultz-R}
R(z,w)&=\sum_{i\in J}(zq_{i}-wq_{i}^{-1})\mathsf E_{ii}\otimes \mathsf E_{ii}+(z-w)\sum_{i\neq j}\mathsf E_{ii}\otimes \mathsf E_{jj}\notag\\
&\quad+w\sum_{j<i}(q_{j}-q_{j}^{-1})\mathsf E_{ij}\otimes \mathsf E_{ji}+z\sum_{j<i}(q_{i}-q_{i}^{-1})\mathsf E_{ji}\otimes \mathsf E_{ij}.
\end{align}
The RTT algebra $\mathcal{U}_{q}^{\mathrm{rtt}}(\widehat{\mathfrak{gl}}_{m|n})$ of \cite{FHS1997,Zyz1997,Zh2016,XLZ2026} is the superalgebra generated by $s_{ij}^{(u)}$, $t_{ij}^{(u)}$, $i,j\in J$, $u\in\mathbb{Z}_{\geq0}$, with $|s_{ij}^{(u)}|=|t_{ij}^{(u)}|=|i|+|j|$, subject to the following defining relations:
\begin{align*}
&R_{23}(z,w)T_{12}(z)T_{13}(w)=T_{13}(w)T_{12}(z)R_{23}(z,w),\\&
R_{23}(z,w)S_{12}(z)S_{13}(w)=S_{13}(w)S_{12}(z)R_{23}(z,w),\\&
R_{23}(z,w)T_{12}(z)S_{13}(w)=S_{13}(w)T_{12}(z)R_{23}(z,w),\\&
R_{23}(z,w)S_{12}(z)T_{13}(w)=T_{13}(w)S_{12}(z)R_{23}(z,w),\\&
t_{ij}^{(0)}=s_{ji}^{(0)}=0\quad \text{for $1\leq j<i\leq m+n$},\\&
t_{ii}^{(0)}s_{ii}^{(0)}=s_{ii}^{(0)}t_{ii}^{(0)}=1\quad \text{for $i\in J$}.
\end{align*}
where the matrices $T(z)$, $S(z)$ are given by
\begin{align*}
&T(z)=\sum_{i,j\in J}t_{ij}(z)\otimes \mathsf E_{ij}
\in(\mathcal{U}_{q}^{\mathrm{rtt}}(\widehat{\mathfrak{gl}}_{m|n})\otimes \End V)[[z^{-1}]],
\quad t_{ij}(z)=\sum_{r\in\mathbb{Z}_{\geq0}}t_{ij}^{(r)}z^{-r},\\&
S(z)=\sum_{i,j\in J}s_{ij}(z)\otimes \mathsf E_{ij}
\in(\mathcal{U}_{q}^{\mathrm{rtt}}(\widehat{\mathfrak{gl}}_{m|n})\otimes \End V)[[z]],
\quad s_{ij}(z)=\sum_{r\in\mathbb{Z}_{\geq0}}s_{ij}^{(r)}z^{r},
\end{align*} 
Write
\[
T(z)^{-1}=\sum_{i,j\in J}t'_{ij}(z)\otimes \mathsf E_{ij},
\qquad
S(z)^{-1}=\sum_{i,j\in J}s'_{ij}(z)\otimes \mathsf E_{ij}.
\]
The quantum Berezinians are defined by
\begin{align*}
qB(T(z))
&=\sum_{\lambda\in S_m}(-q)^{-l(\lambda)}
t_{1,\lambda(1)}(z)t_{2,\lambda(2)}(q^2z)\cdots
t_{m,\lambda(m)}(q^{2m-2}z)\\
&\times
\sum_{\mu\in S_n}(-q)^{-l(\mu)}
t'_{m+1,m+\mu(1)}(q^{2m-2}z)
t'_{m+2,m+\mu(2)}(q^{2m-4}z)\cdots
t'_{m+n,m+\mu(n)}(q^{2m-2n}z),\\
qB(S(z))
&=\sum_{\lambda\in S_m}(-q)^{-l(\lambda)}
s_{1,\lambda(1)}(z)s_{2,\lambda(2)}(q^2z)\cdots
s_{m,\lambda(m)}(q^{2m-2}z)\\
&\times
\sum_{\mu\in S_n}(-q)^{-l(\mu)}
s'_{m+1,m+\mu(1)}(q^{2m-2}z)
s'_{m+2,m+\mu(2)}(q^{2m-4}z)\cdots
s'_{m+n,m+\mu(n)}(q^{2m-2n}z).
\end{align*}
Here \(t'_{ab}(z)\) and \(s'_{ab}(z)\) denote matrix coefficients of the inverse matrices, not inverses of the individual coefficients \(t_{ab}(z)\) and \(s_{ab}(z)\), see \cite{JLZ2025}.

Following the quantum Berezinian reduction used in \cite{JLZ2025,XLZ2026}, we use the special-linear RTT realization
\[
\mathcal{U}_{q}^{\mathrm{rtt}}(\widehat{\mathfrak{sl}}_{m|n})
:=
\mathcal{U}_{q}^{\mathrm{rtt}}(\widehat{\mathfrak{gl}}_{m|n})
\Big/
\big\langle qB(T(z))-1,\ qB(S(z))-1\big\rangle,
\]
where the brackets denote the ideal generated by all coefficients of the two displayed series.  The quantum Berezinian is group-like for the RTT coproduct, so this quotient inherits the coproduct used below.

Here, we use the following convention. Let $A,B,C$ be three superalgebras, and let $T=\sum_{i}a_{i}\otimes b_{i}\in A\otimes B$. For $1\leq k<l\leq 3$, define 
\begin{align*}
& T_{12}:=\sum_{i}a_{i}\otimes b_{i}\otimes 1\in A\otimes B\otimes C,\quad T_{13}:=\sum_{i}a_{i}\otimes 1\otimes b_{i}\in A\otimes C\otimes B,\\ & T_{23}:=\sum_{i}1\otimes a_{i}\otimes b_{i}\in C\otimes A\otimes B.
\end{align*}
More generally, for $1\leq k<l\leq r$, the elements $T_{kl}$ are defined in the natural way by introducing $r-3$ additional auxiliary superalgebras. We shall not use $T_{kl}$ with $k>l$.

Note that $S(z)$ and $T(z)$ admit the following Gauss decomposition
\begin{align*}
T(z) &= \tilde{F}^{+}(z)\tilde{G}^{+}(z)\tilde{E}^{+}(z) \\
&= \Big(\sum_{j<i} \tilde{f}_{ij}^{+}(z)\otimes \mathsf E_{ij} + 1\Big)\Big(\sum_i \tilde{g}_i^{+}(z)\otimes \mathsf E_{ii}\Big)\Big(\sum_{j<i} \tilde{e}_{ji}^{+}(z)\otimes \mathsf E_{ji} + 1\Big),\\[2mm]
S(z) &= \tilde{F}^{-}(z)\tilde{G}^{-}(z)\tilde{E}^{-}(z) \\
&= \Big(\sum_{j<i} \tilde{f}_{ij}^{-}(z)\otimes \mathsf E_{ij} + 1\Big)\Big(\sum_i \tilde{g}_i^{-}(z)\otimes \mathsf E_{ii}\Big)\Big(\sum_{j<i} \tilde{e}_{ji}^{-}(z)\otimes \mathsf E_{ji} + 1\Big).
\end{align*}
Then, with all products taken in the graded tensor product,
\begin{gather*}
t_{ii}(z)=\tilde{g}_{i}^{+}(z)+\sum_{j<i}\epsilon_{iij}\tilde{f}_{ij}^{+}(z)\tilde{g}_{j}^{+}(z)\tilde{e}_{ji}^{+}(z),\\ 
t_{ij}(z)=\tilde{f}_{ij}^{+}(z)\tilde{g}_{j}^{+}(z)+\sum_{j'<j}\epsilon_{ijj'}\tilde{f}_{ij'}^{+}(z)\tilde{g}_{j'}^{+}(z)\tilde{e}_{j'j}^{+}(z),\quad\text{for $j<i$},\\
t_{ji}(z)=\tilde{g}_{j}^{+}(z)\tilde{e}_{ji}^{+}(z)+\sum_{j'<j}\epsilon_{jij'}\tilde{f}_{jj'}^{+}(z)\tilde{g}_{j'}^{+}(z)\tilde{e}_{j'i}^{+}(z),\quad\text{for $j<i$},\\
s_{ii}(z)=\tilde{g}_{i}^{-}(z)+\sum_{j<i}\epsilon_{iij}\tilde{f}_{ij}^{-}(z)\tilde{g}_{j}^{-}(z)\tilde{e}_{ji}^{-}(z),\\
s_{ij}(z)=\tilde{f}_{ij}^{-}(z)\tilde{g}_{j}^{-}(z)+\sum_{j'<j}\epsilon_{ijj'}\tilde{f}_{ij'}^{-}(z)\tilde{g}_{j'}^{-}(z)\tilde{e}_{j'j}^{-}(z),\quad\text{for $j<i$},\\
s_{ji}(z)=\tilde{g}_{j}^{-}(z)\tilde{e}_{ji}^{-}(z)+\sum_{j'<j}\epsilon_{jij'}\tilde{f}_{jj'}^{-}(z)\tilde{g}_{j'}^{-}(z)\tilde{e}_{j'i}^{-}(z)\quad\text{for $j<i$}.
\end{gather*}
For example, the contribution of the $j$-th diagonal Gauss factor to the $(i,i)$ entry is
\[
(\tilde f_{ij}^{\pm}\otimes\mathsf E_{ij})
(\tilde g_j^{\pm}\otimes\mathsf E_{jj})
(\tilde e_{ji}^{\pm}\otimes\mathsf E_{ji})
=
\epsilon_{iij}\tilde f_{ij}^{\pm}\tilde g_j^{\pm}\tilde e_{ji}^{\pm}
\otimes\mathsf E_{ii},
\]
which explains both the index $j$ on the diagonal factor and the sign $\epsilon_{iij}$ in the first and fourth formulas.

With the sign $\epsilon_{ijk}$ fixed above, the RTT algebra
$\mathcal{U}_{q}^{\mathrm{rtt}}(\widehat{\mathfrak{gl}}_{m|n})$ has the coproduct
\begin{align*}
\Delta^{\mathrm{rtt}}(s_{ij}^{(u)})=\sum_{a=0}^{u}\sum_{k\in J}\epsilon_{ijk}s_{ik}^{(a)}\otimes s_{kj}^{(u-a)},\quad \Delta^{\mathrm{rtt}}(t_{ij}^{(u)})=\sum_{a=0}^{u}\sum_{k\in J}\epsilon_{ijk}t_{ik}^{(a)}\otimes t_{kj}^{(u-a)}.
\end{align*}
The quantum Berezinian is group-like for this coproduct; hence the above coproduct descends to the special-linear quotient.
Let $X_{i}^{+}(z)=\tilde{e}_{i,i+1}^{+}(z)-\tilde{e}_{i,i+1}^{-}(z)=\sum_{r}X_{i,r}^{+}z^{-r}$, $X_{i}^{-}(z)=\tilde{f}_{i+1,i}^{+}(z)-\tilde{f}_{i+1,i}^{-}(z)=\sum_{r}X_{i,r}^{-}z^{-r}$, $K_{l}^{\pm}(z)=\tilde{g}_{l}^{\pm}(z)$ for $i\in I$ and $l\in J$. 
\begin{theorem}[{\cite{DF1993,Zh2016}; \cite[Theorem~1 and Section~4]{Zyz1997}}]\label{th: RTT and Drinfeld realization}
For type A with the standard parity sequence, assume $m,n>1$ and $m\ne n$.
As a superalgebra $\mathcal{U}_{q}^{\mathrm{rtt}}(\widehat{\mathfrak{gl}}_{m|n})$ is generated by the coefficients of $X_{i}^{\pm}(z)$, $K_{j}^{\pm}(z)$, $i\in I$, $j\in J$, subject to the relations
\begin{align*}
&K_{i}^{\epsilon}(z)K_{j}^{\epsilon'}(w)=K_{j}^{\epsilon'}(w)K_{i}^{\epsilon}(z)\quad\text{for $\epsilon,\epsilon'\in\{+,-\}$},\\
&K_{i}^{\epsilon}(z)X_{j}^{\epsilon'}(w)=X_{j}^{\epsilon'}(w)K_{i}^{\epsilon}(z)
 \quad\text{for } i\notin\{j,j+1\},\\
&K_{i}^{\epsilon}(z)X_{i}^{\pm}(w)=\left(\frac{q_{i}z-q_{i}^{-1}w}{z-w}\right)^{\mp1}X_{i}^{\pm}(w)K_{i}^{\epsilon}(z),\\
&K_{i+1}^{\epsilon}(z)X_{i}^{\pm}(w)=\left(\frac{q_{i+1}^{-1}z-q_{i+1}w}{z-w}\right)^{\mp1}X_{i}^{\pm}(w)K_{i+1}^{\epsilon}(z),\\
&X_{i}^{\epsilon}(z)X_{j}^{\epsilon}(w)-X_{j}^{\epsilon}(w)X_{i}^{\epsilon}(z)=0\quad\text{if } |i-j|\ge 2,\\
&X_{m}^{\epsilon}(z)X_{m}^{\epsilon}(w)+X_{m}^{\epsilon}(w)X_{m}^{\epsilon}(z)=0,\\
&(q_{i}^{\mp1}z-q_{i}^{\pm1}w)X_{i}^{\pm}(z)X_{i}^{\pm}(w)=(q_{i}^{\pm1}z-q_{i}^{\mp1}w)X_{i}^{\pm}(w)X_{i}^{\pm}(z)\quad\text{if } i\neq m,\\
&(q_{i+1}z-q_{i+1}^{-1}w)X_{i}^{+}(z)X_{i+1}^{+}(w)=(z-w)X_{i+1}^{+}(w)X_{i}^{+}(z),\\
&(z-w)X_{i}^{-}(z)X_{i+1}^{-}(w)=(q_{i+1}z-q_{i+1}^{-1}w)X_{i+1}^{-}(w)X_{i}^{-}(z),\\
&[X_{i}^{+}(z),X_{j}^{-}(w)]=\delta_{ij}(q_{i}-q_{i}^{-1})\delta\!\left(\frac{z}{w}\right)\Bigl(K_{i+1}^{+}(z)K_{i}^{+}(z)^{-1}-K_{i+1}^{-}(w)K_{i}^{-}(w)^{-1}\Bigr),\\
&\mathrm{Sym}_{z_{1},z_{2}}[X_{i}^{\epsilon}(z_{1}),[X_{i}^{\epsilon}(z_{2}),X_{i\pm1}^{\epsilon}(w)]_{q}]_{q^{-1}}=0\quad\text{if } i\neq m,\\
&\mathrm{Sym}_{z_{1},z_{2}}[X_{m}^{\epsilon}(z_{1}),[X_{m+1}^{\epsilon}(y),[X_{m}^{\epsilon}(z_{2}),X_{m-1}^{\epsilon}(w)]_{q}]_{q^{-1}}]=0.
\end{align*}
\end{theorem}

Recall the scalar $\tau_i$ from \eqref{formula for tau i}.  By Theorem~\ref{th: RTT and Drinfeld realization}, together with the special-linear RTT--Drinfeld isomorphism of \cite[Proposition~4.8]{XLZ2026} (see also \cite[Section~1.1]{Zh2018}), we obtain the following corollary.
\begin{corollary}\label{coro: iso from Drinfeld realization to RTT realization}
For type A with the standard parity sequence, assume $m,n>1$ and $m\ne n$.
There exists a unique superalgebra isomorphism
\begin{align*}
\gamma: \mathcal{U}_{q}(\widehat{\mathfrak{sl}}_{m|n})\rightarrow \mathcal{U}_{q}^{\mathrm{rtt}}(\widehat{\mathfrak{sl}}_{m|n})
\end{align*}
such that
\begin{align*}
&x_{i,\geq}^{+}(z)\mapsto  \frac{\tilde{e}_{i,i+1}^{+}(z\tau_{i})}{q_{i}-q_{i}^{-1}},\quad  x_{i,<}^{+}(z)\mapsto \frac{\tilde{e}_{i,i+1}^{-}(z\tau_{i})}{q_{i}-q_{i}^{-1}},\\&
x_{i,>}^{-}(z)\mapsto \frac{\tilde{f}_{i+1,i}^{+}(z\tau_{i})}{q_{i}-q_{i}^{-1}},\quad x_{i,\leq}^{-}(z)\mapsto \frac{\tilde{f}_{i+1,i}^{-}(z\tau_{i})}{q_{i}-q_{i}^{-1}},\quad \phi_{i}^{\pm}(z)\mapsto \tilde{g}_{i+1}^{\pm}(\tau_{i}z)(\tilde{g}_{i}^{\pm}(\tau_{i}z))^{-1},
\end{align*}
where
\begin{align*}
&x_{i,\geq}^{+}(z)=\sum_{r\geq0}x_{i,r}^{+}z^{-r},\quad x_{i,<}^{+}(z)=-\sum_{r<0}x_{i,r}^{+}z^{-r},\\&
x_{i,>}^{-}(z)=\sum_{r>0}x_{i,r}^{-}z^{-r},\quad x_{i,\leq}^{-}(z)=-\sum_{r\leq0}x_{i,r}^{-}z^{-r},\quad
\phi_{i}^{\pm}(z)=\sum_{s\in\mathbb{Z}_{\geq0}}\phi_{i,\pm s}^{\pm}z^{\mp s}.
\end{align*}
Moreover, this isomorphism intertwines the coproduct $\Delta$ with $\Delta^{\mathrm{rtt}}$.
\end{corollary}

\begin{proof}
The special-linear RTT--Drinfeld isomorphism follows from Theorem~\ref{th: RTT and Drinfeld realization} and the quantum Berezinian reduction; see also \cite[Proposition~4.8]{XLZ2026}.  The displayed formulas agree with our normalization of the Gauss currents.  Since the quantum Berezinian ideal is a Hopf ideal, the identification also intertwines the induced RTT coproduct with $\Delta$.
\end{proof}

Recursive formulas for Gauss coefficients in the classical specialization modulo $q-1$ appear in \cite[Lemma~3.13]{XLZ2026}.  The particular quantum coefficient-level recursions for the four Gauss modes $\tilde e_{ji}^{(0)}$, $\tilde f_{ij}^{(0)}$, $\tilde e_{ji}^{(-1)}$, and $\tilde f_{ij}^{(1)}$ needed for our coproduct calculations are established directly from the RTT relations in Proposition~\ref{prop: formulas for eji0,eji-1,fij0,fij1}.

\section{Shifted quantum affine superalgebra}\label{section: Shifted quantum affine superalgebra}
In this section, we formulate the two-sided shifted current presentation of quantum affine superalgebras used in this paper.  We then specialize to standard-parity type A and prove the triangular decomposition and PBWD basis.

Let $U_q(\widehat{\mathfrak g})$ be a quantum affine superalgebra with a fixed Drinfeld current presentation relative to a simple system $\Pi=\{\alpha_i\}_{i\in I_{\Pi}}$, denoted by $U_q^{\mathrm{Dr}}(\widehat{\mathfrak g},\Pi)$.  Write the root currents as $x_i^{\pm}(z)$ and the Cartan--Drinfeld currents as $\phi_i^{\pm}(z)$. 

Shifted quantum loop and toroidal superalgebras were defined by Negu\cb{t} and Tsymbaliuk \cite[Remarks~2.6 and~2.7]{NT2026} using shifted Drinfeld doubles. We now record the two-sided shifted current presentation that will be used throughout this paper.

\begin{definition}\label{def:general-shifted-qas}
Fix integral shift data
\[
\mathbf b^{+}=(b_i^{+})_{i\in I_{\Pi}},\qquad
\mathbf b^{-}=(b_i^{-})_{i\in I_{\Pi}}\in\mathbb Z^{I_{\Pi}}.
\]
We denote by
$U^{\mathrm{Dr}}_{q,\mathbf b^{+},\mathbf b^{-}}
(\widehat{\mathfrak g},\Pi)$
the shifted quantum affine superalgebra in the two-sided shifted current
presentation used below.  This presentation is obtained  from the chosen current presentation by replacing the ordinary Cartan expansions by
\[
\phi_i^{\pm}(z)=\sum_{r\ge -b_i^{\pm}}\phi_{i,\pm r}^{\pm}z^{\mp r},
\qquad
\phi_{i,\mp b_i^{\pm}}^{\pm}\ \text{invertible}\qquad(i\in I_{\Pi}),
\]
declaring coefficients outside these prescribed supports to be zero, and interpreting all current relations coefficientwise with these shifted expansions.  The Cartan--Cartan, Cartan--root, mixed, same-sign current, and Serre relations are otherwise the same as in the chosen Drinfeld presentation.  The two leading Cartan coefficients
$\phi_{i,-b_i^{+}}^{+}$ and $\phi_{i,b_i^{-}}^{-}$
are taken to be independent invertible generators.  In particular, even in the zero-shift case we do not impose
$\phi_{i,0}^{+}\phi_{i,0}^{-}=1$;
the usual zero-shift algebra is recovered only after imposing this relation separately. Any affine central charge, derivation, or other generators present in the chosen current presentation are left unchanged, together with their defining relations.
\end{definition}

We next compare the two-sided shifted current presentation above with the shifted Drinfeld super-double construction of Negu\cb{t} and Tsymbaliuk \cite{NT2026}. Their construction starts from positive and negative current halves associated with quiver data, together with the corresponding Cartan part and Hopf pairing, and forms a Drinfeld super-double; the super version is obtained by incorporating the parity data into this double construction.  The shift leaves the positive and negative current halves unchanged and is encoded in the mixed commutator between the positive and negative root currents.

\begin{proposition}\label{prop:comparison-NT-general}
Suppose that, after adjoining inverses to the leading Cartan generators, the zero-shift algebra
$U^{\mathrm{Dr}}_{q,\mathbf 0,\mathbf 0}(\widehat{\mathfrak g},\Pi)$
of Definition~\ref{def:general-shifted-qas} is identified with a Drinfeld super-double of the form considered in
\cite[Remark~2.6]{NT2026}, with compatible root and Cartan currents and mixed-commutator constants.
Assume also that the colorwise loop regradings preserve the positive and negative current relations, as in the quiver-with-parameters framework of \cite{NT2026}.

For integral shift data $\mathbf b^\pm=(b_i^\pm)_{i\in I_\Pi}$, set $r_i=b_i^++b_i^-$.  After adjoining inverses to the leading Cartan generators, the shifted Drinfeld super-double of \cite[Remark~2.7]{NT2026} with shift $\mathbf r=(r_i)_{i\in I_\Pi}$ is isomorphic to $U^{\mathrm{Dr}}_{q,\mathbf b^+,\mathbf b^-}(\widehat{\mathfrak g},\Pi)$.  More precisely,
\[
e_{i,k}\mapsto x_{i,k-b_i^+}^+,\qquad f_{i,k}\mapsto x_{i,k}^-,\qquad \varphi_{i,d}^+\mapsto\phi_{i,d-b_i^+}^+,\qquad \varphi_{i,d}^-\mapsto\phi_{i,b_i^--d}^- .
\]
Equivalently,
\[
e_i(z)\mapsto z^{-b_i^+}x_i^+(z),\qquad f_i(z)\mapsto x_i^-(z),\qquad \varphi_i^+(z)\mapsto z^{-b_i^+}\phi_i^+(z),\qquad \varphi_i^-(z)\mapsto z^{b_i^-}\phi_i^-(z).
\]
\end{proposition}

\begin{proof}
Write $\varphi_i^+(z)=\sum_{d\ge0}\varphi_{i,d}^+z^{-d}$ and $\varphi_i^-(z)=\sum_{d\ge0}\varphi_{i,d}^-z^d$.  Definition~\ref{def:general-shifted-qas} gives
$\phi_i^+(z)=z^{b_i^+}\sum_{d\ge0}\phi_{i,d-b_i^+}^+z^{-d}$ and
$\phi_i^-(z)=z^{-b_i^-}\sum_{d\ge0}\phi_{i,b_i^--d}^-z^d$.
Thus the displayed assignment identifies the Cartan supports and sends $\varphi_{i,0}^\pm$ to the corresponding leading Cartan coefficients.

The positive and negative halves are unchanged by the shift.  The map on the positive half is the colorwise loop regrading $e_{i,k}\mapsto x_{i,k-b_i^+}^+$, while the negative half is unchanged.  Hence the same-sign current and Serre relations are preserved by the hypothesis on the loop regradings.  The Cartan--Cartan and Cartan--root relations are preserved as well, since the extra powers of the spectral parameters occur identically on both sides.

It remains to check the mixed relation.  In the notation of \cite[Remark~2.7]{NT2026},
\[
[e_{i,k},f_{j,\ell}]=\delta_{ij}\kappa_i\bigl(\varphi_{i,k+\ell}^+\mathbf1_{k+\ell\ge0}-\varphi_{i,-k-\ell+r_i}^-\mathbf1_{k+\ell\le r_i}\bigr),
\]
where $\kappa_i$ is the constant determined by the chosen unshifted normalization.  Put $n=k+\ell$.  Since $r_i=b_i^++b_i^-$, the two Cartan terms map respectively to $\phi_{i,n-b_i^+}^+$ and $\phi_{i,n-b_i^+}^-$.  Moreover, $n\ge0$ is equivalent to $n-b_i^+\ge-b_i^+$, while $n\le r_i$ is equivalent to $n-b_i^+\le b_i^-$.  These are exactly the two shifted Cartan support conditions in Definition~\ref{def:general-shifted-qas}.  Hence the mixed relation becomes
\[
[x_{i,k-b_i^+}^+,x_{j,\ell}^-]=\delta_{ij}\kappa_i\bigl(\phi_{i,n-b_i^+}^+-\phi_{i,n-b_i^+}^-\bigr),
\]
with coefficients outside the prescribed supports understood to be zero, which is precisely the shifted mixed relation.

Thus the displayed assignment defines a homomorphism after adjoining inverses to $\varphi_{i,0}^\pm$.  Its inverse is given by
\[
x_{i,k}^+\mapsto e_{i,k+b_i^+},\qquad x_{i,k}^-\mapsto f_{i,k},\qquad \phi_{i,s}^+\mapsto\varphi_{i,s+b_i^+}^+,\qquad \phi_{i,s}^-\mapsto\varphi_{i,b_i^--s}^-,
\]
so the homomorphism is an isomorphism.
\end{proof}

From this point on, we study the shifted quantum affine superalgebras $\mathcal U_q^{\mu_{+},\mu_{-}}(\widehat{\mathfrak{sl}}_{m|n})$ of type A with the standard parity sequence. We identify $I_{\Pi}$ with the finite-node set $I=\{1,\dots,M-1\}$ fixed in Section~2. Let $\Lambda=\bigoplus_{i\in I}\mathbb{Z}\Lambda_i$ be the weight lattice, and let $\Lambda^{+}$ (resp. $\Lambda^{-}$) denote the set of dominant (resp. antidominant) weights $\beta\in\Lambda$ such that $(\alpha_i\mid\beta)\geq0$ (resp. $(\alpha_i\mid\beta)\leq0$) for all $i\in I$. 

\begin{definition}\label{Def: shifted quantum affine superalgebra}
Assume $m,n\geq1$ with $m\ne n$, use the standard parity sequence, and let $\mu_{+},\mu_{-}\in\Lambda$ with $l_i^{\pm}=(\alpha_i\mid\mu_{\pm})$.
The shifted quantum affine superalgebra $\mathcal{U}^{\mu_{+},\mu_{-}}_{q}(\widehat{\mathfrak{sl}}_{m|n})$ is the superalgebra generated by $x_{i,k}^{\pm}$, $\phi_{i,\pm m_{i}^{\pm}}^{\pm}$, and $(\phi_{i,\mp l_{i}^{\pm}}^{\pm})^{-1}$, $i\in I$, $k\in \mathbb{Z}$, $m_{i}^{\pm}\in\mathbb{Z}_{\geq -l_{i}^{\pm}}$, subject to the following relations.  As above, a relation containing $i\pm1$ is imposed only when the corresponding neighboring index belongs to $I$:
\begin{gather}
[\phi_{i}^{\pm}(z),\phi_{j}^{\pm}(w)]=[\phi_{i}^{\pm}(z),\phi_{j}^{\mp}(w)]=0,\quad \phi_{i,\mp l_{i}^{\pm}}^{\pm}(\phi_{i,\mp l_{i}^{\pm}}^{\pm})^{-1}=(\phi_{i,\mp l_{i}^{\pm}}^{\pm})^{-1}\phi_{i,\mp l_{i}^{\pm}}^{\pm}=1,\label{SQASrelation1}\\
(z-q^{\pm A_{i,j}}w)x_{i}^{\pm}(z)x_{j}^{\pm}(w)=(q^{\pm A_{i,j}}z-w)x_{j}^{\pm}(w)x_{i}^{\pm}(z)\quad\text{for $A_{i,j}\neq0$},\label{SQASrelation2}\\
[x_{i}^{\pm}(z),x_{j}^{\pm}(w)]=0\quad \text{for $A_{i,j}=0$},\label{SQASrelation3}\\
(z-q^{\pm A_{i,j}}w)\phi_{i}^{\epsilon}(z)x_{j}^{\pm}(w)=(q^{\pm A_{i,j}}z-w)x_{j}^{\pm}(w)\phi_{i}^{\epsilon}(z) \quad \text{for $\epsilon\in\{+,-\}$},\label{SQASrelation4}\\
[x_{i}^{+}(z),x_{j}^{-}(w)]=\frac{\delta_{ij}}{q_{i}-q_{i}^{-1}}\delta(\frac{z}{w})(\phi_{i}^{+}(z)-\phi_{i}^{-}(z)),\label{SQASrelation5}\\
\mathrm{Sym}_{z_1,z_2}[x_{i}^{\pm}(z_1),[x_{i}^{\pm}(z_2),x_{i\pm1}^{\pm}(w)]_{q}]_{q^{-1}}=0\quad\text{for $i\in I\setminus\{m\}$},\label{SQASrelation6}\\
\mathrm{Sym}_{z_1,z_2}[x_{m}^{\pm}(z_1),[x_{m+1}^{\pm}(y),[x_{m}^{\pm}(z_2),x_{m-1}^{\pm}(w)]_{q}]_{q^{-1}}]=0\quad\text{if $m,n>1$}.\label{SQASrelation7}
\end{gather}
where
\begin{gather*}
x_{i}^{\pm}(z)=\sum_{r\in\mathbb{Z}}x_{i,r}^{\pm}z^{-r},\quad \phi_{i}^{\pm}(z)=\sum_{r\in\mathbb{Z}_{\geq -l_{i}^{\pm}}}\phi_{i,\pm r}^{\pm}z^{\mp r},\quad \delta(z)=\sum_{r\in\mathbb{Z}}z^{r},
\end{gather*}
and the parity of the generators is given by $|x_{i,k}^{\pm}|=|\alpha_{i}|$ and $|\phi_{i,\pm m_{i}^{\pm}}^{\pm}|=0$.
\end{definition}

When $m=1$ or $n=1$, relation~\eqref{SQASrelation7} is omitted because one of the two neighboring simple roots of the isotropic node is absent.  This agrees with the corresponding rank-boundary convention in \cite[Definition~3.1]{Zh2014}.

Let us introduce another set of Cartan generators $h_{i,\pm r}$, $i\in I$, $r\in \mathbb{Z}_{>0}$, in place of $\phi_{i,\pm m_{i}^{\pm}}^{\pm}$, $i\in I$, $m_{i}^{\pm}\in\mathbb{Z}_{>-l_{i}^{\pm}}$, by
\begin{gather*}
\phi_{i}^{\pm}(z)=z^{\pm l_{i}^{\pm}}\phi^{\pm}_{i,\mp l_{i}^{\pm}}\exp(\pm(q_{i}-q_{i}^{-1})\sum\limits_{r'>0}h_{i,\pm r'}z^{\mp r'}).
\end{gather*}
Then relations \eqref{SQASrelation1}--\eqref{SQASrelation7} are equivalent to the following:
\begin{gather}
[\phi_{i,r}^{\pm},\phi_{j,r'}^{\pm}]=[\phi_{i,r}^{\pm},\phi_{j,r'}^{\mp}]=0,\quad \phi_{i,\mp l_{i}^{\pm}}^{\pm}(\phi_{i,\mp l_{i}^{\pm}}^{\pm})^{-1}=(\phi_{i,\mp l_{i}^{\pm}}^{\pm})^{-1}\phi_{i,\mp l_{i}^{\pm}}^{\pm}=1,\label{SQASrelation11}\\
x_{i,r+1}^{\pm}x_{j,s}^{\pm}-q^{\pm A_{i,j}}x_{i,r}^{\pm}x_{j,s+1}^{\pm}=q^{\pm A_{i,j}}x_{j,s}^{\pm}x_{i,r+1}^{\pm}-x_{j,s+1}^{\pm}x_{i,r}^{\pm}\quad\text{for $A_{i,j}\neq0$},\label{SQASrelation21}\\
[x_{i,r}^{\pm},x_{j,s}^{\pm}]=0\quad \text{for $A_{i,j}=0$},\label{SQASrelation31}\\
\phi_{i,\mp l_{i}^{\pm}}^{\pm}x_{j,s}^{+}=q^{\pm A_{i,j}}x_{j,s}^{+}\phi_{i,\mp l_{i}^{\pm}}^{\pm},\quad [h_{i,r},x_{j,s}^{+}]=\frac{[rs_{i}A_{i,j}]_{q_{i}}}{r}x_{j,s+r}^{+}\quad\text{for $r\neq0$},\label{SQASrelation41}\\
\phi_{i,\mp l_{i}^{\pm}}^{\pm}x_{j,s}^{-}=q^{\mp A_{i,j}}x_{j,s}^{-}\phi_{i,\mp l_{i}^{\pm}}^{\pm},\quad [h_{i,r},x_{j,s}^{-}]=-\frac{[rs_{i}A_{i,j}]_{q_{i}}}{r}x_{j,s+r}^{-}\quad\text{for $r\neq0$}.\label{SQASrelation42}\\
[x_{i,r}^{+},x_{j,s}^{-}]=\frac{\delta_{ij}}{q_{i}-q_{i}^{-1}}(\phi_{i,r+s}^{+}-\phi_{i,r+s}^{-}),\label{SQASrelation51}\\
\mathrm{Sym}_{k_1,k_2}[x_{i,k_1}^{\pm},[x_{i,k_2}^{\pm},x_{i\pm1,r}^{\pm}]_{q}]_{q^{-1}}=0\quad\text{for $i\in I\setminus\{m\}$},\label{SQASrelation61}\\
\mathrm{Sym}_{k_1,k_2}[x_{m,k_1}^{\pm},[x_{m+1,r_1}^{\pm},[x_{m,k_2}^{\pm},x_{m-1,r_2}^{\pm}]_{q}]_{q^{-1}}]=0\quad\text{if $m,n>1$}.\label{SQASrelation71}
\end{gather}

Let $\mathcal{U}^{\mu_{+},\mu_{-}}_{q}(\widehat{\mathfrak{sl}}_{m|n})^{>}$, $\mathcal{U}^{\mu_{+},\mu_{-}}_{q}(\widehat{\mathfrak{sl}}_{m|n})^{<}$, and $\mathcal{U}^{\mu_{+},\mu_{-}}_{q}(\widehat{\mathfrak{sl}}_{m|n})^{0}$ be the subalgebra generated by $\{x_{i,k}^{+}\}_{i\in I, k\in\mathbb{Z}}$, $\{x_{i,k}^{-}\}_{i\in I, k\in\mathbb{Z}}$, and $\{(\phi_{i,\mp l_{i}^{\pm}}^{\pm})^{-1}, \phi_{i,\pm m_{i}^{\pm}}^{\pm}\}_{i\in I, m_{i}^{\pm}\in\mathbb{Z}_{\geq -l_{i}^{\pm}}}$, respectively.

The two shift parameters can be redistributed by reindexing the currents.
\begin{remark}
For type A with the standard parity sequence, let $m,n\geq1$ with $m\ne n$ and let $\mu_{+},\mu_{-}\in\Lambda$.
Reindexing the homogeneous current relations gives an isomorphism
\[
\mathcal{U}^{\mu_{+},\mu_{-}}_{q}(\widehat{\mathfrak{sl}}_{m|n})
\xrightarrow{\ \sim\ }
\mathcal{U}^{0,\mu_{+}+\mu_{-}}_{q}(\widehat{\mathfrak{sl}}_{m|n}),
\]
defined by
\[
x_{i}^{-}(z)\mapsto x_{i}^{-}(z),\qquad
x_{i}^{+}(z)\mapsto z^{l_{i}^{+}}x_{i}^{+}(z),\qquad
\phi_{i}^{\pm}(z)\mapsto z^{l_{i}^{+}}\phi_{i}^{\pm}(z).
\]
Thus, as an abstract current algebra, the isomorphism class depends only on the total shift $\mu=\mu_{+}+\mu_{-}$; we write $\mathcal U_q^{\mu}:=\mathcal U_q^{0,\mu}$.  We nevertheless retain the two-sided notation below because the positive and negative Cartan valuations are separately visible in the shifted RTT realization and in the corresponding coproducts.
\end{remark}

We next record the central reduction that recovers the ordinary zero-shift algebra.
\begin{proposition}\label{prop:central-leading-reduction}
For type A with the standard parity sequence, let $m,n\geq1$ with $m\ne n$ and let $\mu_{+},\mu_{-}\in\Lambda$.
For every $i\in I$ the element
$c_i:=\phi_{i,-l_i^{+}}^{+}\phi_{i,l_i^{-}}^{-}$
and its inverse are central in $\mathcal U_q^{\mu_{+},\mu_{-}}(\widehat{\mathfrak{sl}}_{m|n})$.  For the zero shift, let $J_0$ be the ideal of $\mathcal U_q^{0,0}$ generated by $c_i-1$, $i\in I$.  Then
\[
\mathcal U_q^{0,0}(\widehat{\mathfrak{sl}}_{m|n})/J_0
\simeq
\mathcal U_q(\widehat{\mathfrak{sl}}_{m|n}).
\]
\end{proposition}
\begin{proof}
The leading coefficient $\phi_{i,-l_i^{+}}^{+}$ acts on $x_{j,r}^{\pm}$ by the factor $q^{\pm A_{i,j}}$, whereas $\phi_{i,l_i^{-}}^{-}$ acts by the inverse factor.  Their product therefore commutes with every root current and every Cartan generator.  At zero shift the relations $c_i=1$ identify $\phi_{i,0}^{-}$ with $(\phi_{i,0}^{+})^{-1}$.  Setting $k_i=\phi_{i,0}^{+}$ then turns the defining relations into the Drinfeld presentation in Definition~\ref{Def:Drinfeld realization for quantum affine superalgebra}.  This is the super analogue of the central reduction in \cite[Remark~5.3]{FT2019}.
\end{proof}

The triangular decomposition uses the following stability statement.
\begin{lemma}\label{lem:shifted-relation-stability}
For type A with the standard parity sequence, let $m,n\geq1$ with $m\ne n$ and $\mu_{+},\mu_{-}\in\Lambda$.
Let $\widetilde{\mathcal U}^{\mu_{+},\mu_{-}}$ be the auxiliary current superalgebra obtained from Definition~\ref{Def: shifted quantum affine superalgebra} by imposing the Cartan relations, the Cartan--root relations, and the mixed commutator \eqref{SQASrelation5}, but not yet imposing the same-sign quadratic and Serre relations.  Its positive and negative root-current subalgebras are therefore free superalgebras on the modes $x_{i,r}^{+}$ and $x_{i,r}^{-}$.

Let $I^{\pm}$ be the sums of the subspaces spanned by the coefficients of the same-sign quadratic relations \eqref{SQASrelation2}--\eqref{SQASrelation3}.  After quotienting by the two-sided ideals generated by $I^{+}+I^{-}$, let $J^{\pm}$ be the sums of the subspaces spanned by the coefficients of the ordinary Serre relations \eqref{SQASrelation6}.  After quotienting by those relations as well, let $O^{\pm}$ be the subspaces spanned by the coefficients of the odd Serre relation \eqref{SQASrelation7} when $m,n>1$, and set $O^{\pm}=0$ when $m=1$ or $n=1$.  Then, in the corresponding successive quotients,
\[
[I^{\pm},x_{k,r}^{\mp}]=0,\qquad
[J^{\pm},x_{k,r}^{\mp}]=0,\qquad
[O^{\pm},x_{k,r}^{\mp}]=0
\]
for every $k\in I$ and $r\in\mathbb Z$.  Moreover each of the spaces $I^{\pm}$, $J^{\pm}$, and $O^{\pm}$ is stable under every adjoint operator $\operatorname{ad}(h_{u,s})$, $u\in I$, $s\ne0$.
\end{lemma}
\begin{proof}
This is the shifted version of Lemmas~3.7--3.9 in \cite{Zh2014}.  We spell out why the shift creates no extra boundary terms.  Zhang first proves the three commutator statements for the opposite zero modes $x_{k,0}^{\mp}$.  The computations use only the mixed commutator
\[
[x_{i,r}^{+},x_{j,s}^{-}]
=\frac{\delta_{ij}}{q_i-q_i^{-1}}
(\phi_{i,r+s}^{+}-\phi_{i,r+s}^{-})
\]
and the coefficient identities obtained from the Cartan--root current relation \eqref{SQASrelation4}; no use is made of the unshifted normalization $\phi_{i,0}^{-}=(\phi_{i,0}^{+})^{-1}$.  In the shifted quantum affine superalgebra these two current identities are unchanged.  We adopt the convention that a coefficient $\phi_{i,a}^{\pm}$ outside the prescribed support of $\phi_i^{\pm}(z)$ is zero.  Therefore every coefficient identity occurring in the proofs of \cite[Lemmas~3.7--3.9]{Zh2014} remains valid at the boundary of the shifted support.  The invertibility of the first nonzero Cartan coefficients is not used in these cancellations.  Hence
\[
[I^{\pm},x_{k,0}^{\mp}]=0,
\qquad
[J^{\pm},x_{k,0}^{\mp}]=0,
\qquad
[O^{\pm},x_{k,0}^{\mp}]=0
\]
in the respective successive quotients.  The most delicate identity, present only when $m,n>1$, is the odd-node commutator
$[O^{+},x_{m,r}^{-}]=0$.  For completeness, Appendix~\ref{app:odd-serre-stability}
proves this identity directly at the level of currents for every $r\in\mathbb Z$; in particular,
the proof there does not appeal to the unshifted normalization of the Cartan currents.  When $m=1$ or $n=1$, this third stage is absent.

The stability under $\operatorname{ad}(h_{u,s})$ follows directly from
\eqref{SQASrelation41}--\eqref{SQASrelation42}: applying $\operatorname{ad}(h_{u,s})$ to a coefficient of any one of the quadratic or Serre current relations gives a linear combination of coefficients of the same relation with shifted mode indices.  This again does not depend on the support of the Cartan currents.

It remains only to pass from the opposite zero modes to arbitrary opposite modes.  For a fixed simple root $k$, choose an even Cartan direction $u$ with $A_{u,k}\ne0$: one may take $u=k$ when $k\ne m$, and a neighboring even vertex when $k=m$.  Relation \eqref{SQASrelation41} or \eqref{SQASrelation42} then expresses $x_{k,r\pm1}^{\mp}$ as a nonzero scalar multiple of $[h_{u,\pm1},x_{k,r}^{\mp}]$.  The super Jacobi identity, together with the $\operatorname{ad}(h_{u,\pm1})$-stability just proved, inductively propagates the zero-mode commutator identities to all $r\in\mathbb Z$.
\end{proof}

We can now prove the triangular decomposition.
\begin{theorem}\label{thm:shifted-triangular}
For type A with the standard parity sequence, let $m,n\geq1$ with $m\ne n$.
For arbitrary $\mu_{+},\mu_{-}\in\Lambda$, multiplication induces an isomorphism of $\mathbb C(q)$-vector superspaces
\[
\mathcal{U}^{\mu_{+},\mu_{-}}_{q}(\widehat{\mathfrak{sl}}_{m|n})^{<}
\otimes
\mathcal{U}^{\mu_{+},\mu_{-}}_{q}(\widehat{\mathfrak{sl}}_{m|n})^{0}
\otimes
\mathcal{U}^{\mu_{+},\mu_{-}}_{q}(\widehat{\mathfrak{sl}}_{m|n})^{>}
\xrightarrow{\ \sim\ }
\mathcal{U}^{\mu_{+},\mu_{-}}_{q}(\widehat{\mathfrak{sl}}_{m|n}).
\]
Moreover, the positive and negative factors are independent of the shift: they are presented by the positive and negative parts of \eqref{SQASrelation2}--\eqref{SQASrelation3} and \eqref{SQASrelation6}--\eqref{SQASrelation7}, respectively.  The Cartan factor is the commutative algebra generated by the two invertible leading coefficients $\phi_{i,-l_i^{+}}^{+}$, $\phi_{i,l_i^{-}}^{-}$ and by $h_{i,\pm r}$, $r>0$.
\end{theorem}
\begin{proof}
We follow the successive-quotient argument of \cite[Theorem~3.3]{Zh2014}.  In the auxiliary algebra $\widetilde{\mathcal U}^{\mu_{+},\mu_{-}}$, coefficientwise use of the mixed commutator moves every negative root mode to the left of every positive root mode, while the Cartan--root relations move all Cartan generators between them.  Thus multiplication gives a vector superspace isomorphism
\[
\widetilde U^{-}\otimes\widetilde U^{0}\otimes\widetilde U^{+}
\xrightarrow{\ \sim\ }
\widetilde{\mathcal U}^{\mu_{+},\mu_{-}}.
\]
The shift only changes the allowed Cartan coefficients and hence does not affect this normal-ordering argument.

We now impose the same-sign relations in three stages.  By Lemma~\ref{lem:shifted-relation-stability}, the two-sided ideal generated by $I^{+}+I^{-}$ has the normal-ordered form
\[
\widetilde U^{-}I^{-}\widetilde U^{-}\widetilde U^{0}\widetilde U^{+}
+
\widetilde U^{-}\widetilde U^{0}\widetilde U^{+}I^{+}\widetilde U^{+}.
\]
Consequently the quotient still has triangular decomposition, with its positive and negative factors obtained by quotienting the free halves by the quadratic relations only.  In this quotient the same lemma applies to $J^{\pm}$, and, when $m,n>1$, after imposing the ordinary Serre relations it applies once more to $O^{\pm}$; for $m=1$ or $n=1$ there is no third-stage relation to impose.  Therefore the full same-sign relation ideal has the corresponding normal-ordered form with the complete positive and negative relation ideals on the two sides.  Quotienting yields the triangular decomposition stated in Theorem~\ref{thm:shifted-triangular}. Because at no stage is a relation supported entirely in one factor created from the opposite factor, the positive and negative factors have exactly their stated current presentations, and the Cartan factor remains the indicated commutative Laurent polynomial algebra.
\end{proof}

\begin{corollary}\label{cor:shifted-pbw}
For type A with the standard parity sequence, assume $m,n>1$ and $m\ne n$, and let $\mu_{+},\mu_{-}\in\Lambda$.  Fix PBW--Drinfeld (PBWD) bases $\mathcal B^{>}$ and $\mathcal B^{<}$ of the positive and negative type A quantum loop superalgebras as in \cite[Theorem~5.7 and Remark~5.10]{Tsy2021}, and let $\mathcal B^{0}_{\mu_{+},\mu_{-}}$ be the monomial basis in the commuting generators
\[
(\phi_{i,-l_i^{+}}^{+})^{\pm1},\quad
(\phi_{i,l_i^{-}}^{-})^{\pm1},\quad
h_{i,\pm r}\quad (i\in I,\ r>0).
\]
Then the ordered products $b_{-}b_{0}b_{+}$ with $b_{-}\in\mathcal B^{<}$, $b_{+}\in\mathcal B^{>}$, and $b_0\in\mathcal B^{0}_{\mu_{+},\mu_{-}}$ form a PBWD basis of $\mathcal U_q^{\mu_{+},\mu_{-}}(\widehat{\mathfrak{sl}}_{m|n})$.
\end{corollary}

\begin{remark}
After matching the shift parameters and current conventions, the approach of Negu\cb{t} and Tsymbaliuk \cite[Remarks~2.6 and~2.7]{NT2026} via shifted Drinfeld super-doubles and shuffle realizations provides alternative proofs of the triangular decomposition in Theorem~\ref{thm:shifted-triangular} and the PBWD basis statement in Corollary~\ref{cor:shifted-pbw}.
\end{remark}

Set $\mu=\mu_{+}+\mu_{-}$.  We call $\mathcal{U}^{\mu_{+},\mu_{-}}_{q}(\widehat{\mathfrak{sl}}_{m|n})$ dominant (resp. antidominant) if $\mu\in\Lambda^{+}$ (resp. $\mu\in\Lambda^{-}$). For convenience, we write $\alpha_{i}(\mu):=(\alpha_{i}\mid \mu)$.
Thus $\mu\in\Lambda^{+}$ (respectively, $\mu\in\Lambda^{-}$) precisely when $l_i^{+}+l_i^{-}=\alpha_i(\mu)$ is nonnegative (respectively, nonpositive) for every $i\in I$.

\section{Levendorskii-type finite presentation of \texorpdfstring{$\mathcal{U}^{0,\mu}_{q}(\widehat{\mathfrak{sl}}_{m|n})$ for $\mu\in\Lambda^{-}$}{Uq,0,mu(hat sl m,n)}}\label{section: Levendorskii type presentation of shifted quantum affine superalgebra}
Throughout this section and Appendix~\ref{app:finite-presentation-proof}, assume $m,n>1$ and $m\ne n$.  We obtain a finite presentation by retaining finitely many boundary root-current modes, the leading Cartan coefficients, and $h_{i,\pm1}$, and recovering all remaining modes recursively.  At the isotropic node $m$, where $A_{m,m}=0$, the recursion uses the neighboring even modes $h_{m-1,\pm r}$.

\begin{definition}
For type A with the standard parity sequence, assume $m,n>1$ and $m\ne n$.  Let $\mu_{1},\mu_{2}\in\Lambda^{-}$, set $\mu=\mu_{1}+\mu_{2}$, and write
\[
l_{a,i}:=\alpha_i(\mu_a)\quad(a=1,2),\qquad
l_i:=l_{1,i}+l_{2,i}=\alpha_i(\mu).
\]
Here the subscript $a=1,2$ labels the two antidominant summands of the total shift; it is unrelated to the superscript $\pm$ in the two-sided shift notation $l_i^{\pm}$ of the preceding section.  The superalgebra $\mathcal{E}^{\mu_{1},\mu_{2}}_{0}$ is the associative $\mathbb{C}(q)$-superalgebra with generators $x_{i,r}^{+}$, $x_{i,s}^{-}$, $(\phi_{i,0}^{+})^{\pm1}$, $(\phi_{i,l_{i}}^{-})^{\pm1}$, $h_{i,\pm1}$, $i\in I$, $r\in[l_{2,i}-1,0]$, $s\in[l_{1,i},1]$, subject to the following defining relations:
\begin{gather}
\begin{split}\label{SQASrelation111}
(\phi_{i,0}^{+})^{\pm1}, (\phi_{i,l_{i}}^{-})^{\pm1}, h_{i,\pm1}, i\in I, \text{ commute pairwise},\\
(\phi_{i,0}^{+})^{\pm1}(\phi_{i,0}^{+})^{\mp1}=(\phi_{i,l_{i}}^{-})^{\pm1}(\phi_{i,l_{i}}^{-})^{\mp1}=1, 
\end{split}
\end{gather}
\begin{gather}
x_{i,r+1}^{\pm}x_{j,s}^{\pm}-q^{\pm A_{i,j}}x_{i,r}^{\pm}x_{j,s+1}^{\pm}=q^{\pm A_{i,j}}x_{j,s}^{\pm}x_{i,r+1}^{\pm}-x_{j,s+1}^{\pm}x_{i,r}^{\pm}\quad\text{for $A_{i,j}\neq0$},\label{SQASrelation221}
\end{gather}
\begin{gather}
[x_{i,r}^{\pm},x_{j,s}^{\pm}]=0\quad\text{for $A_{i,j}=0$},\label{SQASrelation331}
\end{gather}
\begin{gather}
\phi_{i,0}^{+}x_{j,r}^{+}=q^{A_{i,j}}x_{j,r}^{+}\phi_{i,0}^{+},\quad \phi_{i,l_{i}}^{-}x_{j,r}^{+}=q^{-A_{i,j}}x_{j,r}^{+}\phi_{i,l_{i}}^{-},\quad [h_{i,\pm1},x_{j,r}^{+}]=[s_{i}A_{i,j}]_{q_{i}}x_{j,r\pm1}^{+},\label{SQASrelation441}
\end{gather}
\begin{gather}
\phi_{i,0}^{+}x_{j,s}^{-}=q^{-A_{i,j}}x_{j,s}^{-}\phi_{i,0}^{+},\quad \phi_{i,l_{i}}^{-}x_{j,s}^{-}=q^{A_{i,j}}x_{j,s}^{-}\phi_{i,l_{i}}^{-},\quad [h_{i,\pm1},x_{j,s}^{-}]=-[s_{i}A_{i,j}]_{q_{i}}x_{j,s\pm1}^{-},\label{SQASrelation442}
\end{gather}
\begin{gather}
[x_{i,r}^{+},x_{j,s}^{-}]=0\text{ if $i\neq j$ and }
[x_{i,r}^{+},x_{i,s}^{-}]=
\begin{cases}
\phi_{i,0}^{+}h_{i,1}, &\text{if $r+s=1$},\\
\phi_{i,l_{i}}^{-}h_{i,-1}, &\text{if $r+s=l_{i}-1$},\\
\frac{\phi_{i,0}^{+}-\delta_{l_{i},0}\phi_{i,l_{i}}^{-}}{q_{i}-q_{i}^{-1}}, &\text{if $r+s=0$},\\
\frac{\delta_{l_{i},0}\phi_{i,0}^{+}-\phi_{i,l_{i}}^{-}}{q_{i}-q_{i}^{-1}}, &\text{if $r+s=l_{i}$},\\
0, &\text{if $l_{i}<r+s<0$},
\end{cases}\label{SQASrelation551}
\end{gather}
\begin{gather}
[x_{i,0}^{\pm},[x_{i,0}^{\pm},x_{i\pm1,0}^{\pm}]_{q}]_{q^{-1}}=0\quad\text{for $i\in I\setminus\{m\}$},\label{SQASrelation661}
\end{gather}
\begin{gather}
[x_{m,0}^{\pm},[x_{m+1,0}^{\pm},[x_{m,0}^{\pm},x_{m-1,0}^{\pm}]_{q}]_{q^{-1}}]=0,\label{SQASrelation771}
\end{gather}
\begin{gather}
[h_{i,1},[x_{i,1}^{-},[h_{i,1},x_{i,0}^{+}]]]=0,\quad [h_{i,-1},[x_{i,l_{2,i}-1}^{+},[h_{i,-1},x_{i,l_{1,i}}^{-}]]]=0\quad\text{for $i\in I\setminus\{m\}$},\label{SQASrelation881}
\end{gather}
\begin{gather}
[h_{m-1,1},[x_{m,1}^{-},[h_{m-1,1},x_{m,0}^{+}]]]=0,\quad [h_{m-1,-1},[x_{m,l_{2,m}-1}^{+},[h_{m-1,-1},x_{m,l_{1,m}}^{-}]]]=0,\label{SQASrelation991}
\end{gather}
for any $i\in I$ and integers $r, s$ for which the above relations are well-defined. The parities of the generators are given by $|x_{i,r}^{+}|=|x_{i,s}^{-}|=|\alpha_{i}|$ and $|(\phi_{i,0}^{+})^{\pm1}|=|(\phi_{i,l_{i}}^{-})^{\pm1}|=|h_{i,\pm1}|=0$.
\end{definition}

Now we define $x_{i,r}^{\pm}$ and $\phi_{i,r}^{\pm}$ recursively as follows. For $i\neq m$, set
\begin{align}
&x_{i,r}^{+}:=[s_{i}A_{i,i}]_{q_{i}}^{-1}
\begin{cases}
[h_{i,1},x_{i,r-1}^{+}]\quad\text{if $r>0$},\\
[h_{i,-1},x_{i,r+1}^{+}]\quad\text{if $r<l_{2,i}-1$},
\end{cases}\label{SQASArelation1}\\&
x_{i,r}^{-}:=-[s_{i}A_{i,i}]_{q_{i}}^{-1}
\begin{cases}
[h_{i,1},x_{i,r-1}^{-}]\quad\text{if $r>1$},\\
[h_{i,-1},x_{i,r+1}^{-}]\quad\text{if $r<l_{1,i}$},
\end{cases}\label{SQASArelation2}
\end{align}
and for $i=m$, set
\begin{align}
&x_{i,r}^{+}:=[s_{i-1}A_{i-1,i}]_{q_{i-1}}^{-1}
\begin{cases}
[h_{i-1,1},x_{i,r-1}^{+}]\quad\text{if $r>0$},\\
[h_{i-1,-1},x_{i,r+1}^{+}]\quad\text{if $r<l_{2,i}-1$},
\end{cases}\label{SQASArelation3}\\&
x_{i,r}^{-}:=-[s_{i-1}A_{i-1,i}]_{q_{i-1}}^{-1}
\begin{cases}
[h_{i-1,1},x_{i,r-1}^{-}]\quad\text{if $r>1$},\\
[h_{i-1,-1},x_{i,r+1}^{-}]\quad\text{if $r<l_{1,i}$}.
\end{cases}\label{SQASArelation4}
\end{align}
For each $i\in I$, set
\begin{align}
\phi_{i,r}^{+}:=(q_{i}-q_{i}^{-1})[x_{i,r-1}^{+},x_{i,1}^{-}]\quad\text{if $r>0$},\label{SQASArelation5}\\
\phi_{i,r}^{-}:=-(q_{i}-q_{i}^{-1})[x_{i,r-l_{1,i}}^{+},x_{i,l_{1,i}}^{-}]\quad\text{if $r<l_{i}$}.\label{SQASArelation6}
\end{align}

\begin{theorem}\label{thm:finite iso}
For type A with the standard parity sequence, assume $m,n>1$ and $m\ne n$.  Let $\mu_{1},\mu_{2}\in\Lambda^{-}$ and set $\mu=\mu_{1}+\mu_{2}$.  There is a unique $\mathbb{C}(q)$-superalgebra isomorphism from $\mathcal{E}^{\mu_{1},\mu_{2}}_{0}$ to $\mathcal{U}^{0,\mu}_{q}(\widehat{\mathfrak{sl}}_{m|n})$ given by
\begin{align*}
x_{i,r}^{\pm}\mapsto x_{i,r}^{\pm},\quad \phi_{i,\pm m_{i}^{\pm}}^{\pm}\mapsto\phi_{i,\pm m_{i}^{\pm}}^{\pm},
\end{align*}
for $i\in I$, $r\in\mathbb{Z}$, $m_{i}^{+}\in\mathbb{Z}_{\geq0}$, and $m_{i}^{-}\in\mathbb{Z}_{\geq -l_{i}}$.
\end{theorem}
\begin{proof}
The proof is a mode-propagation argument.  The finite relations first recover all current modes recursively; one then propagates the mixed commutator, Cartan-current, quadratic-current, and Serre relations by the adjoint actions of $h_{i,\pm1}$ and the derived modes $h_{i,\pm2}$.  The isotropic vertex is handled through the neighboring even Cartan direction specified above.  The complete verification is given in Appendix~\ref{app:finite-presentation-proof}.
\end{proof}

\section{Shifted formal Drinfeld coproducts and shift homomorphisms}\label{section: Coproduct} 

This section develops two constructions used later in the coproduct theory.  The formal Drinfeld current coproduct extends to arbitrary two-sided shifts and is coassociative in a completed tensor product.  For one-sided shifts, we construct the homomorphisms $\iota_{\mu,\nu_1,\nu_2}$ and prove their injectivity; Section~\ref{section: shifted RTT and coproducts} uses them to extend the non-completed coproduct from antidominant to arbitrary shifts.  

\subsection{Shifted formal Drinfeld coproducts} \label{subsec:shifted formula Drinfeld coproducts}
After specializing the affine central charge to level zero, the current coproduct of \cite[Theorem~3]{Zyz1997} gives the formal Drinfeld coproduct of the unshifted algebra
\[
\widetilde\Delta:\mathcal U_q(\widehat{\mathfrak{sl}}_{m|n})
\longrightarrow
\mathcal U_q(\widehat{\mathfrak{sl}}_{m|n})\widehat\otimes
\mathcal U_q(\widehat{\mathfrak{sl}}_{m|n}),
\]
with
\begin{gather}\label{coproduct for Drinfeld realization}
\begin{split}
\widetilde\Delta(x_i^+(z))&=x_i^+(z)\otimes1+\phi_i^-(z)\otimes x_i^+(z),\\
\widetilde\Delta(x_i^-(z))&=1\otimes x_i^-(z)+x_i^-(z)\otimes\phi_i^+(z),\quad \widetilde\Delta(\phi_i^\pm(z))=\phi_i^\pm(z)\otimes\phi_i^\pm(z).
\end{split}
\end{gather}
The same formulas make sense for shifted Cartan supports after passing to a coefficientwise completion.  We use the loop grading
\[
\deg x_{i,r}^{\pm}=r,\qquad \deg\phi_{i,r}^{\pm}=r.
\]
For each fixed total loop degree, the symbol $\widehat\otimes$ allows the one-sided infinite sums in which the degree in the first tensor factor tends to $-\infty$ and that in the second tends to $+\infty$, together with finite shifts of such sums.  For example, a coefficient of $\phi_i^-(z)\otimes x_i^+(z)$ is an allowed sum of the form
\[
\sum_{r\ge r_0}\phi_{i,-r}^-\otimes x_{i,d+r}^+.
\]
All homomorphisms into completed tensor products below are understood to be continuous for this completion.

The shifted supports do not affect preservation of the current relations.
\begin{lemma}\label{lem:formal-coproduct-relations}
Let the source and target shifts be as in Theorem~\ref{thm:formal-shifted-coproduct}, and endow the completed tensor product with the graded multiplication convention fixed in Section~\ref{section: quantum affine superalgebra}.  The assignments in \eqref{coproduct for Drinfeld realization} preserve, coefficientwise, all defining current relations \eqref{SQASrelation1}--\eqref{SQASrelation7}.
\end{lemma}
\begin{proof}
At zero shift the same current identities give the Hopf structure of the unshifted quantum affine superalgebra in \cite[Theorem~3]{Zyz1997}.  The verification uses only the current relations and the graded tensor-product multiplication; it does not use the endpoints of the Cartan supports or a relation identifying the leading coefficient of $\phi_i^+(z)$ with the inverse of the leading coefficient of $\phi_i^-(z)$.  This is the super counterpart of the support-independent argument for arbitrary two-sided shifts in \cite[Lemma~10.2]{FT2019}.

We spell out the mixed relation, where both the super sign and the two Cartan currents occur.  Using the graded tensor-product rule and the Cartan--root exchange relation \eqref{SQASrelation4}, the two mixed cross terms cancel, and hence
\[
[\widetilde\Delta(x_i^+(z)),\widetilde\Delta(x_j^-(w))]
=[x_i^+(z),x_j^-(w)]\otimes\phi_j^+(w)
+\phi_i^-(z)\otimes[x_i^+(z),x_j^-(w)].
\]
Substituting \eqref{SQASrelation5}, this is zero for $i\ne j$.  For $i=j$, the formal delta identity $\delta(z/w)f(w)=\delta(z/w)f(z)$ gives
\begin{align*}
[\widetilde\Delta(x_i^+(z)),\widetilde\Delta(x_i^-(w))]
&=\frac{\delta(z/w)}{q_i-q_i^{-1}}
\bigl(\phi_i^+(z)\otimes\phi_i^+(z)-\phi_i^-(z)\otimes\phi_i^-(z)\bigr)\\
&=\frac{\delta(z/w)}{q_i-q_i^{-1}}
\bigl(\widetilde\Delta(\phi_i^+(z))-\widetilde\Delta(\phi_i^-(z))\bigr),
\end{align*}
which is the image of \eqref{SQASrelation5}.

The Cartan--Cartan relations are immediate, and the Cartan--root relations are preserved because the same exchange relation holds in each tensor factor.  For the same-sign quadratic relations, the pure terms vanish by \eqref{SQASrelation2}--\eqref{SQASrelation3}, while the mixed terms cancel after moving the Cartan current through the root current by \eqref{SQASrelation4}.  Finally, the ordinary Serre and degree-four odd Serre expressions are homogeneous current identities.  After substituting \eqref{coproduct for Drinfeld realization} and moving the Cartan factors to a fixed side, each coefficient is the same graded linear combination of the same-sign current relations and the corresponding Serre expression as in the unshifted verification of \cite[Theorem~3]{Zyz1997}.  The signs at the isotropic node come only from the graded tensor product and the parity of $x_{m,r}^{\pm}$, not from the endpoints of the Cartan supports.  Since a shift changes these endpoints only by fixed integers, all coefficients lie in the completion above.
\end{proof}

We therefore obtain the shifted formal coproduct.
\begin{theorem}\label{thm:formal-shifted-coproduct}
For type A with the standard parity sequence, let $m,n\geq1$ with $m\ne n$.  For any weights $\mu_1^{\pm},\mu_2^{\pm}\in\Lambda$, there is a continuous $\mathbb C(q)$-superalgebra homomorphism
\[
\widetilde\Delta:
\mathcal U_q^{\mu_1^++\mu_2^+,\,\mu_1^-+\mu_2^-}(\widehat{\mathfrak{sl}}_{m|n})
\longrightarrow
\mathcal U_q^{\mu_1^+,\mu_1^-}(\widehat{\mathfrak{sl}}_{m|n})\widehat\otimes
\mathcal U_q^{\mu_2^+,\mu_2^-}(\widehat{\mathfrak{sl}}_{m|n}),
\]
defined by \eqref{coproduct for Drinfeld realization}.
\end{theorem}
\begin{proof}
By Lemma~\ref{lem:formal-coproduct-relations}, the formulas preserve all defining current relations.  It remains to check the shifted supports and the leading invertible coefficients.  Write
\[
a_{k,i}^{\pm}:=(\alpha_i\mid\mu_k^{\pm}),\qquad k=1,2.
\]
The leading terms of the Cartan currents in the $k$-th tensor factor are
\[
\phi_{i,-a_{k,i}^+}^+z^{a_{k,i}^+},\qquad
\phi_{i,a_{k,i}^-}^-z^{-a_{k,i}^-}.
\]
Therefore $\widetilde\Delta(\phi_i^+(z))$ has leading power $z^{a_{1,i}^++a_{2,i}^+}$, while $\widetilde\Delta(\phi_i^-(z))$ has leading power $z^{-(a_{1,i}^-+a_{2,i}^-)}$.  Their leading coefficients are tensor products of invertible leading coefficients and hence are invertible, with the prescribed inverse generators mapped to the tensor products of the corresponding inverses.  Thus the two shifts add as required.  The coefficients in the images of $x_i^\pm(z)$ belong to the chosen completion, so the homomorphism is continuous.
\end{proof}

The same current formulas are coassociative in the completed tensor product.
\begin{proposition}\label{prop:formal-shifted-coassociativity}
For type A with the standard parity sequence, let $m,n\geq1$ with $m\ne n$, and let $(\mu_a^+,\mu_a^-)$, $a=1,2,3$, be three pairs of weights.  With the canonical identification of the iterated completed tensor products,
\[
(\widetilde\Delta_{12}\widehat\otimes\mathrm{id})\circ\widetilde\Delta_{(12),3}
=
(\mathrm{id}\widehat\otimes\widetilde\Delta_{23})\circ\widetilde\Delta_{1,(23)}.
\]
Here the subscripts record only the corresponding pairs of shifts.
\end{proposition}
\begin{proof}
It is enough to check the identity on the generating currents.  Both sides send $x_i^+(z)$ to
\[
x_i^+(z)\otimes1\otimes1
+\phi_i^-(z)\otimes x_i^+(z)\otimes1
+\phi_i^-(z)\otimes\phi_i^-(z)\otimes x_i^+(z),
\]
send $x_i^-(z)$ to
\[
1\otimes1\otimes x_i^-(z)
+1\otimes x_i^-(z)\otimes\phi_i^+(z)
+x_i^-(z)\otimes\phi_i^+(z)\otimes\phi_i^+(z),
\]
and send $\phi_i^\pm(z)$ to $\phi_i^\pm(z)\otimes\phi_i^\pm(z)\otimes\phi_i^\pm(z)$.
\end{proof}

The coproduct above is available for arbitrary two-sided shifts, but it takes values in a completion.  The non-completed coproducts of Section~\ref{section: shifted RTT and coproducts} are a different construction: they are obtained by transporting the RTT matrix coproduct.  We next construct the shift homomorphisms that will be used only to extend the one-sided non-completed coproduct from antidominant shifts to arbitrary shifts.

\subsection{Injective shift homomorphisms}\label{subsec:shift-homomorphisms}
For one-sided shifted algebras we use homomorphisms analogous to those of \cite[Lemma~10.18]{FT2019}.

\begin{lemma}\label{lem:shift homomorphism iota mu nu1 nu2}
For type A with the standard parity sequence, let $m,n\geq1$ with $m\ne n$.  For any $\mu\in\Lambda$ and $\nu_1,\nu_2\in\Lambda^-$, there is a unique superalgebra homomorphism
\[
\iota_{\mu,\nu_1,\nu_2}:
\mathcal U_q^{0,\mu}(\widehat{\mathfrak{sl}}_{m|n})
\longrightarrow
\mathcal U_q^{0,\mu+\nu_1+\nu_2}(\widehat{\mathfrak{sl}}_{m|n})
\]
defined by
\begin{align*}
x_i^+(z)&\longmapsto(1-z^{-1})^{-\alpha_i(\nu_1)}x_i^+(z),\quad x_i^-(z) \longmapsto(1-z^{-1})^{-\alpha_i(\nu_2)}x_i^-(z),\\
\phi_i^\pm(z)&\longmapsto(1-z^{-1})^{-\alpha_i(\nu_1+\nu_2)}\phi_i^\pm(z).
\end{align*}
\end{lemma}
\begin{proof}
Because $\nu_1,\nu_2$ are antidominant, all exponents $-\alpha_i(\nu_a)$ are nonnegative integers.  Direct substitution into \eqref{SQASrelation1}--\eqref{SQASrelation7} preserves the defining relations.  The positive Cartan current keeps its leading coefficient, while the negative leading coefficient satisfies
\[
\iota_{\mu,\nu_1,\nu_2}(\phi_{i,\alpha_i(\mu)}^-)
=(-1)^{\alpha_i(\nu_1+\nu_2)}
\phi_{i,\alpha_i(\mu+\nu_1+\nu_2)}^-.
\]
Thus the target has precisely the required one-sided shift.
\end{proof}

We now record the composition law and prove injectivity.
\begin{theorem}\label{th:iota mu nu1 nu2 is injective}
For type A with the standard parity sequence, assume $m,n>1$ and $m\ne n$.
\begin{itemize}
\item[(1)] For $\mu\in\Lambda$ and $\nu_1,\nu_2,\nu_3,\nu_4\in\Lambda^-$,
$
\iota_{\mu+\nu_1+\nu_2,\nu_3,\nu_4}\circ\iota_{\mu,\nu_1,\nu_2}
=
\iota_{\mu,\nu_1+\nu_3,\nu_2+\nu_4}$.
\item[(2)] For every $\mu\in\Lambda$ and $\nu_1,\nu_2\in\Lambda^-$, the homomorphism $\iota_{\mu,\nu_1,\nu_2}$ is injective.
\end{itemize}
\end{theorem}
\begin{proof}
Part~(1) follows immediately from the defining formulas.  We prove part~(2).

By Theorem~\ref{thm:shifted-triangular}, multiplication gives a triangular decomposition into the negative, Cartan, and positive current factors, and the positive and negative factors are independent of the shift.  Hence it is enough to prove injectivity on these three factors.  On the Cartan factor, after the leading monomial and leading coefficient are separated from each current, multiplication by $(1-z^{-1})^{-\alpha_i(\nu_1+\nu_2)}$ translates the logarithmic modes by scalars and acts triangularly with invertible diagonal.  Thus the Cartan restriction is injective.

For the positive factor, identify the shifted positive current algebra with the unshifted algebra $\mathcal U_q(\widehat{\mathfrak{sl}}_{m|n})^>$.  By \cite[Theorem~5.18]{Tsy2021}, there is a shuffle-algebra isomorphism
\[
\Psi:\mathcal U_q(\widehat{\mathfrak{sl}}_{m|n})^>\xrightarrow{\ \sim\ }S^{(m|n)},
\qquad
\Psi(x_{i,r}^+)=y_{i,1}^r.
\]
We use only the following features of the shuffle realization from \cite[Section~5.4]{Tsy2021}: $S^{(m|n)}$ is $\mathbb N^I$-graded, its homogeneous component $S_{\underline k}^{(m|n)}$ consists of rational functions satisfying the prescribed supersymmetry, pole, and wheel conditions, and its product is obtained by supersymmetrizing the product of two functions against the standard shuffle kernel.

Fix $\nu\in\Lambda^-$, put $d_i=-\alpha_i(\nu)\ge0$, and on degree $\underline k=(k_i)_{i\in I}$ set
\[
P_{\nu,\underline k}:=\prod_{i\in I}\prod_{r=1}^{k_i}(1-y_{i,r}^{-1})^{d_i}.
\]
The factor $P_{\nu,\underline k}$ is symmetric in the variables of every color.  Multiplication by it introduces no poles and preserves both types of wheel conditions, hence preserves $S_{\underline k}^{(m|n)}$.  Moreover, if $F$ and $G$ have degrees $\underline k$ and $\underline l$, respectively, then symmetry of the multiplier gives
\[
(P_{\nu,\underline k}F)\star(P_{\nu,\underline l}G)
=P_{\nu,\underline k+\underline l}(F\star G).
\]
Thus multiplication by the $P_{\nu,\underline k}$ defines an algebra endomorphism of $S^{(m|n)}$.  It is injective on every homogeneous component because $P_{\nu,\underline k}$ is a nonzero Laurent polynomial.

Under $\Psi$, this endomorphism is exactly the positive shift map.  Indeed,
\[
P_{\nu,\mathbf e_i}\Psi(x_{i,r}^+)
=(1-y_{i,1}^{-1})^{d_i}y_{i,1}^r
=\sum_{a=0}^{d_i}(-1)^a\binom{d_i}{a}y_{i,1}^{r-a},
\]
which is the image under $\Psi$ of the coefficient of $z^{-r}$ in $(1-z^{-1})^{d_i}x_i^+(z)$.  Hence the positive restriction is injective.  The same shuffle-multiplier argument for the negative half, equivalently the standard positive--negative current anti-isomorphism, proves injectivity of the negative restriction.  The triangular decomposition now implies injectivity of $\iota_{\mu,\nu_1,\nu_2}$.
\end{proof}

\section{Shifted RTT realizations and non-completed coproducts}
\label{section: shifted RTT and coproducts}

Assume throughout this section that $m,n>1$ and $m\ne n$.  We pass from the completed coproduct of Section~\ref{section: Coproduct} to algebraic, non-completed coproducts.  For antidominant two-sided shifts, the shifted RTT realization and Ding--Frenkel isomorphism transport the RTT matrix coproduct to the current algebras; the one-sided specialization is made explicit by the finite presentation of Section~\ref{section: Levendorskii type presentation of shifted quantum affine superalgebra}.  The shift homomorphisms of Section~\ref{section: Coproduct} then extend the one-sided coproduct to arbitrary shifts.

\subsection{Antidominantly shifted RTT algebras}\label{subsec:shifted-super-rtt}
We introduce an antidominantly shifted RTT superalgebra associated with the Perk--Schultz $R$-matrix, whose Gauss-current ratios match the Cartan valuations in Definition~\ref{Def: shifted quantum affine superalgebra}. We retain the full general-linear diagonal data, since a quantum Berezinian reduction would identify central diagonal parameters that remain independent in the shifted current algebra; see Proposition~\ref{prop:central-leading-reduction}.

Let $\mu_{+},\mu_{-}\in\Lambda^{-}$ and set $l_i^{\pm}:=(\alpha_i\mid\mu_{\pm})\leq0$ for $i\in I$.
Choose integer lifts
\[
\mathbf d^{\pm}=(d_a^{\pm})_{a\in J}\in\mathbb Z^{J},\quad d_{i+1}^{\pm}-d_i^{\pm}=l_i^{\pm}\quad(i\in I).
\]
Thus, each sequence $d_1^{\pm},\ldots,d_M^{\pm}$ is weakly decreasing.  Such a lift is unique up to adding a common integer to all its entries.

\begin{definition}[Antidominantly shifted RTT superalgebra]\label{def:shifted-super-rtt}
Let $\mathscr U^{\mathrm{rtt}}_{\mathbf d^{+},\mathbf d^{-}}$ be the superalgebra generated by homogeneous symbols
$e_{ab}^{(r)}$,$f_{ba}^{(r)}$ $(1\leq a<b\leq M,\ r\in\mathbb Z)$,
together with even symbols
$g_{a,r}^{+}$ $(r\geq-d_a^{+})$,
$g_{a,-r}^{-}$ $(r\geq-d_a^{-})$,
and the inverses of $g_{a,-d_a^{+}}^{+}$ and $g_{a,d_a^{-}}^{-}$.  We set
\[
|e_{ab}^{(r)}|=|f_{ba}^{(r)}|=|a|+|b|,
\quad |g_{a,r}^{\pm}|=0,
\]
and form the Gauss series
\begin{gather*}
e_{ab}^{+}(z)=\sum_{r\geq0}e_{ab}^{(r)}z^{-r},\quad
e_{ab}^{-}(z)=\sum_{r<0}e_{ab}^{(r)}z^{-r},
\\
f_{ba}^{+}(z)=\sum_{r>0}f_{ba}^{(r)}z^{-r},\quad
f_{ba}^{-}(z)=\sum_{r\leq0}f_{ba}^{(r)}z^{-r},\\
g_a^{+}(z)=\sum_{r\geq-d_a^{+}}g_{a,r}^{+}z^{-r},\quad
g_a^{-}(z)=\sum_{r\geq-d_a^{-}}g_{a,-r}^{-}z^{r}.
\end{gather*}
Set
\begin{gather*}
F^{\epsilon}(z)=1+\sum_{a<b}f_{ba}^{\epsilon}(z)\otimes\mathsf E_{ba},\quad
G^{\epsilon}(z)=\sum_{a\in J}g_a^{\epsilon}(z)\otimes\mathsf E_{aa},\quad
E^{\epsilon}(z)=1+\sum_{a<b}e_{ab}^{\epsilon}(z)\otimes\mathsf E_{ab},\\
\mathsf T^{\epsilon}(z)=F^{\epsilon}(z)G^{\epsilon}(z)E^{\epsilon}(z),
\quad \epsilon\in\{+,-\}.
\end{gather*}
The defining relations are the four graded RTT relations associated with the Perk--Schultz matrix
\begin{equation}\label{eq:shifted-super-rtt-relations}
R_{23}(z,w)\mathsf T_{12}^{\epsilon}(z)\mathsf T_{13}^{\epsilon'}(w)
=
\mathsf T_{13}^{\epsilon'}(w)\mathsf T_{12}^{\epsilon}(z)R_{23}(z,w),
\quad \epsilon,\epsilon'\in\{+,-\},
\end{equation}
together with the invertibility of the displayed leading diagonal coefficients.  Here $R(z,w)$ is the Perk--Schultz matrix \eqref{eq:Perk-Schultz-R}, and all tensor products and matrix products are graded.
\end{definition}

At zero shift, Definition~\ref{def:shifted-super-rtt} is the extended super RTT algebra in which the two diagonal leading systems are not identified with one another.  This is the RTT object naturally corresponding to the central extension $\mathcal U_q^{0,0}$, rather than to its central reduction in Proposition~\ref{prop:central-leading-reduction}.

The common normalization of the lifts does not affect the RTT algebra.
\begin{lemma}\label{lem:shifted-super-rtt-lift}
For any $c^{+},c^{-}\in\mathbb Z$, replacing $\mathbf d^{+}$ by $\mathbf d^{+}+c^{+}(1,\ldots,1)$ and $\mathbf d^{-}$ by $\mathbf d^{-}+c^{-}(1,\ldots,1)$ does not change the isomorphism class of $\mathscr U^{\mathrm{rtt}}_{\mathbf d^{+},\mathbf d^{-}}$. More precisely, $\mathsf T^{+}(z)\mapsto z^{c^{+}}\mathsf T^{+}(z)$ and $\mathsf T^{-}(z)\mapsto z^{-c^{-}}\mathsf T^{-}(z)$ define an isomorphism between the corresponding shifted RTT superalgebras.
\end{lemma}
\begin{proof}
The scalar monomials cancel on the two sides of every RTT relation \eqref{eq:shifted-super-rtt-relations}.  They shift all diagonal valuations by the prescribed common constants and leave the off-diagonal Gauss series unchanged.  The inverse assignment uses $-c^{+},-c^{-}$.
\end{proof}

Define the simple Gauss currents in $\mathscr U^{\mathrm{rtt}}_{\mathbf d^{+},\mathbf d^{-}}$ by
\begin{gather}\label{eq:shifted-super-simple-gauss-currents}
\mathsf x_i^{+}(z)
:=\frac{e_{i,i+1}^{+}(\tau_i z)-e_{i,i+1}^{-}(\tau_i z)}{q_i-q_i^{-1}},\quad
\mathsf x_i^{-}(z)
:=\frac{f_{i+1,i}^{+}(\tau_i z)-f_{i+1,i}^{-}(\tau_i z)}{q_i-q_i^{-1}},\\
\mathsf\phi_i^{\pm}(z)
:=g_{i+1}^{\pm}(\tau_i z)\bigl(g_i^{\pm}(\tau_i z)\bigr)^{-1}.\label{eq:shifted-super-simple-gauss-current2}
\end{gather}
The inverses in \eqref{eq:shifted-super-simple-gauss-current2} exist as formal Laurent series because the leading diagonal coefficients are invertible.
By construction,
\begin{align}\label{eq:shifted-super-gauss-support}
\mathsf\phi_i^{+}(z)\in z^{l_i^{+}}\mathscr U^{\mathrm{rtt}}_{\mathbf d^{+},\mathbf d^{-}}[[z^{-1}]],\quad
\mathsf\phi_i^{-}(z)\in z^{-l_i^{-}}\mathscr U^{\mathrm{rtt}}_{\mathbf d^{+},\mathbf d^{-}}[[z]],
\end{align}
and the first nonzero coefficient of each series is invertible.  Let
$\mathscr U^{\mathrm{rtt},\mathrm{sl}}_{\mathbf d^{+},\mathbf d^{-}}$
be the subalgebra generated by the coefficients of the currents in
\eqref{eq:shifted-super-simple-gauss-currents} and the inverses of their leading Cartan coefficients.

\begin{proposition}\label{prop:shifted-super-df-epi}
For $m,n>1$, $m\ne n$, and $\mu_{+},\mu_{-}\in\Lambda^{-}$, the assignment
\[
x_i^{\pm}(z)\longmapsto\mathsf x_i^{\pm}(z),\quad\phi_i^{\pm}(z)\longmapsto\mathsf\phi_i^{\pm}(z)
\]
defines a surjective superalgebra homomorphism
\[
\Gamma_{\mathbf d^{+},\mathbf d^{-}}:
\mathcal U_q^{\mu_{+},\mu_{-}}(\widehat{\mathfrak{sl}}_{m|n})
\twoheadrightarrow
\mathscr U^{\mathrm{rtt},\mathrm{sl}}_{\mathbf d^{+},\mathbf d^{-}}.
\]
It depends, up to the canonical isomorphism of Lemma~\ref{lem:shifted-super-rtt-lift}, only on the shifts $\mu_{+},\mu_{-}$.
\end{proposition}
\begin{proof}
The support statement \eqref{eq:shifted-super-gauss-support} follows from
$d_{i+1}^{\pm}-d_i^{\pm}=l_i^{\pm}$.  Hence the two Cartan currents have exactly the supports required in Definition~\ref{Def: shifted quantum affine superalgebra}, with independent invertible leading coefficients.

It remains to check the current relations.  The Gauss-component calculation which gives Theorem~\ref{th: RTT and Drinfeld realization} is a formal consequence of the graded RTT relation and the Gauss factorization.  In particular, its Cartan--Cartan, Cartan--root, mixed, same-sign, ordinary Serre, and odd Serre identities are identities of formal Laurent series; the calculation does not use the zero-shift locations of the first nonzero diagonal coefficients.  In Definition~\ref{def:shifted-super-rtt} the off-diagonal half-current expansions are exactly the same as in the unshifted RTT algebra, while the diagonal series are merely allowed to begin in different degrees and still have invertible leading coefficients.  Therefore the same Gauss-component identities hold coefficientwise.  Applying the same spectral rescaling $z\mapsto\tau_i z$ and normalization by $q_i-q_i^{-1}$ as in Corollary~\ref{coro: iso from Drinfeld realization to RTT realization} gives precisely relations \eqref{SQASrelation1}--\eqref{SQASrelation7}.  This includes the quartic odd Serre relation at $i=m$: its RTT derivation is unchanged because the parity sequence and the off-diagonal Gauss expansions have not been altered.

Thus the displayed assignment defines a homomorphism.  Surjectivity follows from the definition of $\mathscr U^{\mathrm{rtt},\mathrm{sl}}_{\mathbf d^{+},\mathbf d^{-}}$.  Finally, a common change of lift multiplies $g_{i+1}^{\pm}$ and $g_i^{\pm}$ by the same scalar monomial, so their ratio is unchanged.
\end{proof}

The RTT coproduct itself can be constructed without assuming that
$\Gamma_{\mathbf d^{+},\mathbf d^{-}}$ is injective.

\begin{theorem}\label{thm:shifted-super-rtt-coproduct}
For $a=1,2$, let
$\mathbf d_a^{\pm}=(d_{a,b}^{\pm})_{b\in J}$ be weakly decreasing integer sequences.  There is a unique superalgebra homomorphism
\[
\Delta^{\mathrm{rtt}}_{\mathbf d_1,\mathbf d_2}:
\mathscr U^{\mathrm{rtt}}_{\mathbf d_1^{+}+\mathbf d_2^{+},\,\mathbf d_1^{-}+\mathbf d_2^{-}}
\longrightarrow
\mathscr U^{\mathrm{rtt}}_{\mathbf d_1^{+},\mathbf d_1^{-}}
\otimes
\mathscr U^{\mathrm{rtt}}_{\mathbf d_2^{+},\mathbf d_2^{-}}
\]
characterized by graded matrix multiplication
\begin{equation}\label{eq:shifted-super-rtt-coproduct-matrix}
\Delta^{\mathrm{rtt}}_{\mathbf d_1,\mathbf d_2}(\mathsf T^{\epsilon}(z))
=
\mathsf T_{(1)}^{\epsilon}(z)\mathsf T_{(2)}^{\epsilon}(z),
\quad \epsilon\in\{+,-\}.
\end{equation}
Equivalently, if $t_{ij}^{\epsilon}(z)$ denotes the $(i,j)$ entry of $\mathsf T^{\epsilon}(z)$, then
\begin{equation}\label{eq:shifted-super-rtt-coproduct-entry}
\Delta^{\mathrm{rtt}}_{\mathbf d_1,\mathbf d_2}(t_{ij}^{\epsilon}(z))
=
\sum_{k\in J}\epsilon_{ijk}\,
 t_{ik}^{\epsilon}(z)\otimes t_{kj}^{\epsilon}(z).
\end{equation}
\end{theorem}
\begin{proof}
The RTT relations are preserved by \eqref{eq:shifted-super-rtt-coproduct-matrix} by the standard graded RTT argument: the two tensor factors supercommute and the Perk--Schultz matrix satisfies the graded Yang--Baxter equation.  Once the product matrix is shown to have the required Gauss form, uniqueness of formal Gauss decomposition determines the images of all Gauss generators and also proves uniqueness of the homomorphism.  It therefore remains to verify the prescribed Gauss supports.

Write
\[
\mathsf T_{(a)}^{\epsilon}=F_{(a)}^{\epsilon}G_{(a)}^{\epsilon}E_{(a)}^{\epsilon}
\quad(a=1,2).
\]
In the $z^{-1}$-adic topology for $\epsilon=+$, and in the $z$-adic topology for $\epsilon=-$, the product
$E_{(1)}^{\epsilon}F_{(2)}^{\epsilon}$ admits a unique Gauss decomposition
\[
E_{(1)}^{\epsilon}(z)F_{(2)}^{\epsilon}(z)
=\overline F^{\epsilon}(z)\overline G^{\epsilon}(z)\overline E^{\epsilon}(z).
\]
The diagonal factor has leading term $1$; for $\epsilon=+$ one of the two off-diagonal factors may have a constant upper-triangular term and for $\epsilon=-$ one may have a constant lower-triangular term, exactly as in the unshifted Gauss multiplication.  Consequently,
\begin{align*}
\mathsf T_{(1)}^{\epsilon}\mathsf T_{(2)}^{\epsilon}
=F_{(1)}^{\epsilon}
\bigl(G_{(1)}^{\epsilon}\overline F^{\epsilon}(G_{(1)}^{\epsilon})^{-1}\bigr)
\bigl(G_{(1)}^{\epsilon}\overline G^{\epsilon}G_{(2)}^{\epsilon}\bigr)
\bigl((G_{(2)}^{\epsilon})^{-1}\overline E^{\epsilon}G_{(2)}^{\epsilon}\bigr)
E_{(2)}^{\epsilon}
\end{align*}
is again a Gauss decomposition.

Only the mode bounds require attention.  Conjugating a lower Gauss coefficient in position $(b,a)$, $a<b$, by $G_{(1)}^{+}$ multiplies its leading degree by the ratio $g_b^{+}/g_a^{+}$, whose leading $z$-degree is $d_{1,b}^{+}-d_{1,a}^{+}\leq0$.  Hence no forbidden positive $z$-degree is created in the $+$ expansion.  The corresponding upper-triangular conjugation by $G_{(2)}^{+}$ is controlled by $d_{2,b}^{+}-d_{2,a}^{+}\leq0$.  The same argument in the $z$-adic topology, using the leading degrees $-d_{a,b}^{-}$, gives the required bounds for the $-$ expansion.  The super signs only change coefficients and do not affect these valuations.  Finally,
$G_{(1)}^{\epsilon}\overline G^{\epsilon}G_{(2)}^{\epsilon}$ has diagonal leading degrees equal to the sums
$d_{1,a}^{\pm}+d_{2,a}^{\pm}$, with invertible leading coefficients.  Thus the product satisfies Definition~\ref{def:shifted-super-rtt} for the summed shifts, proving that \eqref{eq:shifted-super-rtt-coproduct-matrix} is well defined.
\end{proof}

\begin{corollary}\label{cor:shifted-super-rtt-coassociativity}
For three pairs of weakly decreasing lifts $\mathbf d_1^{\pm},\mathbf d_2^{\pm},\mathbf d_3^{\pm}$,
\begin{align*}
(\Delta^{\mathrm{rtt}}_{\mathbf d_1,\mathbf d_2}\otimes\id)
\circ\Delta^{\mathrm{rtt}}_{\mathbf d_1+\mathbf d_2,\mathbf d_3}
=
(\id\otimes\Delta^{\mathrm{rtt}}_{\mathbf d_2,\mathbf d_3})
\circ\Delta^{\mathrm{rtt}}_{\mathbf d_1,\mathbf d_2+\mathbf d_3}.
\end{align*}
\end{corollary}
\begin{proof}
Both sides send $\mathsf T^{\epsilon}(z)$ to the same graded triple matrix product
$\mathsf T_{(1)}^{\epsilon}(z)\mathsf T_{(2)}^{\epsilon}(z)\mathsf T_{(3)}^{\epsilon}(z)$.  Uniqueness of the formal Gauss decomposition then implies equality on all Gauss generators.
\end{proof}

\subsubsection{Three-factor coefficient separation}\label{subsubsec:three-factor-rtt-separation}
The epimorphism of Proposition~\ref{prop:shifted-super-df-epi} sends the shifted PBWD basis of Corollary~\ref{cor:shifted-pbw} to a spanning set of the shifted RTT superalgebra $\mathscr U^{\mathrm{rtt}}_{\mathbf d^{+},\mathbf d^{-}}$.  
To prove linear independence, we attach two zero-shift RTT factors and extract an extremal coefficient, using only the ordinary Ding--Frenkel isomorphism, the RTT matrix coproduct, and the shifted PBWD basis.

Write
\[
 U=\mathcal U_q^{\mu_+,\mu_-}(\widehat{\mathfrak{sl}}_{m|n}),
 \quad
 \mathscr R=\mathscr U^{\mathrm{rtt},\mathrm{sl}}_{\mathbf d^+,\mathbf d^-},
 \quad
 \Gamma=\Gamma_{\mathbf d^+,\mathbf d^-}.
\]
Let
$K_i^+=\phi_{i,-l_i^+}^+$,
$K_i^-=\phi_{i,l_i^-}^-$,
and write a Cartan PBWD monomial in the form
\[
 H_{\mathbf c}=K_{\mathbf u}H_{\mathbf v}^{\circ},\quad
 K_{\mathbf u}=\prod_i(K_i^+)^{u_i^+}(K_i^-)^{u_i^-},
 \quad
 H_{\mathbf v}^{\circ}
 =\prod_{i,r>0}h_{i,r}^{v_{i,r}^+}h_{i,-r}^{v_{i,r}^-}.
\]
For the root labels of Corollary~\ref{cor:shifted-pbw}, put
\[
 \mathrm{wt}(\mathbf a^\pm)
 =\sum_{\beta,r}\mathbf a^\pm(\beta,r)\,\beta
 =\sum_i\omega_i^\pm\alpha_i,
 \quad
 |\mathbf v|=\sum_{i,r>0}(v_{i,r}^++v_{i,r}^-).
\]
Thus a full PBWD basis element will be written
\begin{equation}\label{eq:three-factor-full-pbwd}
 B_\lambda=F_{\mathbf a^-}K_{\mathbf u}H_{\mathbf v}^{\circ}E_{\mathbf a^+}.
\end{equation}

We first record the zero-shift diagonal rescalings used in the injectivity argument.
\begin{lemma}
\label{lem:zero-shift-rtt-diagonal-rescaling}
Let
\[
 \mathbb L_{\mathbf p}=\mathbb C(q)[p_1^{\pm1},\ldots,p_{M-1}^{\pm1}],
 \quad
 \mathbb L_{\mathbf r}=\mathbb C(q)[r_1^{\pm1},\ldots,r_{M-1}^{\pm1}].
\]
There are superalgebra homomorphisms from the full zero-shift extended RTT superalgebra of Definition~\ref{def:shifted-super-rtt} to Laurent extensions of the ordinary quantum affine superalgebra $\mathcal U_q(\widehat{\mathfrak{sl}}_{m|n})$,
\begin{align}
 \rho_{\mathbf p}^{\mathrm L}:\mathscr U^{\mathrm{rtt}}_{\mathbf0,\mathbf0}
 &\longrightarrow
 \mathcal U_q(\widehat{\mathfrak{sl}}_{m|n})\otimes\mathbb L_{\mathbf p},
 \label{eq:rho-left}\\
 \rho_{\mathbf r}^{\mathrm R}:\mathscr U^{\mathrm{rtt}}_{\mathbf0,\mathbf0}
 &\longrightarrow
 \mathcal U_q(\widehat{\mathfrak{sl}}_{m|n})\otimes\mathbb L_{\mathbf r},
 \label{eq:rho-right}
\end{align}
which on the simple Gauss currents are given by
\begin{align}
 \rho_{\mathbf p}^{\mathrm L}(\mathsf x_i^+(z))&=x_i^+(z),&
 \rho_{\mathbf p}^{\mathrm L}(\mathsf x_i^-(z))&=p_i x_i^-(z),&
 \rho_{\mathbf p}^{\mathrm L}(\mathsf\phi_i^\pm(z))&=p_i\phi_i^\pm(z),
 \label{eq:rho-left-currents}\\
 \rho_{\mathbf r}^{\mathrm R}(\mathsf x_i^+(z))&=r_i x_i^+(z),&
 \rho_{\mathbf r}^{\mathrm R}(\mathsf x_i^-(z))&=x_i^-(z),&
 \rho_{\mathbf r}^{\mathrm R}(\mathsf\phi_i^\pm(z))&=r_i\phi_i^\pm(z).
 \label{eq:rho-right-currents}
\end{align}
In particular, if \(\Gamma_{\mathbf0,\mathbf0}\) denotes the zero-shift
map of Proposition~\ref{prop:shifted-super-df-epi}, then both
\(\rho_{\mathbf p}^{\mathrm L}\Gamma_{\mathbf0,\mathbf0}\) and
\(\rho_{\mathbf r}^{\mathrm R}\Gamma_{\mathbf0,\mathbf0}\) are injective.
\end{lemma}
\begin{proof}
Choose an even diagonal matrix
\[
 D_{\mathbf p}=\operatorname{diag}(d_1,\ldots,d_M),\quad
 d_1=1,\quad d_{i+1}=p_i d_i.
\]
Since \((D_{\mathbf p}\otimes D_{\mathbf p})R(z,w)
=R(z,w)(D_{\mathbf p}\otimes D_{\mathbf p})\), substituting
\(D_{\mathbf p}T^\pm(z)\) for the two zero-shift RTT matrices preserves all
four RTT relations.  Its Gauss decomposition is
\[
 D_{\mathbf p}F^\pm G^\pm E^\pm
 =(D_{\mathbf p}F^\pm D_{\mathbf p}^{-1})
  (D_{\mathbf p}G^\pm)E^\pm.
\]
Hence the lower simple Gauss coordinate is multiplied by \(p_i\), the
upper one is unchanged, and the consecutive diagonal ratio is multiplied
by \(p_i\).  Now let $T_{\mathrm{sl}}^{\pm}(z)$ denote the images of the zero-shift RTT
matrices in the quantum Berezinian quotient
$\mathcal U_q^{\mathrm{rtt}}(\widehat{\mathfrak{sl}}_{m|n})$.
The matrices $D_{\mathbf p}T_{\mathrm{sl}}^{\pm}(z)$ still satisfy all
four general-linear RTT relations.  Since the \emph{source} of
\eqref{eq:rho-left} is the extended RTT superalgebra of Definition~\ref{def:shifted-super-rtt}, it imposes no
quantum Berezinian relation that these images have to satisfy.  The universal
property of the RTT presentation therefore gives a homomorphism to the
ordinary special-linear RTT quotient tensored with $\mathbb L_{\mathbf p}$.
Composing with the inverse of the isomorphism in
Corollary~\ref{coro: iso from Drinfeld realization to RTT realization}
gives \eqref{eq:rho-left} and the simple-current formulas
\eqref{eq:rho-left-currents}.  For the right rescaling, choose $D_{\mathbf r}=\operatorname{diag}(e_1,\ldots,e_M)$ with $e_1=1$ and $e_{i+1}=r_i e_i$. Then \(T^\pm(z)D_{\mathbf r}\) gives \eqref{eq:rho-right}--\eqref{eq:rho-right-currents} in the same way.

For injectivity, put
\[
 \boldsymbol\delta=\mathbf u_+-\mathbf u_-,
 \quad
 \boldsymbol\sigma=\mathbf u_++\mathbf u_-.
\]
The PBWD basis and \eqref{eq:rho-left-currents} give
\begin{equation}\label{eq:rho-left-pbwd}
 (\rho_{\mathbf p}^{\mathrm L}\Gamma_{\mathbf0,\mathbf0})(B_\lambda)
 =\mathbf p^{\boldsymbol\omega^-+\boldsymbol\sigma}
 F_{\mathbf a^-}
 \Bigl(\prod_i k_i^{\delta_i}\Bigr)
 H_{\mathbf v}^{\circ}E_{\mathbf a^+},
\end{equation}
where \(k_i=\phi_{i,0}^+\) in $\mathcal U_q(\widehat{\mathfrak{sl}}_{m|n})$.
The PBWD monomial in $\mathcal U_q(\widehat{\mathfrak{sl}}_{m|n})$ determines
\(\mathbf a^-,\mathbf a^+,\mathbf v\), and
\(\boldsymbol\delta\); after \(\boldsymbol\omega^-\) is known, the
Laurent exponent determines \(\boldsymbol\sigma\).  Thus it determines
\(\mathbf u_+\) and \(\mathbf u_-\) separately.  Distinct shifted PBWD
basis elements therefore have distinct images in
\eqref{eq:rho-left-pbwd}.  This proves the first injectivity assertion;
the second is analogous.
\end{proof}

Let \(\Delta^{\mathrm{rtt},(3)}_{\mathbf d^+,\mathbf d^-}\) denote the
three-factor RTT homomorphism obtained from
Corollary~\ref{cor:shifted-super-rtt-coassociativity}, characterized by
\[
 \mathsf T^\epsilon(z)\longmapsto
 \mathsf T_{(\mathrm L)}^\epsilon(z)
 \mathsf T_{(\mathrm M)}^\epsilon(z)
 \mathsf T_{(\mathrm R)}^\epsilon(z),
\]
with zero shifts on the left and right and
\((\mathbf d^+,\mathbf d^-)\) in the middle.  Since the two maps in
Lemma~\ref{lem:zero-shift-rtt-diagonal-rescaling} are defined on the
\emph{full} zero-shift RTT algebra, they may be applied directly to the two
outer factors of this three-factor RTT homomorphism.  No current-level
coproduct compatibility, and in particular no injectivity of $\Gamma$, is
used in this step.  Restricting the resulting map to $\mathscr R$ gives
\begin{equation}\label{eq:three-factor-Xi}
 \Xi_{\mathbf p,\mathbf r}:\mathscr R\longrightarrow
 \bigl(\mathcal U_q(\widehat{\mathfrak{sl}}_{m|n})\otimes\mathbb L_{\mathbf p}\bigr)
 \otimes\mathscr U^{\mathrm{rtt}}_{\mathbf d^+,\mathbf d^-}
 \otimes
 \bigl(\mathcal U_q(\widehat{\mathfrak{sl}}_{m|n})\otimes\mathbb L_{\mathbf r}\bigr).
\end{equation}
We keep the full shifted RTT algebra in the middle; no injectivity statement about it is used below.  All coefficient extractions below are algebraic: for a fixed matrix coefficient the RTT matrix product contains only finitely many decompositions of its loop degree, and a fixed PBWD monomial is a finite product of such coefficients.  Thus its image has finite support in each outer PBWD degree used below.

On each outer copy of $\mathcal U_q(\widehat{\mathfrak{sl}}_{m|n})$, expand in the PBWD basis.  We use the standard filtration obtained by bounding separately the total heights of the negative and positive PBWD root factors. Cartan--root reordering preserves these two bounds, while a mixed positive--negative commutator is strictly lower in at least one of them; this is the usual triangularity underlying the PBWD normal ordering.  Thus the associated graded records the two root halves separately. For a basis monomial \(FHE\), let \(d_-(FHE)\) and \(d_+(FHE)\) be the heights of the negative and positive root weights, respectively, and let \(d_0(FHE)\) be the total degree in the nonleading Cartan modes \(h_{i,\pm r}\).  Order pairs of outer PBWD monomials first lexicographically by
\begin{equation}\label{eq:outer-degree-order}
\bigl(d_-^{\mathrm L},\ d_+^{\mathrm R},\ -d_+^{\mathrm L},\ -d_-^{\mathrm R},\ d_0^{\mathrm L}\bigr).
\end{equation}
The proof below shows that, for the image of each full PBWD monomial, the maximal degree in \eqref{eq:outer-degree-order} contains a single outer PBWD--Laurent monomial.  Thus any fixed refinement to a total order gives the same leading term; no choice of refinement is used to create the separation.

\begin{lemma}
\label{lem:three-factor-rtt-leading-coefficient}
For \(B_\lambda\) as in \eqref{eq:three-factor-full-pbwd}, put
$\boldsymbol\delta=\mathbf u_+-\mathbf u_-$,
$\boldsymbol\sigma=\mathbf u_++\mathbf u_-$.
With respect to \eqref{eq:outer-degree-order}, the image
\(\Xi_{\mathbf p,\mathbf r}\Gamma(B_\lambda)\) has the unique leading
outer PBWD term
\begin{equation}\label{eq:three-factor-rtt-leading-term}
 \mathbf p^{\boldsymbol\omega^-+\boldsymbol\sigma}
 \mathbf r^{\boldsymbol\omega^++\boldsymbol\sigma}
 \Bigl(F_{\mathbf a^-}k^{\boldsymbol\delta}H_{\mathbf v}^{\circ}\Bigr)
 \otimes \Gamma(K_{\mathbf u})\otimes
 \Bigl(k^{\boldsymbol\delta}E_{\mathbf a^+}\Bigr).
\end{equation}
Here \(k^{\boldsymbol\delta}=\prod_i k_i^{\delta_i}\).  The outer labels in
\eqref{eq:three-factor-rtt-leading-term} are different for different full
PBWD labels \(\lambda\).
\end{lemma}
\begin{proof}
We first isolate the two outer Gauss factors exactly.  This avoids
taking a lowest-degree component in the full outer algebra, where a
mixed root commutator can produce a Cartan term.
Write $T_a^\epsilon=F_a^\epsilon G_a^\epsilon E_a^\epsilon$ and set
\[
 Q^\epsilon=G_{\mathrm L}^\epsilon E_{\mathrm L}^\epsilon
 T_{\mathrm M}^\epsilon F_{\mathrm R}^\epsilon G_{\mathrm R}^\epsilon.
\]
Then the three-factor matrix product is exactly
$F_{\mathrm L}^\epsilon Q^\epsilon E_{\mathrm R}^\epsilon$.
Let $B_{\mathrm L}^{\geq0}$ be the current Borel generated by the Cartan and positive modes in the left outer copy of $\mathcal U_q(\widehat{\mathfrak{sl}}_{m|n})$, and let $B_{\mathrm R}^{\leq0}$ be the current Borel generated by the Cartan and negative modes in the right outer copy.  Include
the Laurent rescaling parameters in these algebras, and put
\[
 \mathcal C=B_{\mathrm L}^{\geq0}\otimes
 \mathscr U^{\mathrm{rtt}}_{\mathbf d^+,\mathbf d^-}
 \otimes B_{\mathrm R}^{\leq0}.
\]
All coefficients of $Q^\epsilon$ lie in $\mathcal C$.  Recursive Gauss
elimination stays in $\mathcal C$: the leading diagonal coefficients
are products of the invertible leading Cartan coefficients of the
three factors, and only their inverses are required for coefficientwise
inversion.  The prescribed antidominant supports ensure that these products occur in the allowed leading degrees, as in the RTT matrix-product construction.  Consequently
the Gauss decomposition
$Q^\epsilon=\widehat F^\epsilon\widehat G^\epsilon
\widehat E^\epsilon$ has coefficients in $\mathcal C$.
Uniqueness of Gauss decomposition now gives
\begin{equation}\label{eq:three-factor-exact-outer-gauss}
 F_{\rm tot}^\epsilon=F_{\mathrm L}^\epsilon\widehat F^\epsilon,
 \quad G_{\rm tot}^\epsilon=\widehat G^\epsilon,
 \quad E_{\rm tot}^\epsilon=\widehat E^\epsilon E_{\mathrm R}^\epsilon.
\end{equation}
In particular each simple negative current is the left simple-current
plus an element with coefficients in $\mathcal C$, and each simple
positive current is the right simple-current plus such an element.
Every Cartan image also has coefficients in $\mathcal C$.  Thus none
of the remainder coefficients contains a left negative PBWD factor or
a right positive PBWD factor.

Here is the resulting bound, including the effect of contractions.
Express each homogeneous PBWD root vector as its defining finite
linear combination of words in simple root modes.  Each negative
word has exactly the height of its root weight many simple factors;
the same holds on the positive side.  In the expansion of $B_\lambda$,
mark a factor whenever its left negative summand, respectively right
positive summand, from \eqref{eq:three-factor-exact-outer-gauss} is chosen.
Before normal ordering, the number of left negative simple factors
is at most $\sum_i\omega_i^-$, with equality only if every negative
factor was marked.  The corresponding right positive count is at most
$\sum_i\omega_i^+$, with equality only if every positive factor was
marked.  All other coefficients belong to $\mathcal C$.

PBWD normal ordering cannot increase either count.  Cartan--root
reordering preserves root height, as does same-sign root reordering.
On the left, using the commutator term in a mixed simple-root
interchange removes one negative factor; on the right it removes one
positive factor.  Applying these rules first to simple-root words
also proves the bounds for composite PBWD root vectors.  Thus every
normal-ordered term satisfies
\[
 d_-^{\mathrm L}\le\sum_i\omega_i^-,\quad
 d_+^{\mathrm R}\le\sum_i\omega_i^+.
\]
A term attaining both bounds must come from marking every root factor,
and cannot have used any contraction that lowers these counts.
Summing these marked words reconstructs exactly
$F_{\mathbf a^-}$ on the left and $E_{\mathbf a^+}$ on the right,
with the factors $\mathbf p^{\boldsymbol\omega^-}$ and
$\mathbf r^{\boldsymbol\omega^+}$.  In particular this step introduces
neither a different loop-mode label nor a second root monomial of the
same extremal heights.

It remains to determine which Cartan terms have
$d_+^{\mathrm L}=d_-^{\mathrm R}=0$.  This is a quotient calculation
\emph{inside the Borel algebra} $\mathcal C$.  The ideals generated
by the left positive modes and by the right negative modes are
two-sided in their respective Borels; their quotients are the Cartan
algebras.  Denote the resulting quotient map by $\pi$.
There is no mixed-sign commutator in these quotients.  Gauss elimination
commutes with $\pi$, since the inverted leading coefficients remain
invertible.  Explicitly,
\begin{align*}
\pi(Q^\epsilon)=G_{\mathrm L}^\epsilon F_{\mathrm M}^\epsilon G_{\mathrm M}^\epsilon E_{\mathrm M}^\epsilon G_{\mathrm R}^\epsilon=\bigl(G_{\mathrm L}^\epsilon F_{\mathrm M}^\epsilon(G_{\mathrm L}^\epsilon)^{-1}\bigr)\bigl(G_{\mathrm L}^\epsilon G_{\mathrm M}^\epsilon G_{\mathrm R}^\epsilon\bigr)\bigl((G_{\mathrm R}^\epsilon)^{-1}E_{\mathrm M}^\epsilon G_{\mathrm R}^\epsilon\bigr).
\end{align*}
Hence the diagonal factor after this quotient is exactly the product
of the three diagonal factors.

For the marked root contribution the order is already negative--Cartan--positive:
the left negative factor precedes all the $\mathcal C$ coefficients,
and the right positive factor follows them.  Therefore a nonzero
left positive or right negative factor in those coefficients cannot
be contracted away while retaining both extremal heights.  Taking
$\pi$ consequently extracts precisely the part with the maximal
third and fourth entries of \eqref{eq:outer-degree-order}.
In this part the leading Cartan coefficients multiply, and the
normalized Cartan logarithms satisfy
\[
 h_{i,r}\longmapsto h_{i,r}^{(\mathrm L)}
                    +h_{i,r}^{(\mathrm M)}+h_{i,r}^{(\mathrm R)}.
\]
The largest possible $d_0^{\mathrm L}$ is $|\mathbf v|$ and is attained
only by choosing the left summand of every nonleading Cartan factor.
Its coefficient is one.  The rescaled leading Cartan factors supply
$\mathbf p^{\boldsymbol\sigma}\mathbf r^{\boldsymbol\sigma}$,
$k^{\boldsymbol\delta}$ in both outer factors, and
$\Gamma(K_{\mathbf u})$ in the middle.  This proves
\eqref{eq:three-factor-rtt-leading-term}.  The indicated factors are
already in PBWD order, so no further super sign occurs.

Finally the outer PBWD monomials determine
$\mathbf a^-,\mathbf a^+,\mathbf v,\boldsymbol\delta$.
Their root weights determine $\boldsymbol\omega^\pm$, and either
Laurent exponent then determines $\boldsymbol\sigma$.  Thus
$\mathbf u_\pm=(\boldsymbol\sigma\pm\boldsymbol\delta)/2$.
The displayed extremal outer labels are therefore distinct for
distinct full PBWD labels.
\end{proof}

\begin{theorem}
\label{thm:shifted-super-df-isomorphism}
Assume \(m,n>1\) and \(m\ne n\).  For every pair of antidominant shifts
\(\mu_+,\mu_-\), the homomorphism of
Proposition~\ref{prop:shifted-super-df-epi}
\[
 \Gamma_{\mathbf d^+,\mathbf d^-}:
 \mathcal U_q^{\mu_+,\mu_-}(\widehat{\mathfrak{sl}}_{m|n})
 \longrightarrow
 \mathscr U^{\mathrm{rtt},\mathrm{sl}}_{\mathbf d^+,\mathbf d^-}
\]
is an isomorphism.
\end{theorem}
\begin{proof}
Surjectivity is Proposition~\ref{prop:shifted-super-df-epi}.  Suppose that
\(0\ne u=\sum_{\lambda\in L}c_\lambda B_\lambda\) lies in the kernel, where
\(L\) is finite.  Apply \(\Xi_{\mathbf p,\mathbf r}\Gamma\).  By
Lemma~\ref{lem:three-factor-rtt-leading-coefficient}, the images of the
\(B_\lambda\) have pairwise distinct leading outer labels.  Choose the
largest one.  Its coefficient in the image of \(u\) is $c_\lambda\,\Gamma(K_{\mathbf u})$.
The element \(\Gamma(K_{\mathbf u})\) is invertible because every leading
Cartan coefficient is invertible in $\mathscr U^{\mathrm{rtt}}_{\mathbf d^+,\mathbf d^-}$.  Hence the
displayed coefficient is nonzero, a contradiction.  Thus the kernel is
zero.
\end{proof}

\begin{remark}
The proof above uses the RTT coproduct to attach two zero-shift factors and
then separates PBWD labels by ordinary PBWD coefficient extraction in those
factors.  Thus it does not invoke the difference-operator separation
statement appearing in \cite[Conjecture~3.75]{FPT2022}.  The ingredients
used here are the shifted PBWD basis on the Drinfeld side, the ordinary
zero-shift Ding--Frenkel isomorphism, and the two diagonal rescalings which
retain the two independent leading Cartan systems.  Apart from the usual
super PBWD signs and the exterior convention for odd root vectors, the
coefficient comparison itself does not use a special property of the
isotropic vertex.
\end{remark}

\subsection{Full general-linear currents and two-sided antidominant coproducts}
\label{subsec:full-gl-shifted-bridge}

The RTT matrix coproduct is defined on the full shifted RTT superalgebra, while the Ding--Frenkel map above identifies the shifted special-linear current superalgebra with the simple-current subalgebra.  We pass through the full general-linear current superalgebra to compare these levels without a separate all-mode tensor-closure calculation.

\begin{definition}[Shifted full general-linear current superalgebra]
\label{def:shifted-full-gl-current}
For weakly decreasing integer sequences
$\mathbf d^\pm=(d_a^\pm)_{a\in J}$, let
$\mathcal U^{\mathrm{gl}}_{q,\mathbf d^+,\mathbf d^-}$ be the
$\mathbb C(q)$-superalgebra generated by the coefficients of the root currents
$X_i^\pm(z)$, $i\in I$, and the even diagonal currents
\begin{align*}
K_a^+(z)=\sum_{r\ge -d_a^+}K_{a,r}^+z^{-r},\quad K_a^-(z)=\sum_{r\ge -d_a^-}K_{a,-r}^-z^r,\quad a\in J,
\end{align*}
with $K_{a,-d_a^+}^+$ and $K_{a,d_a^-}^-$ invertible.  The defining
relations are the displayed current relations of
Theorem~\ref{th: RTT and Drinfeld realization}, interpreted coefficientwise
with these shifted supports.  The zero-shift normalization
$K_{a,0}^+K_{a,0}^-=1$ of the zero-shift full general-linear RTT superalgebra is not imposed here; in
particular the two leading diagonal coefficients are independent and
invertible.
\end{definition}

Put $\kappa_a^+=K_{a,-d_a^+}^+,\quad \kappa_a^-=K_{a,d_a^-}^-$,
and normalize
\[
 \widehat K_a^+(z)=(\kappa_a^+)^{-1}z^{-d_a^+}K_a^+(z),
 \quad
 \widehat K_a^-(z)=(\kappa_a^-)^{-1}z^{d_a^-}K_a^-(z).
\]
Both normalized series have constant term $1$.  Since the diagonal
coefficients commute, define logarithmic modes by
\[
 \log\widehat K_a^+(z)=\sum_{r>0}H_{a,r}^{\mathrm{gl}}z^{-r},
 \quad
 \log\widehat K_a^-(z)=\sum_{r>0}H_{a,-r}^{\mathrm{gl}}z^{r}.
\]
The passage between the nonleading coefficients of $\widehat K_a^\pm$ and
the $H_{a,\pm r}^{\mathrm{gl}}$ is triangular and invertible.  We use ordered
monomials in these logarithmic modes, together with arbitrary Laurent powers
of the $\kappa_a^\pm$, for the diagonal part.

\begin{lemma}
\label{lem:shifted-full-gl-ordered-spanning}
Every element of $\mathcal U^{\mathrm{gl}}_{q,\mathbf d^+,\mathbf d^-}$
is a finite linear combination of ordered monomials
\begin{equation}\label{eq:full-gl-ordered-monomial}
 F_{\mathbf a^-}\,
 \prod_{a\in J}(\kappa_a^+)^{u_a^+}(\kappa_a^-)^{u_a^-}
 H_{\mathbf v}^{\mathrm{gl}}\,
 E_{\mathbf a^+},
\end{equation}
where $F_{\mathbf a^-}$ and $E_{\mathbf a^+}$ are ordered PBWD root
monomials in the negative and positive current halves of the zero-shift full current superalgebra and
$H_{\mathbf v}^{\mathrm{gl}}$ is an ordered monomial in the logarithmic
modes $H_{a,\pm r}^{\mathrm{gl}}$.
\end{lemma}
\begin{proof}
The root-current relations are unchanged by the diagonal shifts, so the two root halves have the same PBWD spanning sets as in the zero-shift superalgebra. The diagonal coefficients commute.  The Cartan--root relations move every diagonal coefficient between the negative and positive root factors, while the mixed root commutator replaces an out-of-order positive--negative pair by a linear combination involving only the diagonal ratios $K_{i+1}^\epsilon(z)(K_i^\epsilon(z))^{-1}$.  Iterating these relations gives \eqref{eq:full-gl-ordered-monomial}.  Only spanning is used here; linear independence will follow from the RTT coefficient argument below.
\end{proof}

\medskip\noindent\textbf{All-mode recovery of the higher Gauss coefficients.}
For arbitrary weakly decreasing shifts $\mathbf d^\pm$ and every
$1\leq j<k<i\leq M$, $r\in\mathbb Z$, the shifted RTT relations imply
\begin{align}
 e_{ji}^{(r)}
 &=\frac{1}{q_k-q_k^{-1}}
   [e_{ki}^{(0)},e_{jk}^{(r)}]_{q_k^{-1}},
 \label{eq:shifted-all-mode-e-recovery}\\
 f_{ij}^{(r)}
 &=-\frac{1}{q_k-q_k^{-1}}
   [f_{kj}^{(r)},f_{ik}^{(0)}]_{q_k}.
 \label{eq:shifted-all-mode-f-recovery}
\end{align}
These are the all-mode versions of
Proposition~\ref{prop: formulas for eji0,eji-1,fij0,fij1}.  To see directly
that the diagonal valuations play no role, for
\eqref{eq:shifted-all-mode-e-recovery} use the
$\mathsf E_{jk}\otimes\mathsf E_{ki}$ component of the RTT relation.  For
$r\geq0$ use the $++$ relation, while for $r<0$ use the $-+$ relation, and in
the second spectral variable extract the coefficient selecting
$e_{ki}^{(0)}$.  The calculation is made after Gauss factorization and before
expanding any diagonal series into modes.  Every diagonal factor which must
be removed is an invertible formal Laurent series, since it is a monomial in
the spectral parameter times a series with invertible constant term.  The
terms with an intermediate index $j'<j$ cancel by the same induction on the
left endpoint as in the proof of
Proposition~\ref{prop: formulas for eji0,eji-1,fij0,fij1}.  The remaining
half-current identity is
\[
 (q_k-q_k^{-1})e_{ji}^{\epsilon}(z)
   =[e_{ki}^{(0)},e_{jk}^{\epsilon}(z)]_{q_k^{-1}},
\]
with $\epsilon=+$ for the nonnegative modes and $\epsilon=-$ for the negative
modes.

Likewise, the $\mathsf E_{ik}\otimes\mathsf E_{kj}$ component, using the
$--$ relation for $r\leq0$ and the $+-$ relation for $r>0$, gives
\[
 (q_k-q_k^{-1})f_{ij}^{\epsilon}(z)
   =-[f_{kj}^{\epsilon}(z),f_{ik}^{(0)}]_{q_k}.
\]
The graded signs are exactly those already present in the supercommutators of
Proposition~\ref{prop: formulas for eji0,eji-1,fij0,fij1}; changing the
leading powers of the even diagonal series introduces no new sign.  Taking
coefficients proves \eqref{eq:shifted-all-mode-e-recovery}--
\eqref{eq:shifted-all-mode-f-recovery}.  Iterating the first identity with
$k=i-1$ and the second with $k=j+1$ expresses every higher Gauss coefficient
as an iterated graded $q$-commutator of simple Gauss coefficients.  These are
the super counterparts of the shifted type A formulas in
\cite[Lemma~3.46 and Corollary~3.48]{FPT2022}.

\begin{theorem}
\label{thm:shifted-full-gl-df-isomorphism}
Assume $m,n>1$ and $m\ne n$.  The assignments
\begin{align*}
 K_a^\pm(z)\longmapsto g_a^\pm(z),\quad
 X_i^+(z)\longmapsto e_{i,i+1}^+(z)-e_{i,i+1}^-(z),\quad
 X_i^-(z)\longmapsto f_{i+1,i}^+(z)-f_{i+1,i}^-(z)
\end{align*}
define an isomorphism
\begin{equation}\label{eq:shifted-full-gl-df-isomorphism}
 \widetilde\Gamma_{\mathbf d^+,\mathbf d^-}:
 \mathcal U^{\mathrm{gl}}_{q,\mathbf d^+,\mathbf d^-}
 \xrightarrow{\ \sim\ }
 \mathscr U^{\mathrm{rtt}}_{\mathbf d^+,\mathbf d^-}.
\end{equation}
In particular, the ordered monomials of
Lemma~\ref{lem:shifted-full-gl-ordered-spanning} form a basis.
\end{theorem}
\begin{proof}
The Gauss-component proof of Theorem~\ref{th: RTT and Drinfeld realization}
is coefficientwise and uses only the off-diagonal half-current supports and
the invertibility of the leading diagonal coefficients.  Replacing the
zero-degree diagonal series by the shifted series of
Definition~\ref{def:shifted-full-gl-current} therefore gives the same current
relations, so the displayed assignment is a homomorphism.  Surjectivity no
longer requires an appeal to the zero-shift generation theorem: by \eqref{eq:shifted-all-mode-e-recovery}--\eqref{eq:shifted-all-mode-f-recovery}, every higher upper or
lower Gauss coefficient is an iterated graded $q$-commutator of simple Gauss
coefficients, while every diagonal coefficient and the inverse of every
leading diagonal coefficient is the image of the corresponding coefficient
of $K_a^\pm(z)$.  Hence the image contains all generators of
$\mathscr U^{\mathrm{rtt}}_{\mathbf d^+,\mathbf d^-}$.

It remains to prove injectivity.  Attach zero-shift full RTT factors on the
left and right and use the three-factor matrix product exactly as in
Subsubsection~\ref{subsubsec:three-factor-rtt-separation}.  Here, independently of the simple-ratio parameters used there, let $\mathbf p=(p_a)_{a\in J}$ and $\mathbf r=(r_a)_{a\in J}$ be $M$-tuples of algebraically independent Laurent parameters.  On the left outer factor map the extended zero-shift RTT matrices to $D_{\mathbf p}T_{\mathrm{gl}}^\epsilon(z)$ and on the right to $T_{\mathrm{gl}}^\epsilon(z)D_{\mathbf r}$ in the ordinary normalized zero-shift general-linear RTT algebra, where
\[
 D_{\mathbf p}=\operatorname{diag}(p_1,\ldots,p_M),\qquad
 D_{\mathbf r}=\operatorname{diag}(r_1,\ldots,r_M).
\]
These substitutions preserve the RTT relations.  They define
homomorphisms from the extended outer factors because the latter impose no
relation between their positive and negative leading diagonal systems.  We
expand the two zero-shift targets in the standard general-linear PBW normal
form associated with the current presentation of
Theorem~\ref{th: RTT and Drinfeld realization}; equivalently, one may use
the full general-linear PBWD bases discussed in
\cite[Remark~5.12]{Tsy2021}.  The proof of
Lemma~\ref{lem:three-factor-rtt-leading-coefficient} then applies with the
diagonal indices $a\in J$ in place of the simple-ratio indices $i\in I$.
We record the label recovery, which is the only point that changes.

For an ordered monomial \eqref{eq:full-gl-ordered-monomial}, set
\[
 \delta_a=u_a^+-u_a^-,\qquad
 \sigma_a=u_a^++u_a^-.
\]
Let $\boldsymbol\eta^-$ and $\boldsymbol\eta^+$ be the Laurent-weight
vectors contributed by the already specified negative and positive root
PBWD words under the two diagonal rescalings.  Write $k_a=K_{a,0}^+$ in a zero-shift outer factor; there $K_{a,0}^-=k_a^{-1}$.  The exact outer-factor
argument of Lemma~\ref{lem:three-factor-rtt-leading-coefficient} gives the
unique extremal term
\begin{equation}\label{eq:full-gl-three-factor-leading-term}
 \mathbf p^{\boldsymbol\eta^-+\boldsymbol\sigma}
 \mathbf r^{\boldsymbol\eta^++\boldsymbol\sigma}
 \left(F_{\mathbf a^-}k^{\boldsymbol\delta}
 H_{\mathbf v}^{\mathrm{gl}}\right)
 \otimes
 \widetilde\Gamma\!\left(
   \prod_a(\kappa_a^+)^{u_a^+}(\kappa_a^-)^{u_a^-}\right)
 \otimes
 \left(k^{\boldsymbol\delta}E_{\mathbf a^+}\right),
\end{equation}
where $k^{\boldsymbol\delta}=\prod_a k_a^{\delta_a}$.  Thus the ordinary
outer PBW monomials recover the two root words, the logarithmic diagonal word
$H_{\mathbf v}^{\mathrm{gl}}$, and all powers $\delta_a$.  At the extremal root
heights the Borel quotient used in
Lemma~\ref{lem:three-factor-rtt-leading-coefficient} makes the diagonal factor
the product of the three diagonal factors; hence the logarithmic modes add.
We spell out why the last Cartan-degree step cannot collide with a different
root label.  Taking logarithms of the $K$--$X$ exchange relations in
Theorem~\ref{th: RTT and Drinfeld realization} shows, coefficientwise, that
for $r\ne0$
\[
 [H_{a,r}^{\mathrm{gl}},X_{i,s}^{\pm}]
   \in \mathbb C(q)\,X_{i,s+r}^{\pm}
\]
(with zero coefficient when $a$ is not adjacent to $i$).  Thus, when an outer
term is put in PBW order, commuting a logarithmic Cartan mode past a root mode
either keeps that Cartan mode in the ordered term or uses the commutator and
removes one logarithmic Cartan factor.  In particular normal ordering cannot
increase the left logarithmic Cartan degree.  The commutator term may change a
root loop index, but it has strictly smaller left Cartan degree and therefore
cannot compete after the four preceding root-degree entries of
\eqref{eq:outer-degree-order} have been fixed.

Consequently, among terms having the extremal outer root labels, the largest
possible left Cartan degree is $|\mathbf v|$.  It is attained only by choosing
the left summand from every occurrence in
$H_{\mathbf v}^{\mathrm{gl}}$ and by never using a Cartan--root commutator.
This produces exactly the ordered monomial $H_{\mathbf v}^{\mathrm{gl}}$
with coefficient one.  Since the negative and positive outer PBW words have
already fixed all root vectors and their individual loop-mode labels, no
change of a root label can compensate for a different logarithmic Cartan
word at this last stage.  The root words then determine
$\boldsymbol\eta^\pm$, and either Laurent exponent recovers every
$\sigma_a$.  Therefore
\[
 u_a^+=\frac{\sigma_a+\delta_a}{2},\qquad
 u_a^-=\frac{\sigma_a-\delta_a}{2}.
\]
Thus distinct ordered monomials have distinct extremal outer
PBW--Laurent labels.  The exact outer-factor argument in
Lemma~\ref{lem:three-factor-rtt-leading-coefficient} also shows that this
extremal coefficient occurs uniquely and with coefficient equal to an
invertible monomial in the middle leading diagonal coefficients.

Consequently no nontrivial linear combination of the spanning monomials in
Lemma~\ref{lem:shifted-full-gl-ordered-spanning} can vanish: applying the
three-factor map and taking the largest extremal outer label would leave a
nonzero coefficient.  Hence those monomials form a basis.  The same argument
applied to a nonzero element in the kernel of
$\widetilde\Gamma_{\mathbf d^+,\mathbf d^-}$ gives a contradiction, proving
injectivity.
\end{proof}

The special-linear current superalgebra sits naturally inside the full current superalgebra.  If
$d_{i+1}^\pm-d_i^\pm=l_i^\pm=(\alpha_i\mid\mu_\pm)$, define
\begin{equation}\label{eq:shifted-sl-into-full-gl}
 \jmath_{\mathbf d}:\mathcal U_q^{\mu_+,\mu_-}
 \longrightarrow \mathcal U^{\mathrm{gl}}_{q,\mathbf d^+,\mathbf d^-}
\end{equation}
by
\begin{align*}
 x_i^+(z) \longmapsto \frac{X_i^+(\tau_i z)}{q_i-q_i^{-1}}, \quad  
 x_i^-(z) \longmapsto \frac{X_i^-(\tau_i z)}{q_i-q_i^{-1}}, \quad 
 \phi_i^\pm(z) \longmapsto
 K_{i+1}^\pm(\tau_i z)(K_i^\pm(\tau_i z))^{-1}.
\end{align*}
The defining relations give a homomorphism, and
$\widetilde\Gamma_{\mathbf d^+,\mathbf d^-}\circ\jmath_{\mathbf d}
=\Gamma_{\mathbf d^+,\mathbf d^-}$.  Theorems~\ref{thm:shifted-super-df-isomorphism}
and~\ref{thm:shifted-full-gl-df-isomorphism} therefore show that
$\jmath_{\mathbf d}$ is injective.

\begin{proposition}
\label{prop:shifted-rtt-scalar-fixed-algebra}
Let $\mathscr G$ be the group of pairs of invertible scalar series
\[
 a^+(z)\in\mathbb C(q)^\times(1+z^{-1}\mathbb C(q)[[z^{-1}]]),\qquad
 a^-(z)\in\mathbb C(q)^\times(1+z\mathbb C(q)[[z]]).
\]
It acts on the full shifted RTT superalgebra by
$\mathsf T^\pm(z)\longmapsto a^\pm(z)\mathsf T^\pm(z)$.
Then
\begin{equation}\label{eq:shifted-rtt-fixed-algebra}
 \bigl(\mathscr U^{\mathrm{rtt}}_{\mathbf d^+,\mathbf d^-}\bigr)^{\mathscr G}
 =\mathscr U^{\mathrm{rtt},\mathrm{sl}}_{\mathbf d^+,\mathbf d^-}.
\end{equation}
\end{proposition}
\begin{proof}
Under the isomorphism
\eqref{eq:shifted-full-gl-df-isomorphism}, the action is
\[
K_a^\pm(z)\longmapsto a^\pm(z)K_a^\pm(z),\quad X_i^\pm(z)\longmapsto X_i^\pm(z).
\]
Hence the image of $\jmath_{\mathbf d}$ is fixed.  Conversely, use the basis
proved in Theorem~\ref{thm:shifted-full-gl-df-isomorphism}.  For either sign, the diagonal basis may equivalently be written in the
coordinates
\[
 \kappa_1^\pm,\quad
 \kappa_{a+1}^\pm(\kappa_a^\pm)^{-1}\ (a\in I),\quad
 H_{1,\pm r}^{\mathrm{gl}},\quad
 H_{a+1,\pm r}^{\mathrm{gl}}-H_{a,\pm r}^{\mathrm{gl}}
 \ (a\in I,r>0).
\]
This is an invertible triangular change of variables at every finite jet
order.  The constant term of the scalar series multiplies $\kappa_1^\pm$,
while the coefficients of $\log(a^\pm(z)/a_0^\pm)$ translate independently
the variables $H_{1,\pm r}^{\mathrm{gl}}$; all ratio variables are fixed.
Since every element is a finite linear combination of basis monomials,
invariance under all such scalings and translations forces it to be
independent of $\kappa_1^\pm$ and of every $H_{1,\pm r}^{\mathrm{gl}}$.  The fixed Cartan
algebra is precisely the algebra generated by the coefficients of the
consecutive ratios and their invertible leading coefficients.  The spectral
rescalings $z\mapsto\tau_i z$ do not change the coefficient algebra generated
by these ratios, and the same is true for the scalar normalizations of the
root currents in \eqref{eq:shifted-sl-into-full-gl}.  Since all root currents
are fixed, the full fixed algebra is
$\jmath_{\mathbf d}(\mathcal U_q^{\mu_+,\mu_-})$.
Transporting this equality through the two Ding--Frenkel isomorphisms gives
\eqref{eq:shifted-rtt-fixed-algebra}.
\end{proof}

\begin{proposition}
\label{prop:two-sided-antidominant-rtt-closure}
For $a=1,2$, let $\mathbf d_a^\pm$ be weakly decreasing.  Then
\begin{align}\label{eq:two-sided-sl-tensor-closure}
 \Delta^{\mathrm{rtt}}_{\mathbf d_1,\mathbf d_2}
 \bigl(\mathscr U^{\mathrm{rtt},\mathrm{sl}}_{
 \mathbf d_1^++\mathbf d_2^+,\,\mathbf d_1^-+\mathbf d_2^-}\bigr)
 \subseteq
 \mathscr U^{\mathrm{rtt},\mathrm{sl}}_{\mathbf d_1^+,\mathbf d_1^-}
 \otimes
 \mathscr U^{\mathrm{rtt},\mathrm{sl}}_{\mathbf d_2^+,\mathbf d_2^-}.
\end{align}
\end{proposition}
\begin{proof}
Let $\theta_a$ be a scalar rescaling of the $a$-th target RTT matrix.  Matrix
multiplication gives, coefficientwise,
\[
 (\theta_1\otimes\mathrm{id})\Delta^{\mathrm{rtt}}
 =\Delta^{\mathrm{rtt}}\theta_1^{\mathrm{src}},\qquad
 (\mathrm{id}\otimes\theta_2)\Delta^{\mathrm{rtt}}
 =\Delta^{\mathrm{rtt}}\theta_2^{\mathrm{src}},
\]
where on the source the same scalar series multiplies the full source matrix.
Every element of the source simple-current subalgebra is fixed by either
source rescaling.  Its coproduct is therefore fixed independently by the
scalar actions on the first and second target factors.

Over the ground field, if a group acts only on the first factor of
$A\otimes B$, then
$(A\otimes B)^G=A^G\otimes B$: write an element using linearly independent
coefficients in the second factor.  Applying this twice and using
Proposition~\ref{prop:shifted-rtt-scalar-fixed-algebra} gives
\[
 \bigl(\mathscr U^{\mathrm{rtt}}_{\mathbf d_1^+,\mathbf d_1^-}
 \otimes
 \mathscr U^{\mathrm{rtt}}_{\mathbf d_2^+,\mathbf d_2^-}\bigr)^{
 \mathscr G_1\times\mathscr G_2}
 =
 \mathscr U^{\mathrm{rtt},\mathrm{sl}}_{\mathbf d_1^+,\mathbf d_1^-}
 \otimes
 \mathscr U^{\mathrm{rtt},\mathrm{sl}}_{\mathbf d_2^+,\mathbf d_2^-},
\]
which proves \eqref{eq:two-sided-sl-tensor-closure}.
\end{proof}

\begin{theorem}
\label{thm:two-sided-antidominant-current-coproduct}
Assume $m,n>1$ and $m\ne n$.  Let
$(\mu_{a,+},\mu_{a,-})$, $a=1,2$, be two pairs of antidominant shifts and
choose weakly decreasing lifts $\mathbf d_a^\pm$.  There is a unique
superalgebra homomorphism
\begin{align}\label{eq:two-sided-antidominant-current-coproduct}
\Delta_{(\mu_{1,+},\mu_{1,-}),(\mu_{2,+},\mu_{2,-})}:\mathcal U_q^{\mu_{1,+}+\mu_{2,+},\,\mu_{1,-}+\mu_{2,-}}\longrightarrow\mathcal U_q^{\mu_{1,+},\mu_{1,-}}\otimes\mathcal U_q^{\mu_{2,+},\mu_{2,-}}
\end{align}
whose image under the three Ding--Frenkel isomorphisms is the RTT matrix
coproduct.  These homomorphisms are coassociative for three pairs of
antidominant shifts.  When $\mu_{a,+}=0$, they give the one-sided
antidominant coproducts used below.
\end{theorem}
\begin{proof}
Proposition~\ref{prop:two-sided-antidominant-rtt-closure} permits us to restrict
the RTT matrix coproduct to the simple-current subalgebras.  Transporting
this restriction through Theorem~\ref{thm:shifted-super-df-isomorphism}
gives \eqref{eq:two-sided-antidominant-current-coproduct}; uniqueness follows
from the same isomorphisms.  The construction is independent of the chosen
lifts.  Indeed, for each target factor one may change the positive and
negative lifts independently by common integers,
$\mathbf d_a^+\mapsto\mathbf d_a^++c_a^+(1,\ldots,1)$ and
$\mathbf d_a^-\mapsto\mathbf d_a^-+c_a^-(1,\ldots,1)$.
By Lemma~\ref{lem:shifted-super-rtt-lift}, this multiplies the corresponding
full RTT half-matrices by the scalar monomials $z^{c_a^+}$ and $z^{c_a^-}$,
respectively.  Matrix multiplication changes the source lifts by
$c_1^++c_2^+$ and $c_1^-+c_2^-$ in the two signs, while all simple Gauss
currents and their consecutive diagonal ratios are unchanged.  Coassociativity is the
restriction of Corollary~\ref{cor:shifted-super-rtt-coassociativity}.  The
one-sided specialization $\mu_{a,+}=0$ is therefore coassociative for every
antidominant triple, including boundary points of the antidominant chamber.
\end{proof}

We compare this construction with the ordinary type A case.
\begin{remark}
Theorem~\ref{thm:two-sided-antidominant-current-coproduct} is the super counterpart, for the standard parity sequence, of the antidominant special-linear coproduct in \cite[Proposition~3.113]{FPT2022}; its RTT coassociativity corresponds to \cite[Corollary~3.108]{FPT2022}.  The proof here uses the full general-linear current presentation to establish tensor closure and the three-factor RTT coefficient argument for the shifted Ding--Frenkel isomorphisms; it does not use the difference-operator separation conjecture appearing in \cite[Conjecture~3.75]{FPT2022}.
\end{remark}

The distinction between antidominant coassociativity and arbitrary-shift
extensions is discussed in Remark~\ref{rem:noncompleted-coassoc-comparison}
below.

\subsection{One-sided coproducts and extension to arbitrary shifts}

For antidominant $\mu_1,\mu_2$, the one-sided coproduct is the specialization
$\mu_{1,+}=\mu_{2,+}=0$ of
Theorem~\ref{thm:two-sided-antidominant-current-coproduct}.  In the strictly
antidominant case the Levendorskii-type finite presentation makes this
coproduct completely explicit on a finite set of generators.

\begin{proposition}
\label{prop:explicit-strict-antidominant-one-sided-coproduct}
For type A with the standard parity sequence, assume $m,n>1$ and $m\ne n$, and let
$\mu_{1},\mu_{2}\in\Lambda^{-}$ be strictly antidominant.  Set
\[
l_{a,i}:=\alpha_i(\mu_a)\quad(a=1,2),\qquad l_i:=l_{1,i}+l_{2,i}.
\]
Let
\[
\Delta_{\mu_{1},\mu_{2}}:
\mathcal{U}_{q}^{0,\mu_{1}+\mu_{2}}(\widehat{\mathfrak{sl}}_{m|n})
\longrightarrow
\mathcal{U}_{q}^{0,\mu_{1}}(\widehat{\mathfrak{sl}}_{m|n})
\otimes
\mathcal{U}_{q}^{0,\mu_{2}}(\widehat{\mathfrak{sl}}_{m|n})
\]
be the one-sided specialization of Theorem~\ref{thm:two-sided-antidominant-current-coproduct}.  On the finite generating set of Theorem~\ref{thm:finite iso},
\begin{gather*}
\Delta_{\mu_{1},\mu_{2}}(x_{i,r}^{+})=1\otimes x_{i,r}^{+},\quad
\Delta_{\mu_{1},\mu_{2}}(x_{i,s}^{-})=x_{i,s}^{-}\otimes1
\quad\text{for $l_{2,i}\leq r<0,\ l_{1,i}<s\leq0$},\\
\Delta_{\mu_{1},\mu_{2}}(x_{i,0}^{+})
=1\otimes x_{i,0}^{+}+x_{i,0}^{+}\otimes\phi_{i,0}^{+},\quad
\Delta_{\mu_{1},\mu_{2}}(x_{i,l_{1,i}}^{-})
=x_{i,l_{1,i}}^{-}\otimes1+\phi_{i,l_{1,i}}^{-}\otimes x_{i,0}^{-},\\
\Delta_{\mu_{1},\mu_{2}}(\phi_{i,0}^{+})
=\phi_{i,0}^{+}\otimes\phi_{i,0}^{+},\quad
\Delta_{\mu_{1},\mu_{2}}(\phi_{i,l_i}^{-})
=\phi_{i,l_{1,i}}^{-}\otimes\phi_{i,l_{2,i}}^{-}.
\end{gather*}
The images of the remaining boundary generators
$x_{i,l_{2,i}-1}^{+}$, $x_{i,1}^{-}$, $h_{i,\pm1}$
are given explicitly in Proposition~\ref{prop:explicit-antidominant-coproduct-formulas} of Appendix~\ref{app:antidominant-coproduct-formulas}.  Together, these formulas determine the coproduct on the finite presentation.
\end{proposition}
\begin{proof}
The coproduct already exists by Theorem~\ref{thm:two-sided-antidominant-current-coproduct}.  The displayed formulas and the remaining boundary formulas are obtained by transporting the shifted RTT matrix coproduct through the shifted Ding--Frenkel isomorphisms; see Appendix~\ref{app:antidominant-coproduct-formulas}.  The finite presentation of Theorem~\ref{thm:finite iso} then shows that these values determine the coproduct.
\end{proof}

\begin{theorem}\label{Th: commutative diagram}
For type A with the standard parity sequence, assume $m,n>1$ and $m\ne n$.  For any $\mu_{1},\mu_{2}\in\Lambda$ and $\mu=\mu_{1}+\mu_{2}$, there is a unique $\mathbb{C}(q)$-superalgebra homomorphism
\begin{align*}
\Delta_{\mu_{1},\mu_{2}}:\mathcal{U}^{0,\mu}_{q}(\widehat{\mathfrak{sl}}_{m|n})\rightarrow\mathcal{U}^{0,\mu_{1}}_{q}(\widehat{\mathfrak{sl}}_{m|n})\otimes \mathcal{U}^{0,\mu_{2}}_{q}(\widehat{\mathfrak{sl}}_{m|n}),
\end{align*}
such that for any $\nu_{1},\nu_{2}\in\Lambda^{-}$,
\begin{equation}\label{eq:shift-coproduct-compatibility}
(\iota_{\mu_{1},0,\nu_{1}}\otimes\iota_{\mu_{2},\nu_{2},0})\circ\Delta_{\mu_{1},\mu_{2}}
=
\Delta_{\mu_{1}+\nu_{1},\mu_{2}+\nu_{2}}\circ\iota_{\mu,\nu_{2},\nu_{1}}.
\end{equation}
\end{theorem}
\begin{proof}
First, we construct $\Delta_{\mu_{1},\mu_{2}}$. Fix $\nu_{1},\nu_{2}\in\Lambda^{-}$ such that $\alpha_i(\mu_1+\nu_1),\alpha_i(\mu_2+\nu_2)<0$ for each $i\in I$.  The map $\Delta_{\mu_1+\nu_1,\mu_2+\nu_2}$ is the one-sided specialization of Theorem~\ref{thm:two-sided-antidominant-current-coproduct}; its finite-generator formulas are given in Proposition~\ref{prop:explicit-strict-antidominant-one-sided-coproduct}.  Since $\iota_{\mu,\nu_2,\nu_1}$ and $\iota_{\mu_1,0,\nu_1}\otimes\iota_{\mu_2,\nu_2,0}$ are injective, it is enough to prove that
\begin{align}\label{eq:give the above homo in the diagram}
\Delta_{\mu_{1}+\nu_{1},\mu_{2}+\nu_{2}}\iota_{\mu,\nu_{2},\nu_{1}}(\mathcal{U}^{0,\mu}_{q}(\widehat{\mathfrak{sl}}_{m|n}))\subseteq(\iota_{\mu_{1},0,\nu_{1}}\otimes\iota_{\mu_{2},\nu_{2},0})(\mathcal{U}^{0,\mu_{1}}_{q}(\widehat{\mathfrak{sl}}_{m|n})\otimes\mathcal{U}^{0,\mu_{2}}_{q}(\widehat{\mathfrak{sl}}_{m|n})),
\end{align}
then $\Delta_{\mu_{1}+\nu_{1},\mu_{2}+\nu_{2}}$ uniquely determines a homomorphism $\Delta_{\mu_{1},\mu_{2}}$ satisfying \eqref{eq:shift-coproduct-compatibility} for this pair $\nu_1,\nu_2$.

For $\mu\in\Lambda$, write $l_i:=\alpha_i(\mu)$.  Let $\mathcal{U}_{q}^{0,\mu}(\widehat{\mathfrak{sl}}_{m|n})^{>}$ and $\mathcal{U}_{q}^{0,\mu}(\widehat{\mathfrak{sl}}_{m|n})^{\geq}$ (resp. $\mathcal{U}_{q}^{0,\mu}(\widehat{\mathfrak{sl}}_{m|n})^{<}$ and $\mathcal{U}_{q}^{0,\mu}(\widehat{\mathfrak{sl}}_{m|n})^{\leq}$) be the $\mathbb{C}(q)$-subalgebras of $\mathcal{U}_{q}^{0,\mu}(\widehat{\mathfrak{sl}}_{m|n})$ generated by $\{x_{i,r}^{+}\mid i\in I, r\in\mathbb{Z}\}$ and
\[
\{x_{i,r}^{+},(\phi_{i,0}^{+})^{\pm1},(\phi_{i,l_i}^{-})^{\pm1},\phi_{i,s}^{+},\phi_{i,t}^{-}\mid i\in I,\ r\in\mathbb Z,\ s>0,\ t<l_i\}
\]
(resp. $\{x_{i,r}^{-}\mid i\in I, r\in\mathbb{Z}\}$ and the analogous set with $x_{i,r}^{+}$ replaced by $x_{i,r}^{-}$).

We first note the two inclusions
\begin{align}
\mathcal U_q^{0,\eta+\nu}(\widehat{\mathfrak{sl}}_{m|n})^{\geq}
\subseteq
\iota_{\eta,0,\nu}\bigl(\mathcal U_q^{0,\eta}(\widehat{\mathfrak{sl}}_{m|n})\bigr), \quad 
\mathcal U_q^{0,\eta+\nu}(\widehat{\mathfrak{sl}}_{m|n})^{\leq}
\subseteq
\iota_{\eta,\nu,0}\bigl(\mathcal U_q^{0,\eta}(\widehat{\mathfrak{sl}}_{m|n})\bigr)
\end{align}
for every $\eta\in\Lambda$ and $\nu\in\Lambda^{-}$.  Indeed, $\iota_{\eta,0,\nu}$ fixes every positive root current and multiplies each Cartan current by the finite polynomial $(1-z^{-1})^{-\alpha_i(\nu)}$.  On the ordered coefficient strings determined by the shifted supports this is triangular with nonzero diagonal; starting from the invertible leading coefficients, one recursively recovers every Cartan coefficient in the target.  The second inclusion is identical with the negative root currents fixed.

Use the finite-generator formulas of Proposition~\ref{prop:explicit-strict-antidominant-one-sided-coproduct} for $x_{i,-1}^{+}$, $x_{i,1}^{-}$, and $h_{i,\pm1}$.  Together with the relations
\begin{gather*}
[h_{i,\pm1},x_{i,r}^{+}]=[s_{i}A_{i,i}]_{q_{i}}x_{i,r\pm1}^{+},\quad[h_{i,\pm1},x_{i,r}^{-}]=-[s_{i}A_{i,i}]_{q_{i}}x_{i,r\pm1}^{-}\quad\text{for $i\neq m$},\\
[h_{i-1,\pm1},x_{i,r}^{+}]=[s_{i-1}A_{i-1,i}]_{q_{i-1}}x_{i,r\pm1}^{+},\quad[h_{i-1,\pm1},x_{i,r}^{-}]=-[s_{i-1}A_{i-1,i}]_{q_{i-1}}x_{i,r\pm1}^{-}\quad\text{for $i=m$},
\end{gather*}
together with the two inclusions just proved, namely
\begin{align}\label{eq:geq otimes leq subset iota mu1 otimes mu2}
\mathcal{U}_{q}^{0,\mu_{1}+\nu_{1}}(\widehat{\mathfrak{sl}}_{m|n})^{\geq}\otimes\mathcal{U}_{q}^{0,\mu_{2}+\nu_{2}}(\widehat{\mathfrak{sl}}_{m|n})^{\leq}\subseteq(\iota_{\mu_{1},0,\nu_{1}}\otimes\iota_{\mu_{2},\nu_{2},0})(\mathcal{U}_{q}^{0,\mu_{1}}(\widehat{\mathfrak{sl}}_{m|n})\otimes\mathcal{U}_{q}^{0,\mu_{2}}(\widehat{\mathfrak{sl}}_{m|n})),
\end{align}
we obtain
\begin{align}\label{eq:Delta xir subset}
\begin{split}
&\Delta_{\mu_{1}+\nu_{1},\mu_{2}+\nu_{2}}(x_{i,r}^{+})\in 1\otimes x_{i,r}^{+}+\mathcal{U}_{q}^{0,\mu_{1}+\nu_{1}}(\widehat{\mathfrak{sl}}_{m|n})^{>}\otimes\mathcal{U}_{q}^{0,\mu_{2}+\nu_{2}}(\widehat{\mathfrak{sl}}_{m|n})^{\leq},\\&
\Delta_{\mu_{1}+\nu_{1},\mu_{2}+\nu_{2}}(x_{i,r}^{-})\in x_{i,r}^{-}\otimes1+\mathcal{U}_{q}^{0,\mu_{1}+\nu_{1}}(\widehat{\mathfrak{sl}}_{m|n})^{\geq}\otimes\mathcal{U}_{q}^{0,\mu_{2}+\nu_{2}}(\widehat{\mathfrak{sl}}_{m|n})^{<},
\end{split}
\end{align}
for all $r\in\mathbb{Z}$. 
By \eqref{eq:geq otimes leq subset iota mu1 otimes mu2}--\eqref{eq:Delta xir subset} and $\iota_{\mu,\nu_{2},\nu_{1}}(x_{i,r}^{+})=\sum_{s=0}^{-\alpha_{i}(\nu_{2})}(-1)^{s}\binom{-\alpha_{i}(\nu_{2})}{s}x_{i,r-s}^{+}$, we have
\begin{align*}
\Delta_{\mu_{1}+\nu_{1},\mu_{2}+\nu_{2}}(\iota_{\mu,\nu_{2},\nu_{1}}(x_{i,r}^{+}))&\in1\otimes\sum_{s=0}^{-\alpha_{i}(\nu_{2})}(-1)^{s}\binom{-\alpha_{i}(\nu_{2})}{s}x_{i,r-s}^{+}\\&+\mathcal{U}_{q}^{0,\mu_{1}+\nu_{1}}(\widehat{\mathfrak{sl}}_{m|n})^{>}\otimes\mathcal{U}_{q}^{0,\mu_{2}+\nu_{2}}(\widehat{\mathfrak{sl}}_{m|n})^{\leq},
\end{align*}
where the first term on the right-hand side is equal to $(\iota_{\mu_{1},0,\nu_{1}}\otimes\iota_{\mu_{2},\nu_{2},0})(1\otimes x_{i,r}^{+})$, and the second summand on the right-hand side is contained in $(\iota_{\mu_{1},0,\nu_{1}}\otimes\iota_{\mu_{2},\nu_{2},0})(\mathcal{U}_{q}^{0,\mu_{1}}(\widehat{\mathfrak{sl}}_{m|n})\otimes\mathcal{U}_{q}^{0,\mu_{2}}(\widehat{\mathfrak{sl}}_{m|n}))$. Similarly, we have
\begin{align*}
\Delta_{\mu_{1}+\nu_{1},\mu_{2}+\nu_{2}}(\iota_{\mu,\nu_{2},\nu_{1}}(x_{i,r}^{-}))&\in \sum_{s=0}^{-\alpha_{i}(\nu_{1})}(-1)^{s}\binom{-\alpha_{i}(\nu_{1})}{s}x_{i,r-s}^{-}\otimes1\\&+\mathcal{U}_{q}^{0,\mu_{1}+\nu_{1}}(\widehat{\mathfrak{sl}}_{m|n})^{\geq}\otimes\mathcal{U}_{q}^{0,\mu_{2}+\nu_{2}}(\widehat{\mathfrak{sl}}_{m|n})^{<},
\end{align*}
where the first term on the right-hand side is equal to $(\iota_{\mu_{1},0,\nu_{1}}\otimes\iota_{\mu_{2},\nu_{2},0})(x_{i,r}^{-}\otimes1)$, and the second summand on the right-hand side is contained in $(\iota_{\mu_{1},0,\nu_{1}}\otimes\iota_{\mu_{2},\nu_{2},0})(\mathcal{U}_{q}^{0,\mu_{1}}(\widehat{\mathfrak{sl}}_{m|n})\otimes\mathcal{U}_{q}^{0,\mu_{2}}(\widehat{\mathfrak{sl}}_{m|n}))$. Moreover, we have
\begin{gather*}
\Delta_{\mu_{1}+\nu_{1},\mu_{2}+\nu_{2}}(\iota_{\mu,\nu_{2},\nu_{1}}((\phi_{i,0}^{+})^{\pm1}))=(\iota_{\mu_{1},0,\nu_{1}}\otimes\iota_{\mu_{2},\nu_{2},0})((\phi_{i,0}^{+})^{\pm1}\otimes(\phi_{i,0}^{+})^{\pm1}),\\
\Delta_{\mu_{1}+\nu_{1},\mu_{2}+\nu_{2}}(\iota_{\mu,\nu_{2},\nu_{1}}((\phi_{i,l_{i}}^{-})^{\pm1}))=(\iota_{\mu_{1},0,\nu_{1}}\otimes\iota_{\mu_{2},\nu_{2},0})((\phi_{i,l_{1,i}}^{-})^{\pm1}\otimes(\phi_{i,l_{2,i}}^{-})^{\pm1}).
\end{gather*}
Combining the relations $\phi_{i,r}^{+}=(q_{i}-q_{i}^{-1})[x_{i,r}^{+},x_{i,0}^{-}]$ and $\phi_{i,L_{i}-r}^{-}=-(q_{i}-q_{i}^{-1})[x_{i,L_{i}-r}^{+},x_{i,0}^{-}]$ in $\mathcal{U}_{q}^{0,\mu+\nu_{1}+\nu_{2}}(\widehat{\mathfrak{sl}}_{m|n})$ for $r\in\mathbb{Z}_{>0}$, $L_{i}=\alpha_{i}(\mu+\nu_{1}+\nu_{2})$, and $l_{i}-L_{i}=-\alpha_{i}(\nu_{1}+\nu_{2})$ with \eqref{eq:geq otimes leq subset iota mu1 otimes mu2}--\eqref{eq:Delta xir subset}, we obtain
\begin{align*}
&\Delta_{\mu_{1}+\nu_{1},\mu_{2}+\nu_{2}}(\iota_{\mu,\nu_{2},\nu_{1}}(\phi_{i,r}^{+})),\Delta_{\mu_{1}+\nu_{1},\mu_{2}+\nu_{2}}(\iota_{\mu,\nu_{2},\nu_{1}}(\phi_{i,l_{i}-r}^{-}))\\&\in\mathcal{U}_{q}^{0,\mu_{1}+\nu_{1}}(\widehat{\mathfrak{sl}}_{m|n})^{\geq}\otimes\mathcal{U}_{q}^{0,\mu_{2}+\nu_{2}}(\widehat{\mathfrak{sl}}_{m|n})^{\leq}\subseteq(\iota_{\mu_{1},0,\nu_{1}}\otimes\iota_{\mu_{2},\nu_{2},0})(\mathcal{U}_{q}^{0,\mu_{1}}(\widehat{\mathfrak{sl}}_{m|n})\otimes\mathcal{U}_{q}^{0,\mu_{2}}(\widehat{\mathfrak{sl}}_{m|n})),
\end{align*}
where $\iota_{\mu,\nu_{2},\nu_{1}}(\phi_{i,l_{i}-r}^{-})=\sum_{t=0}^{N_{i}}(-1)^{t}\binom{N_{i}}{t}\phi_{i,l_{i}-r-t}^{-}$ and we set $N_{i}:=l_{i}-L_{i}$. Therefore, \eqref{eq:give the above homo in the diagram} holds, and hence we obtain the homomorphism $\Delta_{\mu_{1},\mu_{2}}$ for the fixed $\nu_{1},\nu_{2}$.

Let $\nu_{1}',\nu_{2}'\in\Lambda^{-}$ be such that $\alpha_i(\mu_1+\nu_1'),\alpha_i(\mu_2+\nu_2')<0$ for each $i\in I$.  The composition law of Theorem~\ref{th:iota mu nu1 nu2 is injective} gives
\begin{align*}
\iota_{\mu+\nu_1+\nu_2,\nu_2',\nu_1'}\circ\iota_{\mu,\nu_2,\nu_1}
&=\iota_{\mu,\nu_2+\nu_2',\nu_1+\nu_1'},\\
(\iota_{\mu_1+\nu_1,0,\nu_1'}\otimes\iota_{\mu_2+\nu_2,\nu_2',0})\circ(\iota_{\mu_1,0,\nu_1}\otimes\iota_{\mu_2,\nu_2,0})
&=\iota_{\mu_1,0,\nu_1+\nu_1'}\otimes\iota_{\mu_2,\nu_2+\nu_2',0}.
\end{align*}
The finite-generator formulas of Proposition~\ref{prop:explicit-strict-antidominant-one-sided-coproduct} also give
\begin{align*}(\iota_{\mu_1+\nu_1,0,\nu_1'}\otimes\iota_{\mu_2+\nu_2,\nu_2',0})\circ\Delta_{\mu_1+\nu_1,\mu_2+\nu_2}=\Delta_{\mu_1+\nu_1+\nu_1',\mu_2+\nu_2+\nu_2'}\circ\iota_{\mu+\nu_1+\nu_2,\nu_2',\nu_1'}.
\end{align*}
Composing these identities with \eqref{eq:shift-coproduct-compatibility} shows that the maps $\Delta_{\mu_1,\mu_2}$ obtained from $(\nu_1,\nu_2)$ and from $(\nu_1+\nu_1',\nu_2+\nu_2')$ coincide.  The same argument starting from $(\nu_1',\nu_2')$ proves independence of the choice of the auxiliary antidominant shifts.

Now let $\nu_1,\nu_2\in\Lambda^{-}$ be arbitrary. Choose $\nu_1',\nu_2'\in\Lambda^{-}$ such that $\alpha_i(\mu_1+\nu_1+\nu_1'),\alpha_i(\mu_2+\nu_2+\nu_2')<0$ for every $i\in I$.  By the construction above, compatibility holds for the combined shifts $\nu_1+\nu_1',\nu_2+\nu_2'$ and, applied to the pair $\mu_1+\nu_1,\mu_2+\nu_2$, also for $\nu_1',\nu_2'$.  Using the composition law for the shift homomorphisms, these two identities give
\begin{align*}
(\iota_{\mu_1+\nu_1,0,\nu_1'}\otimes\iota_{\mu_2+\nu_2,\nu_2',0})
(\iota_{\mu_1,0,\nu_1}\otimes\iota_{\mu_2,\nu_2,0})\Delta_{\mu_1,\mu_2}=(\iota_{\mu_1+\nu_1,0,\nu_1'}\otimes\iota_{\mu_2+\nu_2,\nu_2',0})
\Delta_{\mu_1+\nu_1,\mu_2+\nu_2}\iota_{\mu,\nu_2,\nu_1}.
\end{align*}
The common tensor-product factor on the two sides is injective by Theorem~\ref{th:iota mu nu1 nu2 is injective}; cancelling it gives \eqref{eq:shift-coproduct-compatibility} for the original $\nu_1,\nu_2$.  This proves the theorem.
\end{proof}

The arbitrary-shift family retains the following restricted ternary compatibility.
\begin{theorem}
\label{thm:boundary-rtt-current-coassociativity}
For every antidominant triple $\lambda_1,\lambda_2,\lambda_3\in\Lambda^{-}$,
\[
(\Delta_{\lambda_1,\lambda_2}\otimes\id)\circ
\Delta_{\lambda_1+\lambda_2,\lambda_3}
=
(\id\otimes\Delta_{\lambda_2,\lambda_3})\circ
\Delta_{\lambda_1,\lambda_2+\lambda_3}.
\]
More generally, if $\mu_2\in\Lambda^{-}$ and $\mu_1,\mu_3\in\Lambda$ are
arbitrary, then
\[
(\Delta_{\mu_1,\mu_2}\otimes\id)\circ
\Delta_{\mu_1+\mu_2,\mu_3}
=
(\id\otimes\Delta_{\mu_2,\mu_3})\circ
\Delta_{\mu_1,\mu_2+\mu_3}.
\]
Thus the middle shift may lie on the boundary of the antidominant chamber,
where some $\alpha_i(\mu_2)=0$.
\end{theorem}
\begin{proof}
For an antidominant triple, every coproduct in the first identity is the
one-sided specialization of the coassociative two-sided family in
Theorem~\ref{thm:two-sided-antidominant-current-coproduct}; hence the identity
follows immediately.

For arbitrary outer shifts, choose $\nu_1,\nu_3\in\Lambda^{-}$ such that
$\mu_1+\nu_1,\mu_3+\nu_3\in\Lambda^{-}$.  Apply twice the compatibility relation~\eqref{eq:shift-coproduct-compatibility}
of Theorem~\ref{Th: commutative diagram}.  After the injective map
$\iota_{\mu_1,0,\nu_1}\otimes\id\otimes
\iota_{\mu_3,\nu_3,0}$,
the two ternary composites become the corresponding composites for the
antidominant triple
$\mu_1+\nu_1,\mu_2,\mu_3+\nu_3$, which are equal by the first part.
Injectivity of the three-factor embedding, by
Theorem~\ref{th:iota mu nu1 nu2 is injective}, gives the result.
\end{proof}

\subsection{Tensor products across shifts}
The arbitrary-shift homomorphism gives a binary representation-theoretic operation.

\begin{corollary}\label{cor:tensor-products-across-shifts}
For type A with the standard parity sequence, assume $m,n>1$ and $m\ne n$.  Let $\mu_{1},\mu_{2}\in\Lambda$, and let $V_{1}$ and $V_{2}$ be left modules over $\mathcal{U}^{0,\mu_{1}}_{q}(\widehat{\mathfrak{sl}}_{m|n})$ and $\mathcal{U}^{0,\mu_{2}}_{q}(\widehat{\mathfrak{sl}}_{m|n})$, respectively.  Then $V_{1}\otimes V_{2}$ is naturally a left $\mathcal{U}^{0,\mu_{1}+\mu_{2}}_{q}(\widehat{\mathfrak{sl}}_{m|n})$-module through $\Delta_{\mu_{1},\mu_{2}}$.  If $\mu_{1}=\mu_{2}=0$ and both modules factor through the central reduction of Proposition~\ref{prop:central-leading-reduction}, this construction descends to the usual tensor-product module for $\mathcal U_q(\widehat{\mathfrak{sl}}_{m|n})$.
\end{corollary}
\begin{proof}
Compose $\Delta_{\mu_{1},\mu_{2}}$ with the tensor product of the two module actions.  At zero shift the central elements $c_i$ are group-like, so the ideal generated by $c_i-1$ is a coideal and the action descends to the central reduction $\mathcal U_q(\widehat{\mathfrak{sl}}_{m|n})$.
\end{proof}

We close by comparing the two coassociativity statements with the ordinary type A case.
\begin{remark}\label{rem:noncompleted-coassoc-comparison}
Proposition~\ref{prop:formal-shifted-coassociativity} proves coassociativity for the \emph{formal} shifted Drinfeld coproduct in the completed tensor product.  This has an exact non-super precedent in \cite[Lemma~10.2]{FT2019}, where the formal coproduct is defined for arbitrary pairs of two-sided shifts.

The situation for non-completed coproducts is more subtle, but the distinction is \emph{not} between an infinite current presentation and a finite presentation.  For dominant coweights $\mu_a^\pm$, $a=1,2$, the RTT coproducts of \cite[Proposition~3.106]{FPT2022} give, after passage to the special-linear shifted algebras, coproducts with the four antidominant shifts
\[
(-\mu_1^+,-\mu_1^-; -\mu_2^+,-\mu_2^-),
\]
and Corollary~3.108 of \cite{FPT2022} shows that these form a coassociative family.  On the other hand, \cite[Remark~3.115]{FPT2022}, using the arbitrary-shift extension mechanism of \cite[Theorem~10.20]{FT2019}, considers
\[
\Delta_{\nu_1^+,\nu_1^-,\nu_2^+,\nu_2^-}:
U^{\mathrm{sc}}_{\nu_1^++\nu_2^+,\,\nu_1^-+\nu_2^-}
\longrightarrow
U^{\mathrm{sc}}_{\nu_1^+,\nu_1^-}\otimes
U^{\mathrm{sc}}_{\nu_2^+,\nu_2^-}
\]
for arbitrary coweights $\nu_a^\pm$, and explicitly notes that this general family is not coassociative.  Thus the loss of coassociativity occurs in the arbitrary-shift extension, not merely because the coproduct is non-completed or is specified on a finite generating set.  In rank one, \cite[Lemma~10.12]{FT2019} also records a restricted ternary compatibility when the middle one-sided shift is antidominant.

For pairs of antidominant positive and negative shifts, Theorem~\ref{thm:two-sided-antidominant-current-coproduct} gives the corresponding coassociative super family.  The maps $\Delta_{\mu_1,\mu_2}$ of Theorem~\ref{Th: commutative diagram} are its one-sided specialization on the antidominant chamber and extend to arbitrary one-sided shifts by the shift embeddings.  Theorem~\ref{thm:boundary-rtt-current-coassociativity} proves the resulting restricted ternary compatibility whenever the middle one-sided shift is antidominant, including boundary cases where some $\alpha_i(\mu_2)=0$.  We make no unrestricted monoidal or arbitrary four-parameter coassociativity claim after extending both positive and negative shifts beyond the antidominant chamber.
\end{remark}

\section*{Acknowledgements}
The authors are grateful to 
Michael Finkelberg, Hiroyuki Yamane, Honglian Zhang, and Huafeng Zhang for helpful discussions and valuable comments. This work was supported by the National Natural Science Foundation of China (Nos.~12171213 and 12271224). JL was supported by the Beatriz Galindo Senior grant BG24/00114 from the Spanish Ministry of Science, Innovation and Universities.

\section*{Declarations}

\noindent\textbf{Declaration on the Use of Generative AI.} Generative AI (ChatGPT, provided by OpenAI) was used to assist with language editing and improve readability. It did not generate the underlying mathematical ideas, nor was it used to supply missing arguments or complete proofs. All mathematical content, including the validity of every proof, remains the sole responsibility of the authors.

\medskip\noindent\textbf{Conflict of Interest.} The authors declare that they have no conflicts of interest relevant to this work.

\medskip\noindent\textbf{Data Availability Statement.} No datasets were generated or analyzed during the current study.

\appendix

\section{Odd-Serre stability in the shifted current algebra}\label{app:odd-serre-stability}

This appendix verifies directly the quartic odd-Serre stability needed in Lemma~\ref{lem:shifted-relation-stability} for $m,n>1$.  We prove $[O^{+},x_{m,r}^{-}]=0$ coefficientwise without using the unshifted normalization $\phi_{m,0}^{-}=(\phi_{m,0}^{+})^{-1}$; compare \cite[Lemma~3.9 and Appendix~1]{Zh2014}.

To shorten the quartic-current calculation, put $E(z)=x_m^{+}(z)$, $F(u)=x_m^{-}(u)$, $L(w)=x_{m-1}^{+}(w)$, and $R(y)=x_{m+1}^{+}(y)$. For $z_1,z_2,y,w$ independent formal variables, set
\begin{align*}
\mathscr S(z_1,z_2;y,w)
:=\operatorname{Sym}_{z_1,z_2}
 [E(z_1),[R(y),[E(z_2),L(w)]_q]_{q^{-1}}].
\end{align*}
The coefficients of $\mathscr S$ span $O^{+}$.  Since $E$ and $F$ are odd,
whereas $L$, $R$, and the Cartan currents are even, the identities
\eqref{eqs:three terms}, together with
\eqref{SQASrelation5}, give
\begin{align}
&[\mathscr S(z_1,z_2;y,w),F(u)] \notag\\
&\quad=\frac{1}{q_m-q_m^{-1}}
\sum_{\epsilon\in\{+,-\}}\epsilon
\Bigl(
\delta(z_2/u)\,\mathscr A^{\epsilon}(z_1,z_2;y,w)
+\delta(z_1/u)\,\mathscr A^{\epsilon}(z_2,z_1;y,w)
\Bigr),                                      \label{eq:odd-stability-delta}
\end{align}
where, as usual, the symbol $\epsilon$ in the sum means $+1$ for $+$ and
$-1$ for $-$, and
\begin{align}
\mathscr A^{\epsilon}(z,t;y,w)
=[E(z),[R(y),[\phi_m^{\epsilon}(t),L(w)]_q]_{q^{-1}}] -[\phi_m^{\epsilon}(t),[R(y),[E(z),L(w)]_q]_{q^{-1}}].
                                                        \label{eq:Aeps-definition}
\end{align}
Indeed, $F(u)$ supercommutes with $L(w)$ and $R(y)$, while
\[
[E(z),F(u)]
=\frac{1}{q_m-q_m^{-1}}\delta(z/u)
   (\phi_m^{+}(z)-\phi_m^{-}(z)).
\]
Thus it remains to prove $\mathscr A^{\epsilon}=0$ for
$\epsilon=\pm$.

We do this using only the quadratic current relations.  For a formal variable
$s$, define
\[
g_{-}(s,w)=\frac{q^{-1}s-w}{s-q^{-1}w},\qquad
g_{+}(s,y)=\frac{qs-y}{s-qy}.
\]
Because
$A_{m,m-1}=-1$, $A_{m,m+1}=1$, and $A_{m,m}=0$, relation
\eqref{SQASrelation4} and the same-sign quadratic relation
\eqref{SQASrelation2} imply, in the standard formal localization,
\begin{align}
\phi_m^{\epsilon}(t)L(w)&=aL(w)\phi_m^{\epsilon}(t),&
\phi_m^{\epsilon}(t)R(y)&=bR(y)\phi_m^{\epsilon}(t),\notag\\
E(z)L(w)&=cL(w)E(z),& E(z)R(y)&=dR(y)E(z),              \label{eq:abcd-exchange}
\end{align}
where
\[
a=g_{-}(t,w),\qquad b=g_{+}(t,y),\qquad
c=g_{-}(z,w),\qquad d=g_{+}(z,y),
\]
and $E(z)$ commutes with $\phi_m^{\epsilon}(t)$.  Notice that the same
exchange factors occur for $\epsilon=+$ and $\epsilon=-$; the shift changes
only the Laurent support of the Cartan current, not its exchange kernel.

Expand \eqref{eq:Aeps-definition} and use
\eqref{eq:abcd-exchange} to put every term in the order
$L(w)R(y)E(z)\phi_m^{\epsilon}(t)$.  One obtains
\begin{align}
\mathscr A^{\epsilon}(z,t;y,w)
=\frac{\mathscr C(a,b,c,d)}{q}\,
L(w)R(y)E(z)\phi_m^{\epsilon}(t),                       \label{eq:Aeps-coeff}
\end{align}
with
\begin{align}
\mathscr C(a,b,c,d)
=ab(q^2+1-qc-qd)+cd(qa+qb-q^2-1)+q(c+d-a-b).                                     \label{eq:C-abcd}
\end{align}
The vanishing of this coefficient is a two-line rational identity.  Namely,
if $A=g_{-}(s,w)$ and $B=g_{+}(s,y)$, then
\begin{align}
\frac{q(A+B)-(q^2+1)AB}{1-AB}
=\frac{y-q^2w}{y-w},                                   \label{eq:g-invariant}
\end{align}
which is independent of $s$.  Applying \eqref{eq:g-invariant} first to
$(A,B)=(a,b)$ and then to $(A,B)=(c,d)$, and cross-multiplying, gives
precisely $\mathscr C(a,b,c,d)=0$.  Hence
$\mathscr A^{\epsilon}(z,t;y,w)=0$ for both signs $\epsilon$.
Equation \eqref{eq:odd-stability-delta} therefore yields
\[
[\mathscr S(z_1,z_2;y,w),F(u)]=0.
\]
Taking coefficients of $z_1,z_2,y,w,u$ proves
\[
[O^{+},x_{m,r}^{-}]=0\qquad\text{for every }r\in\mathbb Z.
\]

The argument is insensitive to shifted endpoints.  More explicitly,
relations \eqref{eq:abcd-exchange} are just the defining current relations
before coefficient extraction.  If a coefficient of
$\phi_m^{\epsilon}(t)$ lies outside its prescribed shifted support, that
coefficient is zero; after clearing the four linear denominators in
\eqref{eq:abcd-exchange}, the identity above is polynomial in the formal
variables and therefore remains valid coefficientwise at the boundary.
The invertibility of the leading Cartan coefficient is never used.
The negative relation $[O^{-},x_{m,r}^{+}]=0$ follows from the same
calculation with all signs reversed (equivalently, by applying the standard
positive--negative involution).

\section{Detailed proof of the finite presentation}\label{app:finite-presentation-proof}
This appendix gives the mode-propagation proof of Theorem~\ref{thm:finite iso}.

\begin{proof}
First, relations \eqref{SQASrelation111}--\eqref{SQASrelation991} hold in $\mathcal{U}^{0,\mu}_{q}(\widehat{\mathfrak{sl}}_{m|n})$. Hence, there exists a superalgebra homomorphism $\varepsilon: \mathcal{E}^{\mu_{1},\mu_{2}}_{0}\rightarrow \mathcal{U}^{0,\mu}_{q}(\widehat{\mathfrak{sl}}_{m|n})$ given by
\begin{gather*}
x_{i,r}^{+}\mapsto x_{i,r}^{+},\quad x_{i,s}^{-}\mapsto x_{i,s}^{-},\quad
(\phi_{i,0}^{+})^{\pm 1}\mapsto (\phi_{i,0}^{+})^{\pm 1},\quad
(\phi_{i,l_{i}}^{-})^{\pm 1}\mapsto (\phi_{i,l_{i}}^{-})^{\pm 1},\quad
h_{i,\pm 1}\mapsto h_{i,\pm 1}
\end{gather*}
for $i\in I$, $r\in[l_{2,i}-1,0]$, and $s\in[l_{1,i},1]$. It follows from this and relations \eqref{SQASArelation1}--\eqref{SQASArelation6} given above that the map $\varepsilon: x_{i,r}^{\pm}\mapsto x_{i,r}^{\pm}, \phi_{i,m_{i}^{\pm}}^{\pm}\mapsto \phi_{i,m_{i}^{\pm}}^{\pm}$ for $i\in I$, $r\in\mathbb{Z}$, $m_{i}^{+}\in\mathbb{Z}_{\geq0}$, and $m_{i}^{-}\in\mathbb{Z}_{\geq -l_{i}}$, is a surjective superalgebra homomorphism. It remains to show that relations \eqref{SQASrelation1}--\eqref{SQASrelation7} hold in $\mathcal{E}^{\mu_{1},\mu_{2}}_{0}$.

By \eqref{SQASrelation111}, \eqref{SQASrelation441}--\eqref{SQASrelation442}, and the inductive definitions of $x_{i,r}^{\pm}$ and $\phi_{i,r}^{\pm}$, we have
\begin{gather}
\begin{split}
\phi_{i,0}^{+}x_{j,r}^{+}=q^{A_{i,j}}x_{j,r}^{+}\phi_{i,0}^{+},\quad \phi_{i,l_{i}}^{-}x_{j,r}^{+}=q^{-A_{i,j}}x_{j,r}^{+}\phi_{i,l_{i}}^{-},\quad [h_{i,\pm1},x_{j,r}^{+}]=[s_{i}A_{i,j}]_{q_{i}}x_{j,r\pm 1}^{+},\label{v1}
\end{split}
\end{gather}
\begin{gather}
\begin{split}
\phi_{i,0}^{+}x_{j,r}^{-}=q^{-A_{i,j}}x_{j,r}^{-}\phi_{i,0}^{+},\quad \phi_{i,l_{i}}^{-}x_{j,r}^{-}=q^{A_{i,j}}x_{j,r}^{-}\phi_{i,l_{i}}^{-},\quad [h_{i,\pm1},x_{j,r}^{-}]=-[s_{i}A_{i,j}]_{q_{i}}x_{j,r\pm 1}^{-},\label{v2}
\end{split}
\end{gather}
\begin{gather}
[\phi_{i,0}^{+},\phi_{j,\pm m_{j}^{\pm}}^{\pm}]=0,\quad [\phi_{i,l_{i}}^{-},\phi_{j,\pm m_{j}^{\pm}}^{\pm}]=0,\label{v3}
\end{gather}
for all $i,j\in I$, $r\in\mathbb{Z}$, $m_{j}^{+}\in\mathbb{Z}_{\geq0}$, and $m_{j}^{-}\in\mathbb{Z}_{\geq -l_{j}}$. Combining these with identities \eqref{eqs:three terms}, relations \eqref{SQASrelation111}, \eqref{SQASrelation441}--\eqref{SQASrelation551}, and \eqref{SQASArelation5}--\eqref{SQASArelation6}, we obtain
\begin{gather}
\begin{split}
[x_{i,1}^{+},x_{i,0}^{-}]=[x_{i,0}^{+},x_{i,1}^{-}]=\frac{1}{q_{i}-q_{i}^{-1}}\phi_{i,1}^{+},\\
[x_{i,l_{2,i}}^{+},x_{i,l_{1,i}-1}^{-}]=[x_{i,l_{2,i}-1}^{+},x_{i,l_{1,i}}^{-}]=-\frac{1}{q_{i}-q_{i}^{-1}}\phi_{i,l_{i}-1}^{-}.\label{v4}
\end{split}
\end{gather}
Since $\phi_{i,1}^{+}=(q_{i}-q_{i}^{-1})[x_{i,1}^{+},x_{i,0}^{-}]=(q_{i}-q_{i}^{-1})\phi_{i,0}^{+}h_{i,1}$, we have $[h_{i,1},\phi_{i,1}^{+}]=0$ and $[h_{m-1,1},\phi_{m,1}^{+}]=0$. Combining this with identities \eqref{eqs:three terms}, relations \eqref{SQASArelation5}, \eqref{v1}--\eqref{v2}, and \eqref{v4}, and the definition of $\phi_{i,2}^{+}$, we obtain
\begin{gather}
[x_{i,2}^{+},x_{i,0}^{-}]=[x_{i,1}^{+},x_{i,1}^{-}]=[x_{i,0}^{+},x_{i,2}^{-}]=\frac{1}{q_{i}-q_{i}^{-1}}\phi_{i,2}^{+},\label{v5}\\
[x_{i,l_{2,i}}^{+},x_{i,l_{1,i}-2}^{-}]=[x_{i,l_{2,i}-1}^{+},x_{i,l_{1,i}-1}^{-}]=[x_{i,l_{2,i}-2}^{+},x_{i,l_{1,i}}^{-}]=-\frac{1}{q_{i}-q_{i}^{-1}}\phi_{i,l_{i}-2}^{-}.\label{v6}
\end{gather}

\begin{lemma}\label{eq:k+l=0,1,2}
The following relations hold in $\mathcal{E}^{\mu_{1},\mu_{2}}_{0}$:
\begin{itemize}
\item[(a)] $[x_{i,k}^{+},x_{i,l}^{-}]=\frac{1}{q_{i}-q_{i}^{-1}}(\phi_{i,0}^{+}-\delta_{l_{i},0}\phi_{i,l_{i}}^{-})$ for $k+l=0$;
\item[(b)] $[x_{i,k}^{+},x_{i,l}^{-}]=\frac{1}{q_{i}-q_{i}^{-1}}\phi_{i,k+l}^{+}$ for $k+l=1,2$.
\end{itemize}
\end{lemma}
\begin{proof}
We first consider the case $i=m$. For $k+l=0,1$, we proceed by induction on $l$. The case $l=0$ follows from \eqref{SQASrelation551} and \eqref{v4}. Assume that the statements corresponding to $k+l=0,1$ hold for $l=1-t<0$. Applying $h_{m-1,-1}$ to $[x_{m,t}^{+},x_{m,1-t}^{-}]=[x_{m,0}^{+},x_{m,1}^{-}]$ and using \eqref{v4}, we obtain
\begin{align}\label{eq:k+l=0}
[x_{m,t-1}^{+},x_{m,1-t}^{-}]-[x_{m,t}^{+},x_{m,-t}^{-}]=[x_{m,-1}^{+},x_{m,1}^{-}]-[x_{m,0}^{+},x_{m,0}^{-}]=0.
\end{align}
This proves the case $l=-t$ for $k+l=0$.
Applying $h_{m-1,1}$ to \eqref{eq:k+l=0}, we obtain
\begin{align*}
[x_{m,t+1}^{+},x_{m,-t}^{-}]-[x_{m,t}^{+},x_{m,1-t}^{-}]=[x_{m,t}^{+},x_{m,1-t}^{-}]-[x_{m,t-1}^{+},x_{m,2-t}^{-}]=0.
\end{align*}
This proves the case $l=-t$ for $k+l=1$. The case $l>0$ can be proved similarly.

Now we prove the case $k+l=2$. Applying $h_{m-1,1}$ to $[x_{m,t}^{+},x_{m,1-t}^{-}]=[x_{m,0}^{+},x_{m,1}^{-}]$ ($t\in\mathbb{Z}$) and using \eqref{v5}, we obtain
\begin{align*}
[x_{m,t+1}^{+},x_{m,1-t}^{-}]-[x_{m,t}^{+},x_{m,2-t}^{-}]=[x_{m,1}^{+},x_{m,1}^{-}]-[x_{m,0}^{+},x_{m,2}^{-}]=0.
\end{align*}
Therefore, the case $k+l=2$ follows.

For $i\in I\setminus\{m\}$, the same argument, with $h_{m-1,\pm1}$ replaced by $h_{i,\pm1}$, yields (a) and (b).
\end{proof}

Relation \eqref{SQASrelation881} is equivalent to 
\begin{gather}
[h_{i,1},\phi_{i,2}^{+}]=0,\quad [h_{i,-1},\phi_{i,l_{i}-2}^{-}]=0,\label{v71}
\end{gather}
and relation \eqref{SQASrelation991} is equivalent to
\begin{gather}
[h_{m-1,1},\phi_{m,2}^{+}]=0,\quad[h_{m-1,-1},\phi_{m,l_{m}-2}^{-}]=0.\label{v72}
\end{gather}
Then, for any $i,j\in I$, we have
\begin{gather}
\begin{split}
[h_{j,-1},\phi_{i,2}^{+}]=0,\quad [h_{j,1},\phi_{i,l_{i}-2}^{-}]=0.
\end{split}\label{v8}
\end{gather}
Now define elements $h_{i,\pm2}\in\mathcal{E}^{\mu_{1},\mu_{2}}_{0}$ as follows:
\begin{gather}
\begin{split}
h_{i,2}:=\frac{1}{q_{i}-q_{i}^{-1}}(\phi_{i,0}^{+})^{-1}\phi_{i,2}^{+}-\frac{q_{i}-q_{i}^{-1}}{2}h_{i,1}^{2},\\
h_{i,-2}:=-\frac{1}{q_{i}-q_{i}^{-1}}(\phi_{i,l_{i}}^{-})^{-1}\phi_{i,l_{i}-2}^{-}+\frac{q_{i}-q_{i}^{-1}}{2}h_{i,-1}^{2}.\label{a1}
\end{split}
\end{gather}
By \eqref{SQASrelation111} and \eqref{v71}--\eqref{v8}, for every $i,j\in I$, we have
\begin{gather}
[h_{i,\pm1},h_{i,\pm2}]=0,\quad[h_{j,\mp1},h_{i,\pm2}]=0.\label{v9}
\end{gather}

\begin{lemma}
For all $i\in I$ and $r\in\mathbb{Z}$, we have
\begin{gather}
[h_{i,\pm2},x_{i,r}^{+}]=\frac{[2s_{i}A_{i,i}]_{q_{i}}}{2}x_{i,r\pm2}^{+},\quad [h_{i,\pm2},x_{i,r}^{-}]=-\frac{[2s_{i}A_{i,i}]_{q_{i}}}{2}x_{i,r\pm2}^{-},\label{eq:A21}\\
[h_{m-1,\pm2},x_{m,r}^{+}]=\frac{[2s_{m-1}A_{m-1,m}]_{q_{m-1}}}{2}x_{m,r\pm2}^{+},\quad [h_{m-1,\pm2},x_{m,r}^{-}]=-\frac{[2s_{m-1}A_{m-1,m}]_{q_{m-1}}}{2}x_{m,r\pm2}^{-}.\label{eq:A22}
\end{gather}
\end{lemma}
\begin{proof}
We first prove \eqref{eq:A21}. For $i\neq m$, by \eqref{SQASrelation221}, we have $x_{i,0}^{+}x_{i,-1}^{+}-q^{A_{i,i}}x_{i,-1}^{+}x_{i,0}^{+}=0$. Commuting this relation with $h_{i,1}$ and using \eqref{SQASrelation441}, we have $x_{i,1}^{+}x_{i,-1}^{+}-q^{A_{i,i}}(x_{i,0}^{+})^{2}=q^{A_{i,i}}x_{i,-1}^{+}x_{i,1}^{+}-(x_{i,0}^{+})^{2}$. Commuting this relation further with $x_{i,1}^{-}$ and using \eqref{SQASrelation551}, \eqref{v5}, we obtain
\begin{gather*}
\scalemath{0.96}{
\phi_{i,2}^{+}x_{i,-1}^{+}-q^{A_{i,i}}\phi_{i,1}^{+}x_{i,0}^{+}+x_{i,1}^{+}\phi_{i,0}^{+}-q^{A_{i,i}}x_{i,0}^{+}\phi_{i,1}^{+}
=q^{A_{i,i}}x_{i,-1}^{+}\phi_{i,2}^{+}-x_{i,0}^{+}\phi_{i,1}^{+}+q^{A_{i,i}}\phi_{i,0}^{+}x_{i,1}^{+}-\phi_{i,1}^{+}x_{i,0}^{+}.}
\end{gather*}

Since
\begin{gather}
x_{i,1}^{+}\phi_{i,0}^{+}-q^{A_{i,i}}x_{i,0}^{+}\phi_{i,1}^{+}=q^{A_{i,i}}\phi_{i,0}^{+}x_{i,1}^{+}-\phi_{i,1}^{+}x_{i,0}^{+},\label{v10}
\end{gather}
we have,
\begin{gather}
\phi_{i,2}^{+}x_{i,-1}^{+}-q^{A_{i,i}}\phi_{i,1}^{+}x_{i,0}^{+}=q^{A_{i,i}}x_{i,-1}^{+}\phi_{i,2}^{+}-x_{i,0}^{+}\phi_{i,1}^{+}.\label{v11}
\end{gather}
Combining \eqref{v11} with  $\phi_{i,1}^{+}=(q_{i}-q_{i}^{-1})\phi_{i,0}^{+}h_{i,1}$, $\phi_{i,2}^{+}=(q_{i}-q_{i}^{-1})\phi_{i,0}^{+}(h_{i,2}+\frac{q_{i}-q_{i}^{-1}}{2}h_{i,1}^{2})$, and relations \eqref{SQASrelation111}, \eqref{v1}, we obtain $[h_{i,2},x_{i,-1}^{+}]=\frac{[2s_{i}A_{i,i}]_{q_{i}}}{2}x_{i,1}^{+}$. Commuting this relation with $h_{i,\pm1}$ and using relations \eqref{v1}, \eqref{v9}, we obtain $[h_{i,2},x_{i,r}^{+}]=\frac{[2s_{i}A_{i,i}]_{q_{i}}}{2}x_{i,r+2}^{+}$ for all $r\in\mathbb{Z}$.

For $i=m$, by \eqref{SQASrelation331}, we have $[x_{m,0}^{+},x_{m,-1}^{+}]=0$. Commuting this relation with $h_{m-1,1}$ and using \eqref{SQASrelation441}, we have $[x_{m,1}^{+},x_{m,-1}^{+}]=0$. Commuting this relation further with $x_{m,1}^{-}$ and using \eqref{SQASrelation551}, \eqref{v5}, we obtain $\phi_{m,2}^{+}x_{m,-1}^{+}-x_{m,1}^{+}\phi_{m,0}^{+}=x_{m,-1}^{+}\phi_{m,2}^{+}-\phi_{m,0}^{+}x_{m,1}^{+}$. Since 
\begin{align}
x_{m,1}^{+}\phi_{m,0}^{+}-\phi_{m,0}^{+}x_{m,1}^{+}=0,\label{v10''}
\end{align}
we have $\phi_{m,2}^{+}x_{m,-1}^{+}=x_{m,-1}^{+}\phi_{m,2}^{+}$. Combining this relation with $\phi_{m,2}^{+}=(q_{m}-q_{m}^{-1})\phi_{m,0}^{+}(h_{m,2}+\frac{q_{m}-q_{m}^{-1}}{2}h_{m,1}^{2})$, we obtain $[h_{m,2},{x_{m,-1}^{+}}]=0$. Commuting this relation with $h_{m-1,\pm1}$ and using relations \eqref{v1} and \eqref{v9}, we obtain $[h_{m,2},x_{m,r}^{+}]=0$ for all $r\in\mathbb{Z}$.

This completes the proof of \eqref{eq:A21}.

By a similar argument, we obtain $[h_{i,-2},x_{i,r}^{+}]=\frac{[2s_{i}A_{i,i}]_{q_{i}}}{2}x_{i,r-2}^{+}$ for all $r\in\mathbb{Z}$.

The proof of $[h_{i,\pm2},x_{i,r}^{-}]=-\frac{[2s_{i}A_{i,i}]_{q_{i}}}{2}x_{i,r\pm 2}^{-}$ is analogous.

Now we prove \eqref{eq:A22}. By \eqref{SQASrelation221}, we have $[x_{m-1,0}^{+},x_{m,-1}^{+}]_{q^{A_{m-1,m}}}+[x_{m,0}^{+},x_{m-1,-1}^{+}]_{q^{A_{m-1,m}}}=0$. Commuting this with $h_{m-1,1}$ and $h_{m,1}$, using \eqref{SQASrelation41}, and using the nondegeneracy of the matrix
\[
\begin{bmatrix}
[s_{m-1}A_{m-1,m-1}]_{q_{m-1}} & [s_{m-1}A_{m-1,m}]_{q_{m-1}}\\
[s_{m}A_{m,m-1}]_{q_{m}} & [s_{m}A_{m,m}]_{q_{m}}
\end{bmatrix},
\]
we have
\begin{gather*}
[x_{m-1,1}^{+},x_{m,-1}^{+}]_{q^{A_{m-1,m}}}+[x_{m,0}^{+},x_{m-1,0}^{+}]_{q^{A_{m-1,m}}}=0,\\
[x_{m-1,0}^{+},x_{m,0}^{+}]_{q^{A_{m-1,m}}}+[x_{m,1}^{+},x_{m-1,-1}^{+}]_{q^{A_{m-1,m}}}=0.
\end{gather*}
Commuting the sum of the above two equations with $x_{m-1,1}^{-}$ and using relations \eqref{SQASrelation441}, \eqref{SQASrelation551}, and \eqref{v4}--\eqref{v5}, we obtain
\begin{gather*}
\begin{split}
&\phi_{m-1,2}^{+}x_{m,-1}^{+}-q^{A_{m-1,m}}\phi_{m-1,1}^{+}x_{m,0}^{+}+x_{m,1}^{+}\phi_{m-1,0}^{+}-q^{A_{m-1,m}}x_{m,0}^{+}\phi_{m-1,1}^{+}\\&=
q^{A_{m-1,m}}x_{m,-1}^{+}\phi_{m-1,2}^{+}-x_{m,0}^{+}\phi_{m-1,1}^{+}+q^{A_{m-1,m}}\phi_{m-1,0}^{+}x_{m,1}^{+}-\phi_{m-1,1}^{+}x_{m,0}^{+}.
\end{split}
\end{gather*}
Since
\begin{gather}
x_{m,1}^{+}\phi_{m-1,0}^{+}-q^{A_{m-1,m}}x_{m,0}^{+}\phi_{m-1,1}^{+}=q^{A_{m-1,m}}\phi_{m-1,0}^{+}x_{m,1}^{+}-\phi_{m-1,1}^{+}x_{m,0}^{+},\label{v10'}
\end{gather}
We have,
\begin{gather}
\phi_{m-1,2}^{+}x_{m,-1}^{+}-q^{A_{m-1,m}}\phi_{m-1,1}^{+}x_{m,0}^{+}=q^{A_{m-1,m}}x_{m,-1}^{+}\phi_{m-1,2}^{+}-x_{m,0}^{+}\phi_{m-1,1}^{+},\label{v11'}
\end{gather}
Combining \eqref{v11'} with  
\begin{gather*}
\phi_{m-1,1}^{+}=(q_{m-1}-q_{m-1}^{-1})\phi_{m-1,0}^{+}h_{m-1,1},\\
\phi_{m-1,2}^{+}=(q_{m-1}-q_{m-1}^{-1})\phi_{m-1,0}^{+}(h_{m-1,2}+\frac{q_{m-1}-q_{m-1}^{-1}}{2}h_{m-1,1}^{2}),
\end{gather*}
and relations \eqref{SQASrelation11}, \eqref{v1}, we obtain $[h_{m-1,2},x_{m,-1}^{+}]=\frac{[2s_{m-1}A_{m-1,m}]_{q_{m-1}}}{2}x_{m,1}^{+}$. Commuting this relation with $h_{m-1,\pm1}$ and using relations \eqref{v1} and \eqref{v9}, we obtain 
\begin{gather*}
[h_{m-1,2},x_{m,r}^{+}]=\frac{[2s_{m-1}A_{m-1,m}]_{q_{m-1}}}{2}x_{m,r+2}^{+}\quad\text{for all $r\in\mathbb{Z}$}.
\end{gather*}
By a similar argument, we obtain $[h_{m-1,-2},x_{m,r}^{+}]=\frac{[2s_{m-1}A_{m-1,m}]_{q_{m-1}}}{2}x_{m,r-2}^{+}$ for all $r\in\mathbb{Z}$.

The proof of $[h_{m-1,\pm2},x_{m,r}^{-}]=-\frac{[2s_{m-1}A_{m-1,m}]_{q_{m-1}}}{2}x_{m,r\pm 2}^{-}$ is analogous.

\end{proof}

Now we continue the proof of Theorem~\ref{thm:finite iso}.

\textbf{The relation \eqref{SQASrelation2} for $i=j$:} Here $A_{i,i}\neq0$, hence $i\neq m$. Let
\[
X^{+}(i;r,s):=[x_{i,r+1}^{+},x_{i,s}^{+}]_{q^{A_{i,i}}}+[x_{i,s+1}^{+},x_{i,r}^{+}]_{q^{A_{i,i}}}.
\]
We need to prove $X^{+}(i;r,s)=0$ for all $r,s\in\mathbb Z$. By \eqref{SQASrelation221}, $X^{+}(i;r,s)=X^{+}(i;s,r)$ and $X^{+}(i;-1,-1)=0$. For $a\in\{\pm1,\pm2\}$, set
\[
L_{i,a}:=\frac{a}{[as_iA_{i,i}]_{q_i}}\operatorname{ad}(h_{i,a}).
\]
By \eqref{v1} and \eqref{eq:A21},
\[
L_{i,a}X^{+}(i;r,s)=X^{+}(i;r+a,s)+X^{+}(i;r,s+a).
\]
On the coefficient family $X^{+}(i;r,s)$, the operators $L_{i,a}$ act as the commuting shift operators $p_a(u,v)=u^a+v^a$, where $u,v$ record shifts of the first and second indices.  Since
\[
\mathbb C[u^{\pm1},v^{\pm1}]^{\mathfrak S_2}
=\mathbb C[p_1,p_2,p_{-1},p_{-2}],
\]
where $uv=(p_1^2-p_2)/2$ and $(uv)^{-1}=(p_{-1}^2-p_{-2})/2$, every symmetric Laurent shift of $X^{+}(i;-1,-1)$ is obtained by a polynomial in $L_{i,\pm1},L_{i,\pm2}$.  As $X^{+}(i;r,s)$ is symmetric in $r,s$, this gives $X^{+}(i;r,s)=0$ for all $r,s\in\mathbb Z$.

Similarly, we have $[x_{i,r+1}^{-},x_{i,s}^{-}]_{q^{-A_{i,i}}}+[x_{i,s+1}^{-},x_{i,r}^{-}]_{q^{-A_{i,i}}}=0$ for all $r,s\in\mathbb{Z}$.

\textbf{The relation \eqref{SQASrelation2} for $i\neq j$:} Let
\[
X^{+}(i,j;r,s):=[x_{i,r+1}^{+},x_{j,s}^{+}]_{q^{A_{i,j}}}+[x_{j,s+1}^{+},x_{i,r}^{+}]_{q^{A_{i,j}}}.
\]
We only need to consider $A_{i,j}\neq0$. By \eqref{SQASrelation221}, $X^{+}(i,j;-1,-1)=0$.

If $i,j\in I\setminus\{m\}$, then for either choice of sign, \eqref{v1} gives
\begin{align*}
[h_{i,\pm1},X^{+}(i,j;r,s)]&=[s_iA_{i,i}]_{q_i}X^{+}(i,j;r\pm1,s)+[s_iA_{i,j}]_{q_i}X^{+}(i,j;r,s\pm1),\\
[h_{j,\pm1},X^{+}(i,j;r,s)]&=[s_jA_{j,i}]_{q_j}X^{+}(i,j;r\pm1,s)+[s_jA_{j,j}]_{q_j}X^{+}(i,j;r,s\pm1).
\end{align*}
The coefficient matrix
\[
\begin{bmatrix}
[s_iA_{i,i}]_{q_i}&[s_iA_{i,j}]_{q_i}\\
[s_jA_{j,i}]_{q_j}&[s_jA_{j,j}]_{q_j}
\end{bmatrix}
\]
is nondegenerate; since $i,j\neq m$ and $A_{i,j}\neq0$, its determinant is $[3]_q\neq0$. Hence, whenever $X^{+}(i,j;r,s)=0$, the two adjoint relations with $h_{i,\pm1}$ and $h_{j,\pm1}$ imply separately
\[
X^{+}(i,j;r\pm1,s)=X^{+}(i,j;r,s\pm1)=0.
\]
Starting from $(-1,-1)$ and using both signs proves the relation for every $(r,s)\in\mathbb Z^2$.

If one of $i,j$ is $m$, we may assume $i=m$; then $j=m\pm1$.  Using the two even neighboring Cartan directions $m-1$ and $m+1$, we obtain for $u\in\{m-1,m+1\}$
\[
[h_{u,\pm1},X^{+}(m,j;r,s)]
=[s_uA_{u,m}]_{q_u}X^{+}(m,j;r\pm1,s)
+[s_uA_{u,j}]_{q_u}X^{+}(m,j;r,s\pm1).
\]
For $j=m-1$ or $j=m+1$, respectively, the resulting $2\times2$ coefficient matrix has determinant $\pm[2]_q\neq0$.  Thus the two coordinates can again be shifted independently, and the same argument from $X^{+}(m,j;-1,-1)=0$ proves $X^{+}(m,j;r,s)=0$ for all $r,s\in\mathbb Z$.

Similarly, we have $[x_{i,r+1}^{-},x_{j,s}^{-}]_{q^{-A_{i,j}}}+[x_{j,s+1}^{-},x_{i,r}^{-}]_{q^{-A_{i,j}}}=0$ for all $r,s\in\mathbb{Z}$.

\textbf{The relation \eqref{SQASrelation3} for $i=j=m$:} Set $W^{+}(r,s):=[x_{m,r}^{+},x_{m,s}^{+}]$.  Since the odd supercommutator is symmetric, $W^{+}(r,s)=W^{+}(s,r)$, and \eqref{SQASrelation331} gives $W^{+}(0,0)=0$. For $a\in\{\pm1,\pm2\}$, set
\[
M_a:=\frac{a}{[as_{m-1}A_{m-1,m}]_{q_{m-1}}}\operatorname{ad}(h_{m-1,a}).
\]
By \eqref{v1} and \eqref{eq:A22},
\[
M_a W^{+}(r,s)=W^{+}(r+a,s)+W^{+}(r,s+a).
\]
As above, on the coefficient family $W^{+}(r,s)$ the operators $M_{\pm1}$ and $M_{\pm2}$ realize the commuting shift operators
$p_{\pm1}$ and $p_{\pm2}$, respectively.  These elements generate
$\mathbb C[u^{\pm1},v^{\pm1}]^{\mathfrak S_2}$.
Therefore every symmetric Laurent shift of $W^{+}(0,0)$ vanishes, and hence
$[x_{m,r}^{+},x_{m,s}^{+}]=0$ for all $r,s\in\mathbb Z$.

Similarly, we have $[x_{m,r}^{-},x_{m,s}^{-}]=0$ for all $r,s\in\mathbb{Z}$.

\textbf{The relation \eqref{SQASrelation3} for $i\neq j$:} We need to prove $[x_{i,r}^{+},x_{j,s}^{+}]=0$ for $A_{i,j}=0$. By \eqref{SQASrelation331}, we have $[x_{i,0}^{+},x_{j,0}^{+}]=0$ for $A_{i,j}=0$.

If $i,j\in I\setminus\{m\}$, then by \eqref{v1}, we have
\begin{align*}
[h_{i,\pm1},[x_{i,r}^{+},x_{j,s}^{+}]]=[s_{i}A_{i,i}]_{q_{i}}[x_{i,r\pm1}^{+},x_{j,s}^{+}],\\
[h_{j,\pm1},[x_{i,r}^{+},x_{j,s}^{+}]]=[s_{j}A_{j,j}]_{q_{j}}[x_{i,r}^{+},x_{j,s\pm1}^{+}].
\end{align*}
By induction on $r$ and $s$, we obtain $[x_{i,r}^{+},x_{j,s}^{+}]=0$ for all $r,s\in\mathbb{Z}$.

If one of $i$ and $j$ is equal to $m$, we may assume without loss of generality that $i=m$. By \eqref{v1}, we have
\begin{align*}
[h_{i-1,\pm1},[x_{i,r}^{+},x_{j,s}^{+}]]&=[s_{i-1}A_{i-1,i}]_{q_{i-1}}[x_{i,r\pm1}^{+},x_{j,s}^{+}]+[s_{i-1}A_{i-1,j}]_{q_{i-1}}[x_{i,r}^{+},x_{j,s\pm1}^{+}],\\
[h_{j,\pm1},[x_{i,r}^{+},x_{j,s}^{+}]]&=[s_{j}A_{j,j}]_{q_{j}}[x_{i,r}^{+},x_{j,s\pm1}^{+}].
\end{align*}
Combining these equations with the nondegeneracy of the matrix
\[
\begin{bmatrix}
[s_{i-1}A_{i-1,i}]_{q_{i-1}} & [s_{i-1}A_{i-1,j}]_{q_{i-1}}\\
0 & [s_{j}A_{j,j}]_{q_{j}}
\end{bmatrix},
\]
we obtain $[x_{m,r}^{+},x_{j,s}^{+}]=0$ for all $r,s\in\mathbb{Z}$ by induction on $r$ and $s$.

Similarly, we have $[x_{i,r}^{-},x_{j,s}^{-}]=0$ for all $r,s\in\mathbb{Z}$.

\textbf{The relation \eqref{SQASrelation4} for $i=j$:} We only prove the case $\epsilon=+$. The case $\epsilon=-$ is analogous. Let $Y^{+}(i;r,s):=[\phi_{i,r+1}^{+},x_{i,s}^{+}]_{q^{A_{i,i}}}+[x_{i,s+1}^{+},\phi_{i,r}^{+}]_{q^{A_{i,i}}}$. We need to prove $Y^{+}(i;r,s)=0$ for all $r\in\mathbb{N}$ and $s\in\mathbb{Z}$. The case $Y^{+}(i;0,0)=[\phi_{i,1}^{+},x_{i,0}^{+}]_{q^{A_{i,i}}}+[x_{i,1}^{+},\phi_{i,0}^{+}]_{q^{A_{i,i}}}=0$ follows from \eqref{v10} and \eqref{v10''}.

Since
\begin{gather*}
[h_{i,\pm1},Y^{+}(i;0,s)]=[s_{i}A_{i,i}]_{q_{i}}Y^{+}(i;0,s\pm1)\quad\text{for $i\in I\setminus\{m\}$},\\
[h_{m-1,\pm1},Y^{+}(m;0,s)]=[s_{m-1}A_{m-1,m}]_{q_{m-1}}Y^{+}(m;0,s\pm1),
\end{gather*}
we obtain $Y^{+}(i;0,s)=0$ for all $i\in I$ and $s\in\mathbb{Z}$. Using \eqref{v1}, the equalities $X^{+}(i;r,s)=0$ and $[x_{m,r}^{+},x_{m,s}^{+}]=0$ for all $r,s\in\mathbb{Z}$, and the definition of $\phi_{i,r}^{+}$, we have
\begin{gather*}
Y^{+}(i;r,s-1)+Y^{+}(i;s,r-1)=(q_{i}-q_{i}^{-1})[X^{+}(i;r-1,s-1),x_{i,1}^{-}]=0\quad \text{for $i\in I\setminus\{m\}$},\\
Y^{+}(m;r,s-1)-Y^{+}(m;s,r-1)=-(q_{m}-q_{m}^{-1})([[x_{m,r}^{+},x_{m,s-1}^{+}]-[x_{m,s}^{+},x_{m,r-1}^{+}],x_{m,1}^{-}])=0,
\end{gather*}
for $r,s\in\mathbb{Z}_{\geq0}$. In particular,
\begin{gather*}
Y^{+}(i;r,-1)+Y^{+}(i;0,r-1)=0\quad\text{for $i\in I\setminus\{m\}$, $r\in\mathbb{N}$},\\
Y^{+}(m;r,-1)-Y^{+}(m;0,r-1)=0\quad\text{for $r\in\mathbb{N}$},
\end{gather*}
Hence,
\begin{gather}
Y^{+}(i;r,-1)=0\quad \text{for all $r\in\mathbb{N}$}.\label{v12}
\end{gather}
Since
\begin{gather*}
[h_{i,\pm1},Y^{+}(i;1,s)]=[s_{i}A_{i,i}]_{q_{i}}Y^{+}(i;1,s\pm1)\quad \text{for $i\in I\setminus\{m\}$},\\
[h_{m-1,\pm1},Y^{+}(m;1,s)]=[s_{m-1}A_{m-1,m}]_{q_{m-1}}Y^{+}(m;1,s\pm1),
\end{gather*}
we obtain $Y^{+}(i;1,s)=0$ for all $i\in I$ and $s\in\mathbb{Z}$.

We prove the following statements by induction on $N\in\mathbb{Z}_{\geq1}$:

\begin{itemize}
\item[$(i_{N})$] $[(\phi_{i,0}^{+})^{-1}\phi_{i,r}^{+},x_{i,k}^{\pm}]=q^{\pm A_{i,i}}(\phi_{i,0}^{+})^{-1}\phi_{i,r-1}^{+}x_{i,k+1}^{\pm}-q^{\mp A_{i,i}}x_{i,k+1}^{\pm}(\phi_{i,0}^{+})^{-1}\phi_{i,r-1}^{+}$ for $i\in I\setminus\{m\}$ and $r\in[0,N+1]$;
\item[$(ii_{N})$]
\begin{align*}
[(\phi_{m,0}^{+})^{-1}\phi_{m,r}^{+},x_{m-1,k}^{\pm}]&=q^{\pm A_{m,m-1}}(\phi_{m,0}^{+})^{-1}\phi_{m,r-1}^{+}x_{m-1,k+1}^{\pm}\\&-q^{\mp A_{m,m-1}}x_{m-1,k+1}^{\pm}(\phi_{m,0}^{+})^{-1}\phi_{m,r-1}^{+}\quad \text{for $r\in[0,N+1]$};
\end{align*}
\item[$(iii_{N})$] $[h_{i,1},\phi_{i,r}^{+}]=0$ and $[h_{m-1,1},\phi_{m,r}^{+}]=0$ for $r\in[0,N+1]$;
\item[$(iv_{N})$] $[h_{i,-1},\phi_{i,r}^{+}]=0$ and $[h_{m-1,-1},\phi_{m,r}^{+}]=0$ for $r\in[0,N+1]$;
\item[$(v_{N})$] $[x_{i,r}^{+},x_{i,s}^{-}]=\frac{1}{q_{i}-q_{i}^{-1}}\phi_{i,r+s}^{+}$ for any $r,s\in\mathbb{Z}$ with $r+s\in[1,N+2]$;
\item[$(vi_{N})$] $Y^{+}(i;r,s)=0$ for any $r\in[0,N]$ with $s\in\mathbb{Z}$;
\item[$(vii_{N})$] $[\phi_{i,r}^{+},\phi_{i,s}^{+}]=0$ and $[\phi_{m-1,r}^{+},\phi_{m,s}^{+}]=0$ for any $r,s\in\mathbb{Z}_{\geq0}$ with $r+s\in[0,N+2]$.
\end{itemize}

First, $(i_{1})$--$(iv_{1})$ and $(vi_{1})$--$(vii_{1})$ were already proved at the beginning, while $(v_{1})$ follows from relations \eqref{v1}--\eqref{v2}, \eqref{v4}--\eqref{v5}, and \eqref{v71}--\eqref{v72}. Assume that $(i_{N})$--$(vii_{N})$ hold for some $N\in\mathbb{Z}_{\geq1}$.

Consider the polynomial algebra $X:=\mathbb{C}(q)[x_{1},x_{2},\dots]$, which is $\mathbb{N}$-graded by $\deg(x_{r})=r$ for $r\in\mathbb{Z}_{\geq1}$. Define elements $h_{r}\in X$, $r\in\mathbb{Z}_{\geq1}$, by $\exp((q_{i}-q_{i}^{-1})(\sum_{r=1}^{\infty}h_{r}z^{-r}))=1+\sum_{r=1}^{\infty}x_{r}z^{-r}$. Then, for $r\in\mathbb{Z}_{\geq1}$, $h_{r} =\frac{1}{q_{i}-q_{i}^{-1}}x_{r}+p_{r}(x_{1},\dots,x_{r-1})$, where
\begin{align*}
p_{r}(x_{1},\dots,x_{r-1}) =-\frac{1}{q_{i}-q_{i}^{-1}}\sum_{s=2}^{r}\frac{(q_{i}-q_{i}^{-1})^{s}}{s!}\sum_{\substack{r_{1}+\cdots+r_{s}=r,\\ \text{ all }r_{j}\geq1}}h_{r_{1}}\dots h_{r_{s}},
\end{align*}
and $p_{r}(x_{1},\dots,x_{r-1})$ is homogeneous of degree $r$. Now, for $1\leq r\leq N+1$, define $h_{i,r}\in \mathcal{E}^{\mu_{1},\mu_{2}}_{0}$ by 
\begin{align}
\begin{split}
h_{i,r}:&=\frac{(\phi_{i,0}^{+})^{-1}\phi_{i,r}^{+}}{q_{i}-q_{i}^{-1}}+p_{r}((\phi_{i,0}^{+})^{-1}\phi_{i,1}^{+},\dots,(\phi_{i,0}^{+})^{-1}\phi_{i,r-1}^{+})\\&
=\frac{(\phi_{i,0}^{+})^{-1}\phi_{i,r}^{+}}{q_{i}-q_{i}^{-1}}-\frac{1}{q_{i}-q_{i}^{-1}}\sum_{s=2}^{r}\frac{(q_{i}-q_{i}^{-1})^{s}}{s!}\sum_{\substack{r_{1}+\cdots+r_{s}=r,\\ \text{ all }r_{j}\geq1}}h_{i,r_{1}}\dots h_{i,r_{s}}.\label{a3}
\end{split}
\end{align}
By the induction hypothesis $(vii_{N})$ and the degree condition on $p_{r}$, the elements $h_{i,r}$ are well-defined and independent of the choice of $N>r-1$. We will use the identities
\begin{align}
[h_{i,r},x_{i,s}^{+}]=\frac{[rs_{i}A_{i,i}]_{q_{i}}}{r}x_{i,s+r}^{+}\quad \text{for $i\in I$, $r\in[1,N+1]$, $s\in\mathbb{Z}$},\label{v13}\\
[h_{m-1,r},x_{m,s}^{+}]=\frac{[rs_{m-1}A_{m-1,m}]_{q_{m-1}}}{r}x_{m,s+r}^{+}\quad \text{for $r\in[1,N+1]$, $s\in\mathbb{Z}$}.\label{v13'}
\end{align}

$(i_{N+1})$: By $(v_{N})$, \eqref{v1}, \eqref{SQASrelation2}, and \eqref{eqs:three terms}, we have
\begin{align*}
[(\phi_{i,0}^{+})^{-1} & \phi_{i,N+2}^{+},x_{i,k}^{+}]=(q_{i}-q_{i}^{-1})[(\phi_{i,0}^{+})^{-1}[x_{i,k+1}^{+},x_{i,N+1-k}^{-}],x_{i,k}^{+}]\\&=(q_{i}-q_{i}^{-1})(\phi_{i,0}^{+})^{-1}[[x_{i,k+1}^{+},x_{i,N+1-k}^{-}],x_{i,k}^{+}]_{q^{A_{i,i}}}\\&=(q_{i}-q_{i}^{-1})(\phi_{i,0}^{+})^{-1}([x_{i,k+1}^{+},[x_{i,N+1-k}^{-},x_{i,k}^{+}]]_{q^{A_{i,i}}}+[[x_{i,k+1}^{+},x_{i,k}^{+}]_{q^{A_{i,i}}},x_{i,N+1-k}^{-}])\\&=-(\phi_{i,0}^{+})^{-1}[x_{i,k+1}^{+},\phi_{i,N+1}^{+}]_{q^{A_{i,i}}}\\&=q^{A_{i,i}}(\phi_{i,0}^{+})^{-1}\phi_{i,N+1}^{+}x_{i,k+1}^{+}-q^{-A_{i,i}}x_{i,k+1}^{+}(\phi_{i,0}^{+})^{-1}\phi_{i,N+1}^{+}.
\end{align*}

The case of $x_{i,k}^{-}$ is analogous.

$(ii_{N+1})$: By $(v_N)$, \eqref{v1}, \eqref{SQASrelation2}, and \eqref{eqs:three terms}, we obtain
\begin{align*}
[(\phi_{m,0}^{+})^{-1} & \phi_{m,N+2}^{+},x_{m-1,k}^{+}]
=(q_m-q_m^{-1})(\phi_{m,0}^{+})^{-1}
[[x_{m,k+1}^{+},x_{m,N+1-k}^{-}],x_{m-1,k}^{+}]_{q^{A_{m,m-1}}}\\
&=-(q_m-q_m^{-1})(\phi_{m,0}^{+})^{-1}
[x_{m-1,k+1}^{+},[x_{m,k}^{+},x_{m,N+1-k}^{-}]]_{q^{A_{m,m-1}}}\\
&=-(\phi_{m,0}^{+})^{-1}
[x_{m-1,k+1}^{+},\phi_{m,N+1}^{+}]_{q^{A_{m,m-1}}}\\
&=q^{A_{m,m-1}}(\phi_{m,0}^{+})^{-1}\phi_{m,N+1}^{+}x_{m-1,k+1}^{+}
-q^{-A_{m,m-1}}x_{m-1,k+1}^{+}(\phi_{m,0}^{+})^{-1}\phi_{m,N+1}^{+}.
\end{align*}

The case of $x_{m-1,k}^{-}$ is analogous.

$(iii_{N+1})_{a}$: For $i\in I\setminus\{m\}$, expanding the commutators and applying $(i_{N+1})$, $(v_N)$, and \eqref{eqs:three terms} gives
\begin{align*}
[\phi_{i,N+2}^{+},h_{i,1}]
&=\frac{[2s_iA_{i,i}]_{q_i}}{[s_iA_{i,i}]_{q_i}}
[(\phi_{i,0}^{+})^{-1}\phi_{i,N+1}^{+},[x_{i,0}^{+},x_{i,2}^{-}]]
-[(\phi_{i,0}^{+})^{-1}\phi_{i,N}^{+},[x_{i,0}^{+},x_{i,3}^{-}]]\\
&=\frac{[3s_iA_{i,i}]_{q_i}}{[s_iA_{i,i}]_{q_i}}
[(\phi_{i,0}^{+})^{-1}\phi_{i,N}^{+},[x_{i,0}^{+},x_{i,3}^{-}]]
-\frac{[2s_iA_{i,i}]_{q_i}}{[s_iA_{i,i}]_{q_i}}
[(\phi_{i,0}^{+})^{-1}\phi_{i,N-1}^{+},[x_{i,0}^{+},x_{i,4}^{-}]]\\
&=\cdots\\
&=\frac{[(N+2)s_iA_{i,i}]_{q_i}}{[s_iA_{i,i}]_{q_i}}
[(\phi_{i,0}^{+})^{-1}\phi_{i,1}^{+},[x_{i,0}^{+},x_{i,N+2}^{-}]]\\
&=(q_i-q_i^{-1})[(N+2)s_iA_{i,i}]_{q_i}
\bigl([x_{i,1}^{+},x_{i,N+2}^{-}]
-[x_{i,0}^{+},x_{i,N+3}^{-}]\bigr).
\end{align*}
On the other hand, $[\phi_{i,N+2}^{+},h_{i,1}]=-(q_{i}-q_{i}^{-1})[s_{i}A_{i,i}]_{q_{i}}([x_{i,1}^{+},x_{i,N+2}^{-}]-[x_{i,0}^{+},x_{i,N+3}^{-}])$.  Here $s_iA_{i,i}=2$ for $i\neq m$, and $[2N+4]_{q_i}+[2]_{q_i}=[N+3]_{q_i}(q_i^{N+1}+q_i^{-N-1})\neq0$ because $q$ is not a root of unity.  Comparing the two expressions therefore gives $[x_{i,1}^{+},x_{i,N+2}^{-}]=[x_{i,0}^{+},x_{i,N+3}^{-}]$, and consequently $[\phi_{i,N+2}^{+},h_{i,1}]=0$.

$(iv_{N+1})$: By \eqref{v1}--\eqref{v2}, $(v_{N})$, and \eqref{eqs:three terms}, for $r,s\in\mathbb{N}$ such that $r+s=N+2$, we obtain
\begin{align*}
[h_{i,-1},\phi_{i,N+2}^{+}]
&=(q_{i}-q_{i}^{-1})[h_{i,-1},[x_{i,r}^{+},x_{i,s}^{-}]]\\
&=(q_{i}-q_{i}^{-1})[s_{i}A_{i,i}]_{q_{i}}
 \bigl([x_{i,r-1}^{+},x_{i,s}^{-}]-[x_{i,r}^{+},x_{i,s-1}^{-}]\bigr)=0,
\end{align*}
and
\begin{align*}
[h_{m-1,-1},\phi_{m,N+2}^{+}]
=(q_{m}-q_{m}^{-1})[s_{m-1}A_{m-1,m}]_{q_{m-1}}\bigl([x_{m,r-1}^{+},x_{m,s}^{-}]-[x_{m,r}^{+},x_{m,s-1}^{-}]\bigr)=0.
\end{align*}

$(v_{N+1})_{a}$: By $(iii_{N+1})$ and $(v_{N})$, for $r,s\in\mathbb{N}$ with $r+s=N+2$, we obtain
\begin{align*}
0=[h_{i,1},\phi_{i,N+2}^{+}]=(q_{i}-q_{i}^{-1})[s_{i}A_{i,i}]_{q_{i}}([x_{i,r+1}^{+},x_{i,s}^{-}]-[x_{i,r}^{+},x_{i,s+1}^{-}])\quad\text{for $i\in I\setminus\{m\}$.}
\end{align*}
Hence $[x_{i,r+1}^{+},x_{i,s}^{-}]=[x_{i,r}^{+},x_{i,s+1}^{-}]$ for all $i\in I\setminus\{m\}$ and $r,s\in\mathbb{N}$ with $r+s=N+2$. Combining this with the definition of $\phi_{i,N+3}^{+}$, we obtain $[x_{i,r'}^{+},x_{i,N+3-r'}^{-}]=\frac{1}{q_i-q_i^{-1}}\phi_{i,N+3}^{+}$ for $i\in I\setminus\{m\}$ and $0\leq r'\leq N+3$.

$(iii_{N+1})_{b}$: By $(ii_{N+1})$, $(v_N)$, $(v_{N+1})_{a}$, and \eqref{eqs:three terms}, expanding the commutators gives
\begin{align*}
& [\phi_{m,N+2}^{+},h_{m-1,1}]
=\phi_{m,0}^{+}(\phi_{m-1,0}^{+})^{-1}
\Bigg(
\frac{[2s_mA_{m,m-1}]_{q_m}}{[s_mA_{m,m-1}]_{q_m}}
[(\phi_{m,0}^{+})^{-1}\phi_{m,N+1}^{+},
 [x_{m-1,0}^{+},x_{m-1,2}^{-}]]\\
&\hspace{5.8cm}
-[(\phi_{m,0}^{+})^{-1}\phi_{m,N}^{+},
 [x_{m-1,0}^{+},x_{m-1,3}^{-}]]
\Bigg)\\
&=\cdots\\
&=\phi_{m,0}^{+}(\phi_{m-1,0}^{+})^{-1}
\frac{[(N+2)s_mA_{m,m-1}]_{q_m}}{[s_mA_{m,m-1}]_{q_m}}
[(\phi_{m,0}^{+})^{-1}\phi_{m,1}^{+},
 [x_{m-1,0}^{+},x_{m-1,N+2}^{-}]]\\
&= \scalemath{0.96}{ (q_m-q_m^{-1})\phi_{m,0}^{+}(\phi_{m-1,0}^{+})^{-1}
[(N+2)s_mA_{m,m-1}]_{q_m} \bigl([x_{m-1,1}^{+},x_{m-1,N+2}^{-}]
-[x_{m-1,0}^{+},x_{m-1,N+3}^{-}]\bigr)=0. }
\end{align*}

$(v_{N+1})_{b}$: By $(iii_{N+1})$ and $(v_{N})$, for $r,s\in\mathbb{N}$ with $r+s=N+2$, we obtain
\begin{align*}
0=[h_{m-1,1},\phi_{m,N+2}^{+}]=(q-q^{-1})[A_{m-1,m}]_{q}([x_{m,r+1}^{+},x_{m,s}^{-}]-[x_{m,r}^{+},x_{m,s+1}^{-}]).
\end{align*}
Hence $[x_{m,r+1}^{+},x_{m,s}^{-}]=[x_{m,r}^{+},x_{m,s+1}^{-}]$ for $r,s\in\mathbb{N}$ with $r+s=N+2$. Combining this with the definition of $\phi_{m,N+3}^{+}$, we obtain $[x_{m,r'}^{+},x_{m,N+3-r'}^{-}]=\frac{1}{q_m-q_m^{-1}}\phi_{m,N+3}^{+}$ for all $0\leq r'\leq N+3$.

$(vi_{N+1})$: By $(iii_{N+1})$ and $(iv_{N+1})$, we have
\begin{align*}
& [h_{i,\pm1},Y^{+}(i;N+1,s)]=[s_iA_{i,i}]_{q_i}Y^{+}(i;N+1,s\pm1)\quad\text{ for $i\in I\setminus\{m\}$},\\
& [h_{m-1,\pm1},Y^{+}(m;N+1,s)]=[A_{m-1,m}]_{q}Y^{+}(m;N+1,s\pm1).
\end{align*}
By \eqref{v12}, we have $Y^{+}(i;N+1,-1)=0$ for all $i\in I$. Hence $Y^{+}(i;N+1,s)=0$ for all $i\in I$ and $s\in\mathbb{Z}$.

$(vii_{N+1})$: We need to prove $[h_{i,r},\phi_{i,s}^{+}]=0$ and $[h_{m-1,r},\phi_{m,s}^{+}]=0$ for $r,s\in\mathbb{Z}_{>0}$ with $r+s=N+3$. By $(v_{N})$, \eqref{v13}--\eqref{v13'}, and $(v_{N+1})$, we obtain
\begin{align*}
[h_{i,r},\phi_{i,s}^{+}]&=(q_{i}-q_{i}^{-1})[h_{i,r},[x_{i,s}^{+},x_{i,0}^{-}]]=(q_{i}-q_{i}^{-1})([[h_{i,r},x_{i,s}^{+}],x_{i,0}^{-}]+[x_{i,s}^{+},[h_{i,r},x_{i,0}^{-}]])\\&=(q_{i}-q_{i}^{-1})\frac{[rs_{i}A_{i,i}]_{q_{i}}}{r}([x_{i,r+s}^{+},x_{i,0}^{-}]-[x_{i,s}^{+},x_{i,r}^{-}])=0,
\end{align*}
and
\begin{align*}
[h_{m-1,r},\phi_{m,s}^{+}]=(q_{m}-q_{m}^{-1})\frac{[rs_{m-1}A_{m-1,m}]_{q_{m-1}}}{r}([x_{m,r+s}^{+},x_{m,0}^{-}]-[x_{m,s}^{+},x_{m,r}^{-}])=0.
\end{align*}

This completes the proof of $(i_{N})$--$(vii_{N})$. In particular, $(vi_{N})$ implies that $Y^{+}(i;r,s)=0$ for all $r\in \mathbb{N}$ and $s\in\mathbb{Z}$.

\textbf{The relation \eqref{SQASrelation5} for $i=j$:} We need to prove
\begin{gather}
[x_{i,r}^{+},x_{i,N-r}^{-}]=\frac{1}{q_{i}-q_{i}^{-1}}
\begin{cases}
\phi_{i,N}^{+}-\delta_{N,0}\delta_{l_{i},0}\phi_{i,l_{i}}^{-}\quad&\text{if $N\in\mathbb{Z}_{\geq0}$},\\
-\phi_{i,N}^{-}+\delta_{N,0}\delta_{l_{i},0}\phi_{i,0}^{+}\quad&\text{if $N\in\mathbb{Z}_{\leq l_{i}}$},\\
0\quad&\text{if $l_{i}+1\leq N\leq-1$}.
\end{cases}\label{v17}
\end{gather}

\textbf{Case $N\in\mathbb{Z}_{>0}$.} If $0\leq r\leq N$, then $[x_{i,r}^{+},x_{i,N-r}^{-}]=\frac{1}{q_{i}-q_{i}^{-1}}\phi_{i,N}^{+}$ by ($v_{N}$). If $r>N$, we proceed by induction on $r$. By \eqref{v1}--\eqref{v2} and ($iv_{N}$)--($v_{N}$), we obtain
\begin{align*}
[x_{i,r}^{+},x_{i,N-r}^{-}]&=-[s_{i}A_{i,i}]_{q_{i}}^{-1}[x_{i,r}^{+},[h_{i,-1},x_{i,N-r+1}^{-}]]\\&=-[s_{i}A_{i,i}]_{q_{i}}^{-1}([h_{i,-1},[x_{i,r}^{+},x_{i,N-r+1}^{-}]]+[[x_{i,r}^{+},h_{i,-1}],x_{i,N-r+1}^{-}])\\&=[x_{i,r-1}^{+},x_{i,N-r+1}^{-}]=\frac{1}{q_i-q_i^{-1}}\phi_{i,N}^{+},
\end{align*}
for $i\in I\setminus\{m\}$, and
\begin{align*}
[x_{m,r}^{+},x_{m,N-r}^{-}]&=-[s_{m-1}A_{m-1,m}]_{q_{m-1}}^{-1}[x_{m,r}^{+},[h_{m-1,-1},x_{m,N-r+1}^{-}]]\\&=-[s_{m-1}A_{m-1,m}]_{q_{m-1}}^{-1}([h_{m-1,-1},[x_{m,r}^{+},x_{m,N-r+1}^{-}]]+[[x_{m,r}^{+},h_{m-1,-1}],x_{m,N-r+1}^{-}])\\&=[x_{m,r-1}^{+},x_{m,N-r+1}^{-}]=\frac{1}{q_{m}-q_{m}^{-1}}\phi_{m,N}^{+}.
\end{align*}
Hence, $[x_{i,r}^{+},x_{i,N-r}^{-}]=\frac{1}{q_{i}-q_{i}^{-1}}\phi_{i,N}^{+}$ for $r>N$. The case $r\leq -1$ is similar.

\textbf{Case $N\in\mathbb{Z}_{\leq0}$.} We separate the finite gap from the negative tail; this avoids using the positive induction statements $(iv_N)$--$(v_N)$ outside their range.

First suppose $l_i<N<0$ and put $C_N(r):=[x_{i,r}^{+},x_{i,N-r}^{-}]$. We descend on $N$, starting from $N=0$. The case $N=0$ is Lemma~\ref{eq:k+l=0,1,2}, and its right-hand side commutes with $h_{i,-1}$; at the isotropic vertex it also commutes with $h_{m-1,-1}$. If the assertion has been proved at total mode $N+1$, then throughout the gap its value is either the $N+1=0$ Cartan boundary term or $0$. Consequently, for $i\neq m$,
\[
[s_iA_{i,i}]_{q_i}\bigl(C_N(r)-C_N(r+1)\bigr)=[h_{i,-1},C_{N+1}(r+1)]=0,
\]
while for $i=m$ the same computation with $h_{m-1,-1}$ gives
\[
[s_{m-1}A_{m-1,m}]_{q_{m-1}}\bigl(C_N(r)-C_N(r+1)\bigr)=0.
\]
Thus $C_N(r)$ is independent of $r$. The two generator intervals have sum interval $[l_i-1,1]$, so relation~\eqref{SQASrelation551} supplies a defining pair with total $N$ and gives $C_N(r)=0$. This proves the middle line of \eqref{v17}. The boundary $N=l_i$ is treated by the same propagation. If $l_i<-1$, it starts from the zero relation at total mode $l_i+1$; if $l_i=-1$, it starts from the $N=0$ value $\phi_{i,0}^{+}/(q_i-q_i^{-1})$, which commutes with the relevant $h_{-1}$. Hence $C_{l_i}(r)$ is independent of $r$, and the defining pair in \eqref{SQASrelation551} gives
\[
C_{l_i}(r)=-\frac{1}{q_i-q_i^{-1}}\phi_{i,l_i}^{-}.
\]
For $l_i=0$ this is already the $N=0$ case.

It remains to treat $N<l_i$. We use a formal reversal of the rank-one calculation above. Keep the notation $l_{a,i}=\alpha_i(\mu_a)$ and consider the same rank-one calculation in $\mathcal E_{0}^{\mu_2,\mu_1}$. Its finite positive and negative generator ranges are $[l_{1,i}-1,0]$ and $[l_{2,i},1]$, respectively. On words in the $i$-th root generators, the two leading $i$-th Cartan coefficients, and the Cartan modes used to propagate the $i$-th root modes, define the coefficient-field semilinear, parity-preserving anti-correspondence $\mathscr R_i$ by
\[
\mathscr R_i(f(q))=f(q^{-1}),\qquad \mathscr R_i(ab)=\mathscr R_i(b)\mathscr R_i(a),
\]
and
\[
\mathscr R_i(x_{i,r}^{+})=x_{i,l_{1,i}-r}^{-},\quad
\mathscr R_i(x_{i,s}^{-})=x_{i,l_{2,i}-s}^{+},\quad
\mathscr R_i(\phi_{i,0}^{+})=\phi_{i,l_i}^{-},\quad
\mathscr R_i(\phi_{i,l_i}^{-})=\phi_{i,0}^{+},
\]
together with $\mathscr R_i(h_{i,\pm1})=h_{i,\mp1}$ for every $i$.  At the isotropic vertex we also set $\mathscr R_m(h_{m-1,\pm1})=h_{m-1,\mp1}$, because $h_{m-1,\pm1}$ is the neighboring even Cartan mode used in the recursion of the $m$-th root modes.

We record the relations needed below. The source positive range $[l_{1,i}-1,0]$ is sent to the target negative range $[l_{1,i},1]$, and the source negative range $[l_{2,i},1]$ is sent to the target positive range $[l_{2,i}-1,0]$. Since $[k]_{q_i^{-1}}=[k]_{q_i}$, the Cartan--root relations are carried into their opposite-sign counterparts. The same-node quadratic relation (or $[x_{m,r}^{\pm},x_{m,s}^{\pm}]=0$ at the isotropic vertex) is taken to the corresponding relation with the root-current sign $+\leftrightarrow-$. In the mixed boundary relation, the total mode $M$ is changed to $l_i-M$; the five cases $M=1,l_i-1,0,l_i$ and $l_i<M<0$ are paired exactly, while $q_i-q_i^{-1}$ changes sign under $q\mapsto q^{-1}$. Using $\mathscr R_i([u,v])=[\mathscr R_i(v),\mathscr R_i(u)]$, the two relations in~\eqref{SQASrelation881} are carried to $-1$ times one another, whereas the two odd-node relations in~\eqref{SQASrelation991} are carried to one another exactly. Finally, the root-mode recursions \eqref{SQASArelation1}--\eqref{SQASArelation4} are carried into one another exactly. Thus every rank-one identity used in the positive-tail argument is transported by $\mathscr R_i$. No Serre relation is used in this rank-one reversal, and no image of a higher Cartan coefficient is assumed here.

Now fix $a>0$. The positive-tail result, already proved in $\mathcal E_{0}^{\mu_2,\mu_1}$, gives for every $r\in\mathbb Z$
\[
[x_{i,r}^{+},x_{i,a-r}^{-}]=\frac{1}{q_i-q_i^{-1}}\phi_{i,a}^{+}.
\]
Here $\mathscr R_i(\phi_{i,a}^{+})$ means the image of its defining word in \eqref{SQASArelation5}. Applying $\mathscr R_i$ gives in $\mathcal E_{0}^{\mu_1,\mu_2}$
\[
[x_{i,l_{2,i}-a+r}^{+},x_{i,l_{1,i}-r}^{-}]
=-\frac{1}{q_i-q_i^{-1}}\mathscr R_i(\phi_{i,a}^{+}).
\]
Set $N=l_i-a<l_i$. Taking $r=0$ and using the defining formula \eqref{SQASArelation6} gives directly $\mathscr R_i(\phi_{i,a}^{+})=\phi_{i,N}^{-}$. Since $r$ is arbitrary, the first index $l_{2,i}-a+r$ ranges over all integers, and therefore
\[
[x_{i,r}^{+},x_{i,N-r}^{-}]
=-\frac{1}{q_i-q_i^{-1}}\phi_{i,N}^{-}\qquad(N<l_i,\ r\in\mathbb Z).
\]
This proves the second line of \eqref{v17} without using any negative-tail Cartan relation in advance and completes the $N\leq0$ case.

\textbf{The relation \eqref{SQASrelation1} for $i=j$:} We need to prove $[\phi_{i,r}^{+},\phi_{i,m_{i}^{+}}^{+}]=[\phi_{i,r}^{+},\phi_{i,-m_{i}^{-}}^{-}]=0$ for any $r,m_{i}^{+}\in\mathbb{Z}_{\geq0}$ and $m_{i}^{-}\in\mathbb{Z}_{\geq-l_{i}}$. The cases $r=0$, $m_{i}^{+}=0$, and $m_{i}^{-}=-l_{i}$ are clear. It remains to show that $[h_{i,r},\phi_{i,m_{i}^{+}}^{+}]=0$ and $[h_{i,r},\phi_{i,-m_{i}^{-}}^{-}]=0$ for $r,m_{i}^{+}\in\mathbb{Z}_{> 0}$ and $m_{i}^{-}\in\mathbb{Z}_{>-l_{i}}$.

For $r,m_{i}^{+}\in\mathbb{Z}_{> 0}$, by \eqref{v13}, we obtain
\begin{align*}
[h_{i,r},\phi_{i,m_{i}^{+}}^{+}]&=(q_{i}-q_{i}^{-1})[h_{i,r},[x_{i,m_{i}^{+}-1}^{+},x_{i,1}^{-}]]\\&=(q_{i}-q_{i}^{-1})([[h_{i,r},x_{i,m_{i}^{+}-1}^{+}],x_{i,1}^{-}]+[x_{i,m_{i}^{+}-1}^{+},[h_{i,r},x_{i,1}^{-}]])\\&=(q_{i}-q_{i}^{-1})\frac{[rs_{i}A_{i,i}]_{q_{i}}}{r}([x_{i,m_{i}^{+}+r-1}^{+},x_{i,1}^{-}]-[x_{i,m_{i}^{+}-1}^{+},x_{i,r+1}^{-}])=0.
\end{align*}

For $r\in\mathbb{Z}_{> 0}$ and $m_{i}^{-}\in\mathbb{Z}_{>-l_{i}}$, using \eqref{v13} and relation \eqref{SQASrelation5} for $i=j$, which was proved before, we obtain
\begin{align*}
[h_{i,r},\phi_{i,-m_{i}^{-}}^{-}]&=-(q_{i}-q_{i}^{-1})[h_{i,r},[x_{i,-m_{i}^{-}}^{+},x_{i,0}^{-}]]\\&=-(q_{i}-q_{i}^{-1})([[h_{i,r},x_{i,-m_{i}^{-}}^{+}],x_{i,0}^{-}]+[x_{i,-m_{i}^{-}}^{+},[h_{i,r},x_{i,0}^{-}]])\\&=-(q_{i}-q_{i}^{-1})\frac{[rs_{i}A_{i,i}]_{q_{i}}}{r}([x_{i,-m_{i}^{-}+r}^{+},x_{i,0}^{-}]-[x_{i,-m_{i}^{-}}^{+},x_{i,r}^{-}])=0.
\end{align*}

\textbf{The relation \eqref{SQASrelation5} for $i\neq j$:} We need to prove $[x_{i,r}^{+},x_{j,s}^{-}]=0$ for any $r,s\in\mathbb{Z}$. Note that  $[x_{i,0}^{+},x_{j,0}^{-}]=0$ by \eqref{SQASrelation551}. If neither $i$ nor $j$ is equal to $m$, then by \eqref{v1}--\eqref{v2},
\begin{align*}
[h_{i,\pm1},[x_{i,r}^{+},x_{j,s}^{-}]]=[s_{i}A_{i,i}]_{q_{i}}[x_{i,r\pm 1}^{+},x_{j,s}^{-}]-[s_{i}A_{i,j}]_{q_{i}}[x_{i,r}^{+},x_{j,s\pm 1}^{-}],\\
[h_{j,\pm 1},[x_{i,r}^{+},x_{j,s}^{-}]]=[s_{j}A_{j,i}]_{q_{j}}[x_{i,r\pm 1}^{+},x_{j,s}^{-}]-[s_{j}A_{j,j}]_{q_{j}}[x_{i,r}^{+},x_{j,s\pm 1}^{-}].
\end{align*}
Combining these identities with the nondegeneracy of the matrix
\[
\begin{bmatrix}
[s_{i}A_{i,i}]_{q_{i}} & -[s_{i}A_{i,j}]_{q_{i}}\\
[s_{j}A_{j,i}]_{q_{j}} &-[s_{j}A_{j,j}]_{q_{j}}
\end{bmatrix},
\]
we obtain $[x_{i,r}^{+},x_{j,s}^{-}]=0$ for all $r,s\in\mathbb{Z}$ in this case.

If $i=m$, choose
\[
u:=\begin{cases}m+1,&j=m-1,\\ m-1,&j\neq m-1.\end{cases}
\]
Then $u$ is an even Cartan index distinct from $j$. By \eqref{v1}--\eqref{v2}, for either sign,
\begin{align*}
[h_{u,\pm1},[x_{m,r}^{+},x_{j,s}^{-}]]&=[s_uA_{u,m}]_{q_u}[x_{m,r\pm1}^{+},x_{j,s}^{-}]-[s_uA_{u,j}]_{q_u}[x_{m,r}^{+},x_{j,s\pm1}^{-}],\\
[h_{j,\pm1},[x_{m,r}^{+},x_{j,s}^{-}]]&=[s_jA_{j,m}]_{q_j}[x_{m,r\pm1}^{+},x_{j,s}^{-}]-[s_jA_{j,j}]_{q_j}[x_{m,r}^{+},x_{j,s\pm1}^{-}].
\end{align*}
The corresponding $2\times2$ coefficient matrix is nondegenerate; direct inspection gives determinant $[2]_q$ up to sign.  Starting from $[x_{m,0}^{+},x_{j,0}^{-}]=0$, the two signs therefore shift the two mode indices independently.

If $j=m$, choose $u=m+1$ when $i=m-1$ and $u=m-1$ otherwise.  Using the rows $u$ and $i$ gives
\[
\begin{bmatrix}
[s_uA_{u,i}]_{q_u}&-[s_uA_{u,m}]_{q_u}\\
[s_iA_{i,i}]_{q_i}&-[s_iA_{i,m}]_{q_i}
\end{bmatrix},
\]
whose determinant is $-[2]_q\neq0$.  The same two-sign propagation from $[x_{i,0}^{+},x_{m,0}^{-}]=0$ proves the result.  Hence $[x_{i,r}^{+},x_{j,s}^{-}]=0$ for all $i\neq j$ and all $r,s\in\mathbb Z$.

\textbf{The relation \eqref{SQASrelation4} for $i\neq j$:} We prove the case $\epsilon=+$; the case $\epsilon=-$ is analogous.
Let $Y^{+}(i,j;r,s)=[\phi_{i,r+1}^{+},x_{j,s}^{+}]_{q^{A_{i,j}}}+[x_{j,s+1}^{+},\phi_{i,r}^{+}]_{q^{A_{i,j}}}$. Using the relation \eqref{SQASrelation5}, which has been proved for both $i\neq j$ and $i=j$, together with $X^{+}(i,j;r,s)=0$, we obtain
\begin{align*}
[[x_{i,r+1}^{+},x_{j,s}^{+}]_{q^{A_{i,j}}},x_{i,0}^{-}]&=[x_{i,r+1}^{+},[x_{j,s}^{+},x_{i,0}^{-}]]_{q^{A_{i,j}}}+[[x_{i,r+1}^{+},x_{i,0}^{-}],x_{j,s}^{+}]_{q^{A_{i,j}}}\\&=\frac{1}{q_{i}-q_{i}^{-1}}[\phi_{i,r+1}^{+},x_{j,s}^{+}]_{q^{A_{i,j}}},\\
[[x_{j,s+1}^{+},x_{i,r}^{+}]_{q^{A_{i,j}}},x_{i,0}^{-}]&=[x_{j,s+1}^{+},[x_{i,r}^{+},x_{i,0}^{-}]]_{q^{A_{i,j}}}+[[x_{j,s+1}^{+},x_{i,0}^{-}],x_{i,r}^{+}]_{q^{A_{i,j}}}\\&=\frac{1}{q_{i}-q_{i}^{-1}}[x_{j,s+1}^{+},\phi_{i,r}^{+}]_{q^{A_{i,j}}},
\end{align*}
Thus,
\begin{align*}
Y^{+}(i,j;r,s)&=(q_{i}-q_{i}^{-1})([[x_{i,r+1}^{+},x_{j,s}^{+}]_{q^{A_{i,j}}},x_{i,0}^{-}]+[[x_{j,s+1}^{+},x_{i,r}^{+}]_{q^{A_{i,j}}},x_{i,0}^{-}])\\&=(q_{i}-q_{i}^{-1})[X^{+}(i,j;r,s),x_{i,0}^{-}]=0.
\end{align*}

\textbf{The relation \eqref{SQASrelation1} for $i\neq j$:} We prove the case $\epsilon=\epsilon'=+$; the other cases are analogous.

It suffices to prove that $[h_{i,r},\phi_{j,s}^{+}]=0$ for $r,s\in\mathbb{Z}_{> 0}$, where $h_{i,r}$ is defined in \eqref{a3}. Analogously to the derivation of \eqref{v13}--\eqref{v13'}, relations \eqref{SQASrelation4}--\eqref{SQASrelation5} imply
\begin{align*}
[h_{i,r},x_{j,s}^{+}]=\frac{[rs_{i}A_{i,j}]_{q_{i}}}{r}x_{j,s+r}^{+}\quad\text{and}\quad[h_{i,r},x_{j,s}^{-}]=-\frac{[rs_{i}A_{i,j}]_{q_{i}}}{r}x_{j,s+r}^{-}\quad \text{for $r\in\mathbb{Z}_{>0}$, $s\in\mathbb{Z}$}.
\end{align*}
Hence,
\begin{align*}
[h_{i,r},\phi_{j,s}^{+}]=(q_{j}-q_{j}^{-1})[h_{i,r},[x_{j,s}^{+},x_{j,0}^{-}]]& =(q_{j}-q_{j}^{-1})([[h_{i,r},x_{j,s}^{+}],x_{j,0}^{-}]+[x_{j,s}^{+},[h_{i,r},x_{j,0}^{-}]])\\&=(q_{j}-q_{j}^{-1})\frac{[rs_{i}A_{i,j}]_{q_{i}}}{r}([x_{j,s+r}^{+},x_{j,0}^{-}]-[x_{j,s}^{+},x_{j,r}^{-}])=0.
\end{align*}

\textbf{The relation \eqref{SQASrelation6}:} Fix $i\in I\setminus\{m\}$ and a neighboring vertex $j=i\pm1\in I$, and set
\[
Z^{+}_{i,j}(k_1,k_2;r):=\operatorname{Sym}_{k_1,k_2}
[x_{i,k_1}^{+},[x_{i,k_2}^{+},x_{j,r}^{+}]_{q}]_{q^{-1}}.
\]
This is symmetric in $k_1,k_2$, and \eqref{SQASrelation661} gives $Z^{+}_{i,j}(0,0;0)=0$.  By the already established Cartan--root relation \eqref{SQASrelation4}, for $a\in\{\pm1,\pm2\}$ and $u\in\{i,j\}$,
\[
[h_{u,a},Z^{+}_{i,j}(k_1,k_2;r)]
=c_{u,i}(a)P_aZ^{+}_{i,j}(k_1,k_2;r)
+c_{u,j}(a)T_aZ^{+}_{i,j}(k_1,k_2;r),
\]
where $c_{u,v}(a):=\frac{[as_uA_{u,v}]_{q_u}}{a}$
and
\begin{align*}
P_aZ(k_1,k_2;r)&:=Z(k_1+a,k_2;r)+Z(k_1,k_2+a;r),\\
T_aZ(k_1,k_2;r)&:=Z(k_1,k_2;r+a).
\end{align*}
For $a=\pm1,\pm2$, the matrix
\[
\begin{bmatrix}
c_{i,i}(a)&c_{i,j}(a)\\
c_{j,i}(a)&c_{j,j}(a)
\end{bmatrix}
\]
is nondegenerate.  If neither vertex is isotropic, its determinant is $[2]_q^2-1=[3]_q$ for $|a|=1$ and $([4]_q^2-[2]_q^2)/4$ for $|a|=2$; if $j=m$, the determinant is $\pm1$ for $|a|=1$ and $\pm[2]_q^2/4$ for $|a|=2$.  All these scalars are nonzero because $q$ is not a root of unity.  Hence the adjoint actions allow $P_a$ and $T_a$ to be applied separately to any vanishing coefficient. Repeating this from $Z^{+}_{i,j}(0,0;0)=0$, every polynomial in the commuting operators $P_{\pm1},P_{\pm2},T_{\pm1}$ annihilates the initial coefficient.  The operators $T_{\pm1}$ generate all Laurent shifts of $r$, while, as above,
\[
\mathbb C[u^{\pm1},v^{\pm1}]^{\mathfrak S_2}
=\mathbb C[p_1,p_2,p_{-1},p_{-2}]
\]
shows that $P_{\pm1},P_{\pm2}$ generate every symmetric Laurent shift of $(k_1,k_2)$.  Therefore
$Z^{+}_{i,j}(k_1,k_2;r)=0$, $k_1,k_2,r\in\mathbb Z$.
The negative-current relation is proved identically.

\textbf{The relation \eqref{SQASrelation7}:} Set
\[
Z^{+}_{m}(k_1,k_2;r_1,r_2):=
\operatorname{Sym}_{k_1,k_2}[x_{m,k_1}^{+},[x_{m+1,r_1}^{+},[x_{m,k_2}^{+},x_{m-1,r_2}^{+}]_{q}]_{q^{-1}}].
\]
It is symmetric in $k_1,k_2$, and \eqref{SQASrelation771} gives $Z^{+}_{m}(0,0;0,0)=0$.  For $a\in\{\pm1,\pm2\}$ define the three commuting shift operators
\begin{align*}
P_aZ(k_1,k_2;r_1,r_2)&:=Z(k_1+a,k_2;r_1,r_2)+Z(k_1,k_2+a;r_1,r_2),\\
T_a^{+}Z(k_1,k_2;r_1,r_2)&:=Z(k_1,k_2;r_1+a,r_2),\quad T_a^{-}Z(k_1,k_2;r_1,r_2):=Z(k_1,k_2;r_1,r_2+a).
\end{align*}
For every Cartan index $u$, relation \eqref{SQASrelation4} gives
\[
[h_{u,a},Z^{+}_{m}]
=c_{u,m}(a)P_aZ^{+}_{m}
+c_{u,m+1}(a)T_a^{+}Z^{+}_{m}
+c_{u,m-1}(a)T_a^{-}Z^{+}_{m}.
\]
Since $m,n>1$ and $m\neq n$, at least one of $m,n$ is at least $3$.  If $m\geq3$, take the three Cartan rows $u=m-2,m,m+1$; otherwise $m=2$ and hence $n\geq3$, and we take $u=m+2,m,m-1$.  In either case, for $a=\pm1,\pm2$ the resulting $3\times3$ coefficient matrix in the columns $(m,m+1,m-1)$ is nondegenerate: with the indicated ordering it is triangular up to a permutation, with determinant $-1$ for $|a|=1$ and $-[2]_q^3/8$ for $|a|=2$.  Consequently $P_a,T_a^{+},T_a^{-}$ may be separated from the three adjoint equations.

Repeating these separated shift operations from $Z^{+}_{m}(0,0;0,0)=0$, the operators $T_{\pm1}^{+}$ and $T_{\pm1}^{-}$ generate arbitrary Laurent shifts in $r_1,r_2$, while $P_{\pm1},P_{\pm2}$ generate all symmetric Laurent shifts in $k_1,k_2$. Hence
$Z^{+}_{m}(k_1,k_2;r_1,r_2)=0$, $k_1,k_2,r_1,r_2\in\mathbb Z$.
The negative-current odd Serre relation is analogous.

For completeness, the order reversal used in the negative-tail rank-one argument is also compatible with the Serre polynomials once the all-mode relations have been established. Indeed, for any parity-preserving semilinear anti-map $\mathscr R$ with $\mathscr R(q)=q^{-1}$,
\[
\mathscr R([u,v]_{q^a})=[\mathscr R(v),\mathscr R(u)]_{q^{-a}},\qquad
[v,u]_{q^{-a}}=-(-1)^{|u||v|}q^{-a}[u,v]_{q^a}.
\]
These identities carry the ordinary cubic Serre polynomial to the corresponding polynomial with the root-current sign $+\leftrightarrow-$. At the isotropic vertex, if $u$ is odd and $v,w$ are even, they give the exact identity
\[
\mathscr R\!\left([u,[v,[u,w]_q]_{q^{-1}}]\right)
=[\mathscr R(u),[\mathscr R(v),[\mathscr R(u),\mathscr R(w)]_q]_{q^{-1}}].
\]
Thus the $q\mapsto q^{-1}$ order reversal preserves the quartic odd Serre polynomial without an additional scalar or sign.
This completes the proof of Theorem~\ref{thm:finite iso}.
\end{proof}

\section{Zero-shift RTT coefficient calculations}\label{app:rtt-coproduct-proof}
This appendix derives the zero-shift RTT coefficient identities used as the model for the finite-generator formulas in Appendix~\ref{app:antidominant-coproduct-formulas}.  We continue the notation and Gauss decomposition of Subsection~\ref{subsec:rtt-gauss}.

Let us express the matrix coefficients of $\tilde{F}^{\pm}(z)$, $\tilde{G}^{\pm}(z)$, $\tilde{E}^{\pm}(z)$ as Taylor series in $z^{\mp1}$:
\begin{align*}
&\tilde{e}_{ji}^{+}(z)=\sum_{r\geq0}\tilde{e}_{ji}^{(r)}z^{-r},\quad \tilde{e}_{ji}^{-}(z)=\sum_{r<0}\tilde{e}_{ji}^{(r)}z^{-r},\\&
\tilde{f}_{ij}^{+}(z)=\sum_{r>0}\tilde{f}_{ij}^{(r)}z^{-r},\quad \tilde{f}_{ij}^{-}(z)=\sum_{r\leq0}\tilde{f}_{ij}^{(r)}z^{-r}, \quad
\tilde{g}_{i}^{\pm}(z)=\tilde{g}_{i}^{\pm}+\sum_{r>0}\tilde{g}_{i}^{(\pm r)}z^{\mp r}.
\end{align*}
From Corollary~\ref{coro: iso from Drinfeld realization to RTT realization}, we have
\begin{align}\label{equations for gamma-1}
\begin{split}
&\gamma^{-1}(\tilde{e}_{i,i+1}^{(0)})=
(q_{i}-q_{i}^{-1})x_{i,0}^{+},
\quad \gamma^{-1}(\tilde{e}_{i,i+1}^{(-1)})=-\tau_{i}^{-1}(q_{i}-q_{i}^{-1})x_{i,-1}^{+},\\&
\gamma^{-1}(\tilde{f}_{i+1,i}^{(0)})=-(q_{i}-q_{i}^{-1})x_{i,0}^{-},
\quad
\gamma^{-1}(\tilde{f}_{i+1,i}^{(1)})=\tau_{i}(q_{i}-q_{i}^{-1})x_{i,1}^{-}.
\end{split}
\end{align}

For convenience, write $[z^{r}]F(z)$ for the coefficient of $z^{r}$ in the Laurent series $F(z)$, and set $p_{ij}:=|i|+|j|\in\mathbb Z_2$. We first record the following lemma.
\begin{lemma}\label{Lemma: relations for g e f in RTT}
For any $1\leq j<i\leq m+n$ and $1\leq a,b \leq m+n$, the following identities hold:
\[
\begin{array}{l@{\qquad\qquad}l}
\text{(1)}\quad
\tilde{g}_{a}^{\epsilon}\tilde{g}_{b}^{\epsilon'}
=\tilde{g}_{b}^{\epsilon'}\tilde{g}_{a}^{\epsilon},
\quad \epsilon,\epsilon'\in\{+,-\},
&
\text{(4)}\quad
\tilde{g}_{a}^{\pm}\tilde{e}_{ji}^{(-1)}
=q_{a}^{\pm\delta_{ai}\mp\delta_{aj}}
\tilde{e}_{ji}^{(-1)}\tilde{g}_{a}^{\pm},
\\[1.2ex]
\text{(2)}\quad
\tilde{g}_{a}^{\pm}\tilde{e}_{ji}^{(0)}
=q_{a}^{\pm\delta_{ai}\mp\delta_{aj}}
\tilde{e}_{ji}^{(0)}\tilde{g}_{a}^{\pm},
&
\text{(5)}\quad
\tilde{g}_{a}^{\pm}\tilde{f}_{ij}^{(1)}
=q_{a}^{\mp\delta_{ai}\pm\delta_{aj}}
\tilde{f}_{ij}^{(1)}\tilde{g}_{a}^{\pm},
\\[1.2ex]
\text{(3)}\quad
\tilde{g}_{a}^{\pm}\tilde{f}_{ij}^{(0)}
=q_{a}^{\mp\delta_{ai}\pm\delta_{aj}}
\tilde{f}_{ij}^{(0)}\tilde{g}_{a}^{\pm}.
\end{array}
\]
\end{lemma}
\begin{proof}
Part (1) follows by taking the lowest coefficient in the diagonal--diagonal RTT relation.  For the remaining parts it is convenient to put
\[
\eta_a(j,i):=\delta_{ai}-\delta_{aj}.
\]
All four weight relations are obtained from the same coefficient extraction.  More precisely, the indicated matrix component and coefficient give
\[
\begin{array}{c|c|c|c}
\text{quantity} & \text{RTT relation} & \text{matrix component} & \text{coefficient}\\ \hline
\tilde g_a^+\tilde e_{ji}^{(0)} & TT & \mathsf E_{aa}\otimes \mathsf E_{ji} & z^1w^0\\
\tilde g_a^-\tilde f_{ij}^{(0)} & SS & \mathsf E_{aa}\otimes \mathsf E_{ij} & z^0w^1\\
\tilde g_a^+\,[z]s_{ji}(z) & TS & \mathsf E_{aa}\otimes \mathsf E_{ji} & z^1w^1\\
\tilde g_a^-\,[z^{-1}]t_{ij}(z) & TS & \mathsf E_{ij}\otimes \mathsf E_{aa} & z^0w^0.
\end{array}
\]
In each row the result is the same weight factor $q_a^{\eta_a(j,i)}$.  To illustrate the only nontrivial point, when $a=j<i$ the first extraction gives
\[
q_a\tilde g_a^+\tilde g_j^+\tilde e_{ji}^{(0)}
 =\tilde g_j^+\tilde e_{ji}^{(0)}\tilde g_a^+,
\]
and part (1) yields
$\tilde g_a^+\tilde e_{ji}^{(0)}=q_a^{-1}\tilde e_{ji}^{(0)}\tilde g_a^+$.
For $a=i$ the same extraction gives the factor $q_a$, while for
$a\notin\{j,i\}$ it gives factor $1$; these are exactly the five possible relative positions of $a,j,i$.

This proves (2) and, from the second row, (3).  For (4), write the Gauss coefficient as
\[
[z]s_{ji}(z)=\tilde g_j^-\tilde e_{ji}^{(-1)}
 +\sum_{j'<j}\epsilon_{jij'}\tilde f_{jj'}^{(0)}\tilde g_{j'}^-\tilde e_{j'i}^{(-1)}.
\]
The third row gives weight $q_a^{\eta_a(j,i)}$ for the left-hand side.  Using (1), (3), and induction on $j$, every summand with $j'<j$ already has this weight, so subtraction leaves the asserted relation for $\tilde e_{ji}^{(-1)}$.  The fourth row and the expansion
\[
[z^{-1}]t_{ij}(z)=\tilde f_{ij}^{(1)}\tilde g_j^+
 +\sum_{j'<j}\epsilon_{ijj'}\tilde f_{ij'}^{(1)}\tilde g_{j'}^+\tilde e_{j'j}^{(0)}
\]
give (5) by the same induction.  Thus no separate calculation for the five relative positions is needed beyond the endpoint coefficient displayed above.
\end{proof}

We next record recursive formulas for the first Gauss modes $e_{ji}^{(0)},e_{ji}^{(-1)},f_{ij}^{(0)},f_{ij}^{(1)}$ in terms of the simple Gauss generators.
\begin{proposition}\label{prop: formulas for eji0,eji-1,fij0,fij1}
For type A with the standard parity sequence, assume $m,n>1$ and $m\ne n$.
For any $1\leq j<k<i\leq m+n$, the following identities hold:
\[
\begin{array}{l@{\qquad\qquad}l}
\text{(1)}\quad
\tilde{e}_{ji}^{(0)}
=\frac{1}{q_k-q_k^{-1}}
[\tilde{e}_{ki}^{(0)},\tilde{e}_{jk}^{(0)}]_{q_k^{-1}},
&
\text{(3)}\quad
\tilde{e}_{ji}^{(-1)}
=\frac{1}{q_k-q_k^{-1}}
[\tilde{e}_{ki}^{(0)},\tilde{e}_{jk}^{(-1)}]_{q_k^{-1}},
\\[1.2ex]
\text{(2)}\quad
\tilde{f}_{ij}^{(0)}
=\frac{-1}{q_k-q_k^{-1}}
[\tilde{f}_{kj}^{(0)},\tilde{f}_{ik}^{(0)}]_{q_k},
&
\text{(4)}\quad
\tilde{f}_{ij}^{(1)}
=\frac{-1}{q_k-q_k^{-1}}
[\tilde{f}_{kj}^{(1)},\tilde{f}_{ik}^{(0)}]_{q_k}.
\end{array}
\]
\end{proposition}
\begin{proof}
For (1), take the $\mathsf E_{jk}\otimes \mathsf E_{ki}$ component of the $TT$ relation and then the coefficient of $z^1w^0$.  This gives
\[
\tilde g_j^+\tilde e_{jk}^{(0)}\tilde g_k^+\tilde e_{ki}^{(0)}
 +(q_k-q_k^{-1})\tilde g_k^+\tilde g_j^+\tilde e_{ji}^{(0)}
 =\tilde g_k^+\tilde e_{ki}^{(0)}\tilde g_j^+\tilde e_{jk}^{(0)}.
\]
Using Lemma~\ref{Lemma: relations for g e f in RTT} and cancelling the invertible diagonal factors yields (1).

The other three formulas come from the same extraction with the following choices:
\[
\begin{array}{c|c|c|c}
 & \text{RTT relation} & \text{matrix component} & \text{coefficient}\\ \hline
(2) & SS & \mathsf E_{ik}\otimes \mathsf E_{kj} & z^0w^1\\
(3) & ST & \mathsf E_{jk}\otimes \mathsf E_{ki} & z^1w^1\\
(4) & ST & \mathsf E_{ik}\otimes \mathsf E_{kj} & z^0w^0.
\end{array}
\]
For (2), substitution of the degree-zero Gauss coefficients followed by Lemma~\ref{Lemma: relations for g e f in RTT} gives the stated $q_k$-commutator immediately.  For (3), the coefficient of $s_{jk}(z)$ is
\[
[z]s_{jk}(z)=\tilde g_j^-\tilde e_{jk}^{(-1)}
 +\sum_{j'<j}\epsilon_{jkj'}\tilde f_{jj'}^{(0)}\tilde g_{j'}^-\tilde e_{j'k}^{(-1)}.
\]
After inserting this expression, the terms with $j'<j$ satisfy the desired identity by induction on $j$.  Here
$\epsilon_{jkj'}=\epsilon_{jij'}$: if $j$ and $j'$ have the same vector parity this is immediate from the definition of $\epsilon$, while if they straddle the unique parity wall then $j<k<i$ forces $k$ and $i$ to have the same vector parity.  Moreover
$[\tilde e_{ki}^{(0)},\tilde f_{jj'}^{(0)}]=0$ for $j'<j<k<i$, because the two interval roots have disjoint supports.  Thus the graded sign of every lower summand is exactly the inductive one.
Subtracting these lower terms leaves (3).  Formula (4) is identical, using
\[
[z^{-1}]t_{kj}(z)=\tilde f_{kj}^{(1)}\tilde g_j^+
 +\sum_{j'<j}\epsilon_{kjj'}\tilde f_{kj'}^{(1)}\tilde g_{j'}^+\tilde e_{j'j}^{(0)},
\]
together with
$\epsilon_{kjj'}=\epsilon_{ijj'}$ (by the same parity-wall argument) and
$[\tilde e_{j'j}^{(0)},\tilde f_{ik}^{(0)}]=0$ because the corresponding intervals are disjoint.  This proves (2)--(4).
\end{proof}

The following formulas follow from Proposition~\ref{prop: formulas for eji0,eji-1,fij0,fij1}.
\begin{corollary}\label{coro: equations for eji0,eji-1,fij0,fij1}
For type A with the standard parity sequence, assume $m,n>1$ and $m\ne n$.
For any $1\leq j<i\leq m+n$, we have
\begin{align}\label{eq:equivalent formulas for eji0,fij0,eji-1,fij1}
\begin{split}
\tilde{e}_{ji}^{(0)}&=\Big(\prod_{k=j+1}^{i-1}\frac{1}{q_{k}-q_{k}^{-1}}\Big)[\tilde{e}_{i-1,i}^{(0)},[\tilde{e}_{i-2,i-1}^{(0)},\cdots,[\tilde{e}_{j+1,j+2}^{(0)},\tilde{e}_{j,j+1}^{(0)}]_{q_{j+1}^{-1}}\cdots]_{q_{i-2}^{-1}}]_{q_{i-1}^{-1}}
\\&=\Big(\prod_{k=j+1}^{i-1}\frac{1}{q_{k}-q_{k}^{-1}}\Big)[[\cdots[\tilde{e}_{i-1,i}^{(0)},\tilde{e}_{i-2,i-1}^{(0)}]_{q_{i-1}^{-1}},\cdots,\tilde{e}_{j+1,j+2}^{(0)}]_{q_{j+2}^{-1}},\tilde{e}_{j,j+1}^{(0)}]_{q_{j+1}^{-1}},
\end{split}\\
\begin{split}
\tilde{f}_{ij}^{(0)}&=\Big(\prod_{k=j+1}^{i-1}\frac{1}{q_{k}^{-1}-q_{k}}\Big)[\tilde{f}_{j+1,j}^{(0)},[\tilde{f}_{j+2,j+1}^{(0)},\cdots,[\tilde{f}_{i-1,i-2}^{(0)},\tilde{f}_{i,i-1}^{(0)}]_{q_{i-1}}\cdots]_{q_{j+2}}]_{q_{j+1}}
\\&=\Big(\prod_{k=j+1}^{i-1}\frac{1}{q_{k}^{-1}-q_{k}}\Big)[[\cdots[\tilde{f}_{j+1,j}^{(0)},\tilde{f}_{j+2,j+1}^{(0)}]_{q_{j+1}},\cdots,\tilde{f}_{i-1,i-2}^{(0)}]_{q_{i-2}},\tilde{f}_{i,i-1}^{(0)}]_{q_{i-1}},
\end{split}\\
\begin{split}
\tilde{e}_{ji}^{(-1)}&=\Big(\prod_{k=j+1}^{i-1}\frac{1}{q_{k}-q_{k}^{-1}}\Big)[\tilde{e}_{i-1,i}^{(0)},[\tilde{e}_{i-2,i-1}^{(0)},\cdots,[\tilde{e}_{j+1,j+2}^{(0)},\tilde{e}_{j,j+1}^{(-1)}]_{q_{j+1}^{-1}}\cdots]_{q_{i-2}^{-1}}]_{q_{i-1}^{-1}}
\\&=\Big(\prod_{k=j+1}^{i-1}\frac{1}{q_{k}-q_{k}^{-1}}\Big)[[\cdots[\tilde{e}_{i-1,i}^{(0)},\tilde{e}_{i-2,i-1}^{(0)}]_{q_{i-1}^{-1}},\cdots,\tilde{e}_{j+1,j+2}^{(0)}]_{q_{j+2}^{-1}},\tilde{e}_{j,j+1}^{(-1)}]_{q_{j+1}^{-1}},
\end{split}\\
\begin{split}
\tilde{f}_{ij}^{(1)}&=\Big(\prod_{k=j+1}^{i-1}\frac{1}{q_{k}^{-1}-q_{k}}\Big)[\tilde{f}_{j+1,j}^{(1)},[\tilde{f}_{j+2,j+1}^{(0)},\cdots,[\tilde{f}_{i-1,i-2}^{(0)},\tilde{f}_{i,i-1}^{(0)}]_{q_{i-1}}\cdots]_{q_{j+2}}]_{q_{j+1}}
\\&=\Big(\prod_{k=j+1}^{i-1}\frac{1}{q_{k}^{-1}-q_{k}}\Big)[[\cdots[\tilde{f}_{j+1,j}^{(1)},\tilde{f}_{j+2,j+1}^{(0)}]_{q_{j+1}},\cdots,\tilde{f}_{i-1,i-2}^{(0)}]_{q_{i-2}},\tilde{f}_{i,i-1}^{(0)}]_{q_{i-1}}.
\end{split}
\end{align}
\end{corollary}

For the translation of Gauss coefficients into Drinfeld generators, we use the following root-vector notation.  For \(1\le j<i\le M\), set
\begin{align}\label{eq: Eji and Fij}
&E_{j,j+1}^{(0)}=\widetilde E_{j,j+1}^{(0)}=x_{j,0}^{+},
\qquad E_{j,j+1}^{(-1)}=x_{j,-1}^{+},\notag\\
&F_{j+1,j}^{(0)}=\widetilde F_{j+1,j}^{(0)}=x_{j,0}^{-},
\qquad F_{j+1,j}^{(1)}=x_{j,1}^{-},
\end{align}
and, for \(i\ge j+2\), define recursively
\begin{align*}
E_{ji}^{(0)}&=[x_{i-1,0}^{+},E_{j,i-1}^{(0)}]_{q_{i-1}^{-1}},
&
\widetilde E_{ji}^{(0)}&=[x_{i-1,0}^{+},\widetilde E_{j,i-1}^{(0)}]_{q_{i-1}},\\
E_{ji}^{(-1)}&=[x_{i-1,0}^{+},E_{j,i-1}^{(-1)}]_{q_{i-1}^{-1}},
&
F_{ij}^{(0)}&=[x_{j,0}^{-},F_{i,j+1}^{(0)}]_{q_{j+1}},\\
\widetilde F_{ij}^{(0)}&=[x_{j,0}^{-},\widetilde F_{i,j+1}^{(0)}]_{q_{j+1}^{-1}},
&
F_{ij}^{(1)}&=[x_{j,1}^{-},F_{i,j+1}^{(0)}]_{q_{j+1}}.
\end{align*}
These are the nested root vectors used in the formulas below.  Comparing these recursive definitions with Corollary~\ref{coro: equations for eji0,eji-1,fij0,fij1} and using the simple-current identifications in the Ding--Frenkel isomorphism gives the following translation.

\begin{corollary}\label{eq:from matrix elements in RTT to drinfeld realization}
For type A with the standard parity sequence, assume $m,n>1$ and $m\ne n$.  For $1\leq j<i\leq m+n$, we have
\begin{align*}
&\gamma^{-1}(\tilde{e}_{ji}^{(0)})=(q_{j}-q_{j}^{-1})E_{ji}^{(0)},\quad
\gamma^{-1}(\tilde{e}_{ji}^{(-1)})=-\tau_{j}^{-1}(q_{j}-q_{j}^{-1})E_{ji}^{(-1)},\\
&\gamma^{-1}(\tilde{f}_{ij}^{(0)})=-(q_{j}-q_{j}^{-1})F_{ij}^{(0)},\quad
\gamma^{-1}(\tilde{f}_{ij}^{(1)})=\tau_{j}(q_{j}-q_{j}^{-1})F_{ij}^{(1)}.
\end{align*}
\end{corollary}

Let
\begin{align}\label{def: Aji+,Aij-}
\begin{split}
&A_{ji}^{+}:=\sum_{s\geq1}\sum_{j=j_{1}<\cdots<j_{s+1}=i}(-1)^{s-1}\tilde{e}_{j_{1}j_{2}}^{(0)}\cdots \tilde{e}_{j_{s}j_{s+1}}^{(0)},\\&
A_{ij}^{-}:=\sum_{s\geq1}\sum_{j=j_{1}<\cdots<j_{s+1}=i}(-1)^{s-1}\tilde{f}_{j_{s+1}j_{s}}^{(0)}\cdots \tilde{f}_{j_{2}j_{1}}^{(0)},
\end{split}
\end{align}
for $1\leq j<i\leq m+n$.
The alternating-path elements satisfy
\begin{align}
\begin{split}\label{equation for A+}
A_{ji}^{+}&=\Big(\prod_{k=j+1}^{i-1}\frac{1}{q_{k}-q_{k}^{-1}}\Big)[\tilde{e}_{i-1,i}^{(0)},[\tilde{e}_{i-2,i-1}^{(0)},\cdots,[\tilde{e}_{j+1,j+2}^{(0)},\tilde{e}_{j,j+1}^{(0)}]_{q_{j+1}}\cdots]_{q_{i-2}}]_{q_{i-1}}\\&
=\Big(\prod_{k=j+1}^{i-1}\frac{1}{q_{k}-q_{k}^{-1}}\Big)[[\cdots[\tilde{e}_{i-1,i}^{(0)},\tilde{e}_{i-2,i-1}^{(0)}]_{q_{i-1}},\cdots,\tilde{e}_{j+1,j+2}^{(0)}]_{q_{j+2}},\tilde{e}_{j,j+1}^{(0)}]_{q_{j+1}},
\end{split}\\
\begin{split}\label{equation for A-}
A_{ij}^{-}&=\Big(\prod_{k=j+1}^{i-1}\frac{1}{q_{k}^{-1}-q_{k}}\Big)[\tilde{f}_{j+1,j}^{(0)},[\tilde{f}_{j+2,j+1}^{(0)},\cdots,[\tilde{f}_{i-1,i-2}^{(0)},\tilde{f}_{i,i-1}^{(0)}]_{q_{i-1}^{-1}}\cdots]_{q_{j+2}^{-1}}]_{q_{j+1}^{-1}}\\&=\Big(\prod_{k=j+1}^{i-1}\frac{1}{q_{k}^{-1}-q_{k}}\Big)[[\cdots[\tilde{f}_{j+1,j}^{(0)},\tilde{f}_{j+2,j+1}^{(0)}]_{q_{j+1}^{-1}},\cdots,\tilde{f}_{i-1,i-2}^{(0)}]_{q_{i-2}^{-1}},\tilde{f}_{i,i-1}^{(0)}]_{q_{i-1}^{-1}}.
\end{split}
\end{align}
We prove the two formulas for $A_{ji}^{+}$; the formulas for $A_{ij}^{-}$ are obtained in the same way.  The alternating path definition gives the two recursions
\[
A_{ji}^{+}=\tilde e_{ji}^{(0)}-\sum_{j<k<i}\tilde e_{jk}^{(0)}A_{ki}^{+}
          =\tilde e_{ji}^{(0)}-\sum_{j<k<i}A_{jk}^{+}\tilde e_{ki}^{(0)}.
\]
Using Proposition~\ref{prop: formulas for eji0,eji-1,fij0,fij1} in the first recursion and induction on $i-j$, the terms with an intermediate vertex cancel in pairs; generators supported away from the last edge commute.  What remains is
\[
(q_{i-1}-q_{i-1}^{-1})A_{ji}^{+}
   =[\tilde e_{i-1,i}^{(0)},A_{j,i-1}^{+}]_{q_{i-1}}.
\]
Starting with $A_{j,j+1}^{+}=\tilde e_{j,j+1}^{(0)}$ and iterating this identity gives the first expression in \eqref{equation for A+}.  Applying the same argument to the second recursion gives
\[
(q_{j+1}-q_{j+1}^{-1})A_{ji}^{+}
   =[A_{j+1,i}^{+},\tilde e_{j,j+1}^{(0)}]_{q_{j+1}},
\]
which yields the second expression.  This proves \eqref{equation for A+}; the proof of \eqref{equation for A-} is identical with the negative root vectors and the corresponding relations in Proposition~\ref{prop: formulas for eji0,eji-1,fij0,fij1}.

Combining these identities with \eqref{equations for gamma-1}, we obtain the following translation.

\begin{corollary}\label{eq:from matrix elements in RTT to drinfeld realization2}
For type A with the standard parity sequence, assume $m,n>1$ and $m\ne n$.  For $1\leq j<i\leq m+n$, we have
\[
\gamma^{-1}(A_{ji}^{+})=(q_{j}-q_{j}^{-1})\widetilde{E}_{ji}^{(0)},\qquad
\gamma^{-1}(A_{ij}^{-})=-(q_{j}-q_{j}^{-1})\widetilde{F}_{ij}^{(0)}.
\]
\end{corollary}

To derive the explicit formula for $\Delta^{\mathrm{rtt}}(\tilde{g}_{i}^{(\pm1)})$, we first establish the following lemmas.
\begin{lemma}\label{RTT coproduct for eji0, fij0}
For $1\leq j<i\leq m+n$, we have
\begin{align}
\begin{split}\label{RTT coproduct for eij0}
&\Delta^{\mathrm{rtt}}(\tilde{e}_{ji}^{(0)})=1\otimes \tilde{e}_{ji}^{(0)}+\tilde{e}_{ji}^{(0)}\otimes (\tilde{g}_{j}^{+})^{-1}\tilde{g}_{i}^{+}+\sum_{j<a<i}\tilde{e}_{ja}^{(0)}\otimes (\tilde{g}_{j}^{+})^{-1}\tilde{g}_{a}^{+}\tilde{e}_{ai}^{(0)},
\end{split}\\
\begin{split}\label{RTT coproduct for fij0}
\Delta^{\mathrm{rtt}}(\tilde{f}_{ij}^{(0)})=\tilde{f}_{ij}^{(0)}\otimes 1+\tilde{g}_{i}^{-}(\tilde{g}_{j}^{-})^{-1}\otimes \tilde{f}_{ij}^{(0)}+\sum_{j<a<i}\tilde{f}_{ia}^{(0)}\tilde{g}_{a}^{-}(\tilde{g}_{j}^{-})^{-1}\otimes \tilde{f}_{aj}^{(0)}.
\end{split}
\end{align}
\end{lemma}
\begin{proof}
We prove \eqref{RTT coproduct for eij0}; the proof of \eqref{RTT coproduct for fij0} is similar. Since $\tilde{g}_{j}^{+}=[z^{0}]t_{jj}(z)$, we have
\begin{align*}
\Delta^{\mathrm{rtt}}(\tilde{g}_{j}^{+})=[z^{0}]\Big(\sum_{a=1}^{m+n}\epsilon_{jja} t_{ja}(z)\otimes t_{aj}(z)\Big)=[z^{0}](\epsilon_{jjj}t_{jj}(z)\otimes t_{jj}(z))=\tilde{g}_{j}^{+}\otimes \tilde{g}_{j}^{+}.
\end{align*}
 Since $\tilde{g}_{j}^{+}\tilde{e}_{ji}^{(0)}=[z^{0}]t_{ji}(z)$, we have
 \begin{align*}
\Delta^{\mathrm{rtt}}(\tilde{g}_{j}^{+}\tilde{e}_{ji}^{(0)})&=[z^{0}]\Big(\sum_{a=1}^{m+n}\epsilon_{jia}t_{ja}(z)\otimes t_{ai}(z)\Big)
\\&=[z^{0}](\epsilon_{jij}t_{jj}(z)\otimes t_{ji}(z)+\epsilon_{jii}t_{ji}(z)\otimes t_{ii}(z)+\sum_{j<a<i}\epsilon_{jia}t_{ja}(z)\otimes t_{ai}(z))
\\&=\epsilon_{jij}\tilde{g}_{j}^{+}\otimes\tilde{g}_{j}^{+}\tilde{e}_{ji}^{(0)}+\epsilon_{ji i}\tilde{g}_{j}^{+}\tilde{e}_{ji}^{(0)}\otimes \tilde{g}_{i}^{+}+\sum_{j<a<i}\epsilon_{jia}\tilde{g}_{j}^{+}\tilde{e}_{ja}^{(0)}\otimes \tilde{g}_{a}^{+}\tilde{e}_{ai}^{(0)}
\\&=\tilde{g}_{j}^{+}\otimes\tilde{g}_{j}^{+}\tilde{e}_{ji}^{(0)}+\tilde{g}_{j}^{+}\tilde{e}_{ji}^{(0)}\otimes \tilde{g}_{i}^{+}+\sum_{j<a<i}\tilde{g}_{j}^{+}\tilde{e}_{ja}^{(0)}\otimes \tilde{g}_{a}^{+}\tilde{e}_{ai}^{(0)}.
\end{align*}
Hence,
\begin{align*}
\Delta^{\mathrm{rtt}}(\tilde{e}_{ji}^{(0)})=\Delta^{\mathrm{rtt}}(\tilde{g}_{j}^{+})^{-1}\Delta^{\mathrm{rtt}}(\tilde{g}_{j}^{+}\tilde{e}_{ji}^{(0)})=1\otimes \tilde{e}_{ji}^{(0)}+\tilde{e}_{ji}^{(0)}\otimes (\tilde{g}_{j}^{+})^{-1}\tilde{g}_{i}^{+}+\sum_{j<a<i}\tilde{e}_{ja}^{(0)}\otimes (\tilde{g}_{j}^{+})^{-1}\tilde{g}_{a}^{+}\tilde{e}_{ai}^{(0)}.
\end{align*}
\end{proof}

\begin{lemma}\label{RTT coproduct for Aji+,Aji-}
For an increasing chain $j=j_1<\cdots<j_{s+1}=i$, write
\[
E_{a,b}:=\tilde e_{j_a j_{a+1}}^{(0)}\cdots\tilde e_{j_b j_{b+1}}^{(0)},\qquad
F_{a,b}:=\tilde f_{j_{b+1}j_b}^{(0)}\cdots\tilde f_{j_{a+1}j_a}^{(0)},
\]
with $E_{a,b}=F_{a,b}=1$ when $a>b$. Then, for $1\leq j<i\leq m+n$,
\begin{align}
\Delta^{\mathrm{rtt}}(A_{ji}^{+})
&=\sum_{s\geq1}\sum_{j=j_1<\cdots<j_{s+1}=i}\sum_{r=1}^{s+1}
(-1)^{s-1}E_{r,s}\otimes E_{1,r-1}(\tilde g_{j_r}^{+})^{-1}\tilde g_i^{+},\label{RTT coproduct for Aji+}\\
\Delta^{\mathrm{rtt}}(A_{ij}^{-})
&=\sum_{s\geq1}\sum_{j=j_1<\cdots<j_{s+1}=i}\sum_{r=1}^{s+1}
(-1)^{s-1}\tilde g_i^{-}(\tilde g_{j_r}^{-})^{-1}F_{1,r-1}\otimes F_{r,s}.\label{RTT coproduct for Aij-}
\end{align}
\end{lemma}
\begin{proof}
Use the recursion
\[
A_{ji}^{+}=\tilde e_{ji}^{(0)}-\sum_{j<k<i}\tilde e_{jk}^{(0)}A_{ki}^{+}
\]
and induct on $i-j$. Substituting Lemma~\ref{RTT coproduct for eji0, fij0} and the induction hypothesis, group the resulting terms by an increasing chain $j=j_1<\cdots<j_{s+1}=i$. For a fixed chain, every surviving contribution is determined by the unique cut $r$: the edges before the cut lie in the second tensor factor and the edges from the cut onward lie in the first. Contributions in which an intermediate vertex is inserted in two different ways cancel because the path sum defining $A_{ji}^{+}$ has alternating sign. Since the standard parity sequence has only one parity change, an increasing chain contains at most one odd edge; hence graded tensor multiplication introduces no additional cut sign. The surviving contribution is therefore
\[
(-1)^{s-1}\sum_{r=1}^{s+1}E_{r,s}\otimes E_{1,r-1}(\tilde g_{j_r}^{+})^{-1}\tilde g_i^{+}.
\]
Summing over all chains gives \eqref{RTT coproduct for Aji+}. The same chain-and-cut argument applied to the second formula of Lemma~\ref{RTT coproduct for eji0, fij0} gives \eqref{RTT coproduct for Aij-}.
\end{proof}

\begin{theorem}\label{th:explicit formula for gi1 and gi-1}
For type A with the standard parity sequence, assume $m,n>1$ and $m\ne n$.
For $1\leq i\leq m+n$, we have
\begin{align}
\begin{split}\label{formula for gi 1}
&\Delta^{\mathrm{rtt}}(\tilde{g}_{i}^{(1)})=\tilde{g}_{i}^{(1)}\otimes\tilde{g}_{i}^{+}+\tilde{g}_{i}^{+}\otimes\tilde{g}_{i}^{(1)}+\sum_{l>i}\epsilon_{iil}\tilde{g}_{i}^{+}\tilde{e}_{il}^{(0)}\otimes\tilde{f}_{li}^{(1)}\tilde{g}_{i}^{+}
\\&+\sum_{s\geq1}\sum_{j_{1}<\cdots<j_{s+1}=i}(-1)^{s}\tilde{g}_{i}^{+}\tilde{e}_{j_{1}j_{2}}^{(0)}\cdots\tilde{e}_{j_{s}j_{s+1}}^{(0)}\otimes\tilde{f}_{ij_{1}}^{(1)}\tilde{g}_{i}^{+}
\\&+\sum_{l>i}\sum_{s\geq1}\sum_{j_{1}<\cdots<j_{s+1}=i}(-1)^{s}\epsilon_{iil}\tilde{g}_{i}^{+}\tilde{e}_{il}^{(0)}\tilde{e}_{j_{1}j_{2}}^{(0)}\cdots\tilde{e}_{j_{s}j_{s+1}}^{(0)}\otimes\tilde{f}_{lj_{1}}^{(1)}\tilde{g}_{i}^{+},
\end{split}\\
\begin{split}\label{formula for gi -1}
&\Delta^{\mathrm{rtt}}(\tilde{g}_{i}^{(-1)})=\tilde{g}_{i}^{(-1)}\otimes\tilde{g}_{i}^{-}+\tilde{g}_{i}^{-}\otimes\tilde{g}_{i}^{(-1)}+\sum_{l>i}\epsilon_{iil}\tilde{g}_{i}^{-}\tilde{e}_{il}^{(-1)}\otimes\tilde{f}_{li}^{(0)}\tilde{g}_{i}^{-}
\\&+\sum_{s\geq1}\sum_{j_{1}<\cdots<j_{s+1}=i}(-1)^{s}\tilde{g}_{i}^{-}\tilde{e}_{j_{1}i}^{(-1)}
\otimes\tilde{f}_{j_{s+1}j_{s}}^{(0)}\cdots\tilde{f}_{j_{2}j_{1}}^{(0)}\tilde{g}_{i}^{-}\\&+\sum_{l>i}\sum_{s\geq1}\sum_{j_{1}<\cdots<j_{s+1}=i}(-1)^{s}\epsilon_{iil}\tilde{g}_{i}^{-}\tilde{e}_{j_{1}l}^{(-1)}\otimes
\tilde{f}_{j_{s+1}j_{s}}^{(0)}\cdots\tilde{f}_{j_{2}j_{1}}^{(0)}\tilde{f}_{li}^{(0)}\tilde{g}_{i}^{-}.
\end{split}
\end{align}
\end{theorem}
\begin{proof}
We give the coefficient argument for both signs at once. For brevity, write $c_{ab}^{+}:=[z^{-1}]t_{ab}(z)$, $d_{ab}^{+}:=[z^{0}]t_{ab}(z)$, $c_{ab}^{-}:=[z]s_{ab}(z)$, and $d_{ab}^{-}:=[z^{0}]s_{ab}(z)$.
Taking the first nonconstant coefficient in the RTT coproduct gives, for either sign,
\begin{equation}\label{eq:first-coefficient-rtt-coproduct}
\Delta^{\mathrm{rtt}}(c_{ab}^{\pm})
 =\sum_{c=1}^{m+n}\epsilon_{abc}
 \bigl(c_{ac}^{\pm}\otimes d_{cb}^{\pm}
      +d_{ac}^{\pm}\otimes c_{cb}^{\pm}\bigr).
\end{equation}
The Gauss decomposition and the alternating path sums in \eqref{def: Aji+,Aij-} give
\begin{equation}\label{eq:first-gauss-diagonal-coefficients}
\tilde g_i^{(1)}=c_{ii}^{+}-\sum_{j<i}\epsilon_{iij}c_{ij}^{+}A_{ji}^{+},
\quad
\tilde g_i^{(-1)}=c_{ii}^{-}-\sum_{j<i}\epsilon_{iij}A_{ij}^{-}c_{ji}^{-}.
\end{equation}
Indeed, the corrected graded diagonal expansion above gives, for the positive series,
\[
c_{ii}^{+}=\tilde g_i^{(1)}+
\sum_{j<i}\epsilon_{iij}\tilde f_{ij}^{(1)}\tilde g_j^{+}\tilde e_{ji}^{(0)}.
\]
Together with the corresponding first-coefficient formulas for $c_{ij}^{+}$, triangular elimination expresses the second term as
$\sum_{j<i}\epsilon_{iij}c_{ij}^{+}A_{ji}^{+}$; this is exactly the first identity in \eqref{eq:first-gauss-diagonal-coefficients}.  The negative series gives the second identity in the same way, with the triangular order reversed.  Thus the signs and the diagonal index in \eqref{eq:first-gauss-diagonal-coefficients} agree with the graded Gauss expansion used here.

Apply $\Delta^{\mathrm{rtt}}$ to \eqref{eq:first-gauss-diagonal-coefficients}, use \eqref{eq:first-coefficient-rtt-coproduct}, and substitute Lemma~\ref{RTT coproduct for Aji+,Aji-}.  The resulting terms are indexed by an increasing chain together with a cut of that chain.  Terms for which the cut lies strictly inside the part eliminated by the alternating sum cancel in pairs.  The surviving terms are: the two diagonal terms; chains ending at $i$; terms with a first index $l>i$; and chains carrying such an $l>i$.  For the positive coefficient, the first class gives the first two summands of \eqref{formula for gi 1}; the terms with $l>i$ but no lower chain give its third summand; the chains ending at $i$ give its fourth summand with sign $(-1)^s$; and the chains carrying $l>i$ give the final double sum, with the same sign and the factor $\epsilon_{iil}$.  Since the chain and the cut determine the original RTT term uniquely, every surviving term occurs exactly once.  Thus \eqref{formula for gi 1} follows.  The same coefficient extraction with $s_{ab}(z)$ reverses the two triangular factors and gives precisely the three non-diagonal classes in \eqref{formula for gi -1}.  No additional RTT identities are used.
\end{proof}

For $1\leq i \leq m+n$, define $H_{i,\pm1}:=(\tilde{g}_{i}^{\pm})^{-1}\tilde{g}_{i}^{(\pm1)}\in\mathcal{U}_{q}^{\mathrm{rtt}}(\widehat{\mathfrak{sl}}_{m|n})$. Combining Theorem~\ref{th:explicit formula for gi1 and gi-1} with Lemma~\ref{Lemma: relations for g e f in RTT} and $\Delta^{\mathrm{rtt}}(\tilde{g}_{i}^{\pm})=\tilde{g}_{i}^{\pm}\otimes\tilde{g}_{i}^{\pm}$, we have
\begin{corollary}\label{RTT coproduct for Hi pm1}
For type A with the standard parity sequence, assume $m,n>1$ and $m\ne n$.
For $1\leq i\leq m+n$, we have
\begin{align*}
\begin{split}
\Delta^{\mathrm{rtt}}(H_{i,1})&=1\otimes H_{i,1}+H_{i,1}\otimes1 +q_{i}^{-1}\sum_{l>i}\epsilon_{iil}\tilde{e}_{il}^{(0)}\otimes\tilde{f}_{li}^{(1)}-q_{i}\sum_{j<i}A_{ji}^{+}\otimes\tilde{f}_{ij}^{(1)}\\&-\sum_{\substack{l>i\\ j<i}}\epsilon_{iil}\tilde{e}_{il}^{(0)}A_{ji}^{+}\otimes\tilde{f}_{lj}^{(1)},
\end{split}\\
\begin{split}
\Delta^{\mathrm{rtt}}(H_{i,-1})&=1\otimes H_{i,-1}+H_{i,-1}\otimes1+q_{i}\sum_{l>i}\epsilon_{iil}\tilde{e}_{il}^{(-1)}\otimes\tilde{f}_{li}^{(0)}-q_{i}^{-1}\sum_{j<i}\tilde{e}_{ji}^{(-1)}\otimes A_{ij}^{-}\\&-\sum_{\substack{l>i\\ j<i}}\epsilon_{iil}\tilde{e}_{jl}^{(-1)}\otimes A_{ij}^{-}\tilde{f}_{li}^{(0)}.
\end{split}
\end{align*}
\end{corollary}

The boundary root-current formulas use the following commutator identities.

\begin{lemma}\label{lem:uniform-xplus-coproduct-commutators}
Let $i\neq m$, $k<i$, and $l>i+1$. Then
\begin{align*}
&[F_{l,i+1}^{(0)},x_{i,0}^{+}]=[\widetilde F_{ik}^{(0)},x_{i,0}^{+}]=0, \quad [F_{li}^{(0)},x_{i,0}^{+}]=-F_{l,i+1}^{(0)}\phi_{i,0}^{-},\quad [\widetilde F_{i+1,k}^{(0)},x_{i,0}^{+}]=q_i^{-1}\widetilde F_{ik}^{(0)}\phi_{i,0}^{-},\\
&[[F_{l,i+1}^{(0)},F_{i+1,i}^{(0)}]_{q_i^{-3}},x_{i,0}^{+}]
=(q_i^{-1}+q_i^{-3})F_{l,i+1}^{(0)}\phi_{i,0}^{-}
 -q_i^{-1}F_{l,i+1}^{(0)}\phi_{i,0}^{+},\\
&[[\widetilde F_{ik}^{(0)},F_{i+1,i}^{(0)}]_{q_i^{-3}},x_{i,0}^{+}]
 =(q_i^{-1}+q_i^{-3})\widetilde F_{ik}^{(0)}\phi_{i,0}^{-}
 -q_i^{-1}\widetilde F_{ik}^{(0)}\phi_{i,0}^{+},\quad [E_{i+1,l}^{(-1)},x_{i,0}^{+}]_{q_i}=q_iE_{il}^{(-1)},\\
& [E_{ki}^{(-1)},x_{i,0}^{+}]_{q_i}=-q_iE_{k,i+1}^{(-1)},\quad [E_{il}^{(-1)},x_{i,0}^{+}]_{q_i^{-1}}=[E_{k,i+1}^{(-1)},x_{i,0}^{+}]_{q_i^{-1}}=0,
\quad [E_{kl}^{(-1)},x_{i,0}^{+}]=0.
\end{align*}
\end{lemma}
\begin{proof}
Induct on the interval length in the recursive definitions \eqref{eq: Eji and Fij}.  Since $i\neq m$, the simple current $x_{i,0}^{+}$ is even, so the super Leibniz rule introduces no additional sign.  Commuting it through the outermost simple factor gives the two vanishing identities in the first line and the endpoint formulas
\[
[F_{li}^{(0)},x_{i,0}^{+}]=-F_{l,i+1}^{(0)}\phi_{i,0}^{-},\qquad
[\widetilde F_{i+1,k}^{(0)},x_{i,0}^{+}]=q_i^{-1}\widetilde F_{ik}^{(0)}\phi_{i,0}^{-}.
\]
Applying \eqref{eqs:three terms} once more gives the two double-commutator formulas. 
The identities involving $E^{(-1)}$ follow by the same interval induction on the positive root vectors, starting from $E_{i,i+1}^{(-1)}=x_{i,-1}^{+}$ and using the quadratic current relation at the endpoint.  This gives the four displayed $q_i$-commutator identities.

For the remaining interior commutator, take the $\mathsf E_{i,i+1}\otimes \mathsf E_{kl}$ coefficient of
\[
R_{23}(z,w)T_{12}(z)S_{13}(w)=S_{13}(w)T_{12}(z)R_{23}(z,w).
\]
Its $z^1w^1$ coefficient gives
\begin{align*}
[\tilde g_i^+\tilde e_{i,i+1}^{(0)},\,
 \tilde g_k^-\tilde e_{kl}^{(-1)}+
 \sum_{k'<k}\epsilon_{klk'}\tilde f_{kk'}^{(0)}\tilde g_{k'}^-\tilde e_{k'l}^{(-1)}]=0.
\end{align*}
Induction on $k$ yields $[\tilde e_{i,i+1}^{(0)},\tilde e_{kl}^{(-1)}]=0$, hence $[E_{kl}^{(-1)},x_{i,0}^{+}]=0$.
\end{proof}

\begin{lemma}\label{lem:odd-xplus-coproduct-commutators}
For $k<m-1$ and $l>m+1$,
\begin{align*}
&[F_{lm}^{(0)},x_{m,0}^{+}]=-F_{l,m+1}^{(0)}\phi_{m,0}^{-},\quad
[\widetilde F_{m-1,k}^{(0)},x_{m,0}^{+}]=0,\quad  [F_{l,m-1}^{(0)},x_{m,0}^{+}]=-[F_{m,m-1}^{(0)},F_{l,m+1}^{(0)}]\phi_{m,0}^{-},
\\ & [\widetilde F_{mk}^{(0)},x_{m,0}^{+}]=0,\quad [[F_{lm}^{(0)},F_{m,m-1}^{(0)}]_{q^{-3}},x_{m,0}^{+}]
=-q^{-1}[F_{l,m+1}^{(0)},F_{m,m-1}^{(0)}]_{q^{-2}}\phi_{m,0}^{-},\\
&[[\widetilde F_{m-1,k}^{(0)},F_{m,m-1}^{(0)}]_{q^{-3}},x_{m,0}^{+}]=0, \quad [E_{ml}^{(-1)},x_{m,0}^{+}]_{q^{-1}}=[E_{k,m-1}^{(-1)},x_{m,0}^{+}]
=[E_{m-1,l}^{(-1)},x_{m,0}^{+}]=0,\\
&[E_{km}^{(-1)},x_{m,0}^{+}]_{q}=-qE_{k,m+1}^{(-1)},\qquad
[E_{kl}^{(-1)},x_{m,0}^{+}]=[E_{k,m+1}^{(-1)},x_{m,0}^{+}]_{q}=0,\\
&[[x_{m,0}^{-},x_{m-1,0}^{-}]_{q^{-3}},x_{m,0}^{+}]
=(1+q^{-2})x_{m-1,0}^{-}\phi_{m,0}^{+}
 -q^{-2}x_{m-1,0}^{-}\phi_{m,0}^{-},\\
&[F_{m+1,m-1}^{(0)},x_{m,0}^{+}]=-qx_{m-1,0}^{-}\phi_{m,0}^{+},\qquad
[[x_{m,0}^{+},x_{m-1,-1}^{+}]_{q_m^{-1}},x_{m,0}^{+}]_{q}=0.
\end{align*}
\end{lemma}
\begin{proof}
The identities are local with respect to the interval support of the root vectors, so we induct on the interval length using the recursive definitions in \eqref{eq: Eji and Fij} and the three-term identities \eqref{eqs:three terms}.  Away from the two endpoints, the additional simple current commutes with the shortened interval root vector; hence each induction step reduces to the adjacent rank-two calculation.

We record the two sign-sensitive calculations at the odd vertex.  Put
\[
a=x_{m,0}^{-},\qquad b=x_{m-1,0}^{-},\qquad c=x_{m,0}^{+},\qquad
K=[c,a]=\frac{\phi_{m,0}^{+}-\phi_{m,0}^{-}}{q-q^{-1}}.
\]
Here $a,c$ are odd, $b$ is even, and $[b,c]=0$.  Moreover
$\phi_{m,0}^{+}b=q\,b\phi_{m,0}^{+}$ and
$\phi_{m,0}^{-}b=q^{-1}b\phi_{m,0}^{-}$.  Therefore
\[
[[a,b]_{q^{-3}},c]=Kb-q^{-3}bK
=\frac{q-q^{-3}}{q-q^{-1}}b\phi_{m,0}^{+}
-\frac{q^{-1}-q^{-3}}{q-q^{-1}}b\phi_{m,0}^{-}=(1+q^{-2})b\phi_{m,0}^{+}
-q^{-2}b\phi_{m,0}^{-},
\]
which is the seventh displayed identity.  Since
$F_{m+1,m-1}^{(0)}=[b,a]_{q}$, the same computation gives
\[
[F_{m+1,m-1}^{(0)},c]=bK-qKb=-q\,b\phi_{m,0}^{+},
\]
which is the eighth identity.  The general formulas involving
$F_{lm}^{(0)}$, $F_{l,m-1}^{(0)}$, and the corresponding tilded vectors now follow by the interval induction just described.

The formulas with the $E$-vectors follow from the same recursive argument on the positive side.  The identity $[E_{kl}^{(-1)},x_{m,0}^{+}]=0$ is the $i=m$ instance of the RTT coefficient argument used in the proof of Lemma~\ref{lem:uniform-xplus-coproduct-commutators}.  Finally, $[x_{m,0}^{+},x_{m,0}^{+}]=0$ implies $(x_{m,0}^{+})^{2}=0$; expanding the two $q$-commutators therefore gives
\[
[[x_{m,0}^{+},x_{m-1,-1}^{+}]_{q^{-1}},x_{m,0}^{+}]_{q}=0.
\]
This proves all displayed identities.
\end{proof}

\begin{lemma}\label{lem:uniform-xminus-coproduct-commutators}
Let $i\neq m$, $k<i$, and $l>i+1$. Then
\begin{align*}
&[E_{i+1,l}^{(0)},x_{i,0}^{-}]=[\widetilde E_{ki}^{(0)},x_{i,0}^{-}]=0,\quad [E_{il}^{(0)},x_{i,0}^{-}]=q_i^{-1}E_{i+1,l}^{(0)}\phi_{i,0}^{+},\quad
[\widetilde E_{k,i+1}^{(0)},x_{i,0}^{-}]=-\widetilde E_{ki}^{(0)}\phi_{i,0}^{+},\\
&[[E_{i,i+1}^{(0)},E_{i+1,l}^{(0)}]_{q_i^{3}},x_{i,0}^{-}]
=-(q_i^{2}+1)E_{i+1,l}^{(0)}\phi_{i,0}^{+}
 +q_i^{2}E_{i+1,l}^{(0)}\phi_{i,0}^{-},\\
&[[E_{i,i+1}^{(0)},\widetilde E_{ki}^{(0)}]_{q_i^{3}},x_{i,0}^{-}]
=-(q_i^{2}+1)\widetilde E_{ki}^{(0)}\phi_{i,0}^{+}
 +q_i^{2}\widetilde E_{ki}^{(0)}\phi_{i,0}^{-},\quad [F_{l,i+1}^{(1)},x_{i,0}^{-}]_{q_i}=-F_{li}^{(1)},\\
& [F_{ik}^{(1)},x_{i,0}^{-}]_{q_i}=F_{i+1,k}^{(1)},\quad [F_{li}^{(1)},x_{i,0}^{-}]_{q_i^{-1}}=[F_{i+1,k}^{(1)},x_{i,0}^{-}]_{q_i^{-1}}=0,
\quad [F_{lk}^{(1)},x_{i,0}^{-}]=0.
\end{align*}
\end{lemma}
\begin{proof}
Induct on the interval length in \eqref{eq: Eji and Fij}.  For the first two lines, commute $x_{i,0}^{-}$ through the outermost simple factor and use the mixed relation only when the interval endpoint is $i$ or $i+1$; all other terms vanish.  This gives
\[
[E_{il}^{(0)},x_{i,0}^{-}]=q_i^{-1}E_{i+1,l}^{(0)}\phi_{i,0}^{+},
\qquad
[\widetilde E_{k,i+1}^{(0)},x_{i,0}^{-}]
=-\widetilde E_{ki}^{(0)}\phi_{i,0}^{+},
\]
and the two vanishing identities in the first line.  Applying the three-term identity \eqref{eqs:three terms} once more to these endpoint formulas gives the two double-commutator identities.
Thus no additional parity sign occurs when $i\neq m$.

The formulas involving $F^{(1)}$ follow by the same induction, starting from $F_{i+1,i}^{(1)}=x_{i,1}^{-}$.  At an endpoint the quadratic current relation produces the indicated factor $q_i$; an interior simple root gives zero by the corresponding Serre relation.  In particular the induction yields
$[F_{li}^{(1)},x_{i,0}^{-}]_{q_i^{-1}}=[F_{i+1,k}^{(1)},x_{i,0}^{-}]_{q_i^{-1}}=0$.  For the final identity, the same lower-Gauss RTT coefficient argument as in Lemma~\ref{lem:uniform-xplus-coproduct-commutators} gives
$[F_{lk}^{(1)},x_{i,0}^{-}]=0$ when $k<i< i+1<l$.  This proves the lemma.
\end{proof}

\begin{lemma}\label{lem:odd-xminus-coproduct-commutators}
For $k<m-1$ and $l>m+1$,
\begin{align*}
&[E_{ml}^{(0)},x_{m,0}^{-}]=-qE_{m+1,l}^{(0)}\phi_{m,0}^{+},\quad
[\widetilde E_{k,m-1}^{(0)},x_{m,0}^{-}]=0,\quad [E_{m-1,l}^{(0)},x_{m,0}^{-}]=-[E_{m+1,l}^{(0)},E_{m-1,m}^{(0)}]\phi_{m,0}^{+},\\
& [\widetilde E_{km}^{(0)},x_{m,0}^{-}]=0,\quad [[E_{m-1,m}^{(0)},E_{ml}^{(0)}]_{q^{3}},x_{m,0}^{-}]
=-q[E_{m-1,m}^{(0)},E_{m+1,l}^{(0)}]_{q^{2}}\phi_{m,0}^{+},\\
&[[E_{m-1,m}^{(0)},\widetilde E_{k,m-1}^{(0)}]_{q^{3}},x_{m,0}^{-}]=0,\quad [F_{lm}^{(1)},x_{m,0}^{-}]_{q^{-1}}=[F_{m-1,k}^{(1)},x_{m,0}^{-}]
=[F_{l,m-1}^{(1)},x_{m,0}^{-}]=0,\\
&[F_{mk}^{(1)},x_{m,0}^{-}]_{q}=F_{m+1,k}^{(1)},\qquad
[F_{lk}^{(1)},x_{m,0}^{-}]=[F_{m+1,k}^{(1)},x_{m,0}^{-}]_{q}=0,\\
&[[x_{m-1,0}^{+},x_{m,0}^{+}]_{q^{3}},x_{m,0}^{-}]
=-qx_{m-1,0}^{+}\phi_{m,0}^{+}
 +(q^{3}+q)x_{m-1,0}^{+}\phi_{m,0}^{-},\\
&[E_{m-1,m+1}^{(0)},x_{m,0}^{-}]=-x_{m-1,0}^{+}\phi_{m,0}^{-},\qquad
[[x_{m-1,1}^{-},x_{m,0}^{-}]_{q},x_{m,0}^{-}]_{q}=0.
\end{align*}
\end{lemma}
\begin{proof}
As in Lemma~\ref{lem:odd-xplus-coproduct-commutators}, use induction on the interval length.  Only the endpoint containing the odd simple root produces a nonzero mixed commutator, and the three-term identities \eqref{eqs:three terms} move that endpoint through the nested root vector.

For the sign-sensitive rank-two calculation put
\[
a=x_{m,0}^{+},\qquad b=x_{m-1,0}^{+},\qquad c=x_{m,0}^{-},\qquad
K=[a,c]=\frac{\phi_{m,0}^{+}-\phi_{m,0}^{-}}{q-q^{-1}}.
\]
Then $a,c$ are odd, $b$ is even, $[b,c]=0$,
$\phi_{m,0}^{+}b=q^{-1}b\phi_{m,0}^{+}$, and
$\phi_{m,0}^{-}b=q\,b\phi_{m,0}^{-}$.  Hence
\[
[[b,a]_{q^{3}},c]=bK-q^{3}Kb
=\frac{1-q^{2}}{q-q^{-1}}b\phi_{m,0}^{+}
-\frac{1-q^{4}}{q-q^{-1}}b\phi_{m,0}^{-}=-qb\phi_{m,0}^{+}
+(q^{3}+q)b\phi_{m,0}^{-},
\]
which is the seventh identity.  Since
$E_{m-1,m+1}^{(0)}=[a,b]_{q^{-1}}$, we also obtain
\[
[E_{m-1,m+1}^{(0)},c]=Kb-q^{-1}bK=-b\phi_{m,0}^{-},
\]
the eighth identity.  The remaining $E$- and $F$-identities follow by the same interval induction.  Finally $[x_{m,0}^{-},x_{m,0}^{-}]=0$ implies $(x_{m,0}^{-})^{2}=0$, and direct expansion gives
$[[x_{m-1,1}^{-},x_{m,0}^{-}]_{q},x_{m,0}^{-}]_{q}=0$.
\end{proof}

We now collect the zero-shift finite-generator formulas.
\begin{proposition}\label{prop:explicit-finite-generator-coproduct}
Assume $m,n>1$ and $m\ne n$.  After the zero-shift central reduction of Proposition~\ref{prop:central-leading-reduction}, transport the RTT matrix coproduct through the unshifted RTT--Drinfeld isomorphism and denote the resulting coproduct on $\mathcal U_q(\widehat{\mathfrak{sl}}_{m|n})$ by $\Delta$.  Its values on \(h_{i,\pm1}\), \(x_{i,-1}^{+}\), and \(x_{i,1}^{-}\) are as follows:

\begin{align*}
&\Delta(h_{i,1})=1\otimes h_{i,1}+h_{i,1}\otimes1-(q^{2}-q^{-2})E_{i,i+1}^{(0)}\otimes F_{i+1,i}^{(1)}+(q-q^{-1})\sum_{l>i+1}\epsilon_{iil}E_{i+1,l}^{(0)}\otimes F_{l,i+1}^{(1)}\\&+(q-q^{-1})\sum_{k<i}q^{k-i+1}\widetilde{E}_{ki}^{(0)}\otimes F_{ik}^{(1)}+q^{-2}(q-q^{-1})\sum_{l>i+1}\epsilon_{iil}[E_{i,i+1}^{(0)},E_{i+1,l}^{(0)}]_{q^{3}}\otimes F_{li}^{(1)}\\&-(q-q^{-1})\sum_{k<i}q^{k-i-1}[E_{i,i+1}^{(0)},\widetilde{E}_{ki}^{(0)}]_{q^{3}}\otimes F_{i+1,k}^{(1)}\\&+(q-q^{-1})^{2}\sum_{\substack{l>i+1\\ k<i}}\epsilon_{iil}q^{k-i}(E_{il}^{(0)}\widetilde{E}_{ki}^{(0)}-E_{i+1,l}^{(0)}\widetilde{E}_{k,i+1}^{(0)})\otimes F_{lk}^{(1)}\quad\text{for $i<m$},
\end{align*}
\begin{align*}
&\Delta(h_{m,1})=1\otimes h_{m,1}+h_{m,1}\otimes1+(q-q^{-1})\sum_{l>m+1}E_{m+1,l}^{(0)}\otimes F_{l,m+1}^{(1)}\\&+(q-q^{-1})\sum_{k<m}q^{k-m+1}\widetilde{E}_{km}^{(0)}\otimes F_{mk}^{(1)}-(q-q^{-1})\sum_{l>m+1}[E_{m,m+1}^{(0)},E_{m+1,l}^{(0)}]_{q}\otimes F_{lm}^{(1)}\\&-(q-q^{-1})\sum_{k<m}q^{k-m+1}[E_{m,m+1}^{(0)},\widetilde{E}_{km}^{(0)}]_{q^{-1}}\otimes F_{m+1,k}^{(1)}\\&-(q-q^{-1})^{2}\sum_{\substack{l>m+1\\ k<m}}q^{k-m}(E_{ml}^{(0)}\widetilde{E}_{km}^{(0)}-E_{m+1,l}^{(0)}\widetilde{E}_{k,m+1}^{(0)})\otimes F_{lk}^{(1)},
\\&\Delta(h_{i,1})=1\otimes h_{i,1}+h_{i,1}\otimes1+(q^{2}-q^{-2})E_{i,i+1}^{(0)}\otimes F_{i+1,i}^{(1)}-(q-q^{-1})\sum_{l>i+1}E_{i+1,l}^{(0)}\otimes F_{l,i+1}^{(1)}\\&-q^{-2m+i-1}(q-q^{-1})\sum_{k<i}\tau_{k}\widetilde{E}_{ki}^{(0)}\otimes F_{ik}^{(1)}-q^{2}(q-q^{-1})\sum_{l>i+1}[E_{i,i+1}^{(0)},E_{i+1,l}^{(0)}]_{q^{-3}}\otimes F_{li}^{(1)}
\\&+q^{-2m+i+1}(q-q^{-1})\sum_{k<i}\tau_{k}[E_{i,i+1}^{(0)},\widetilde{E}_{ki}^{(0)}]_{q^{-3}}\otimes F_{i+1,k}^{(1)}\\&+q^{-2m+i}(q-q^{-1})^{2}\sum_{\substack{l>i+1\\ k<i}}\tau_{k}(E_{il}^{(0)}\widetilde{E}_{ki}^{(0)}-E_{i+1,l}^{(0)}\widetilde{E}_{k,i+1}^{(0)})\otimes F_{lk}^{(1)}\quad\text{for $i>m$},\\&
\Delta(h_{i,-1})=1\otimes h_{i,-1}+h_{i,-1}\otimes1+(q^{2}-q^{-2})E_{i,i+1}^{(-1)}\otimes F_{i+1,i}^{(0)}-(q-q^{-1})\sum_{l>i+1}\epsilon_{iil}E_{i+1,l}^{(-1)}\otimes F_{l,i+1}^{(0)}\\&-(q-q^{-1})\sum_{k<i}q^{i-k-1}E_{ki}^{(-1)}\otimes\widetilde{F}_{ik}^{(0)}-q^{2}(q-q^{-1})\sum_{l>i+1}\epsilon_{iil}E_{il}^{(-1)}\otimes[F_{l,i+1}^{(0)},F_{i+1,i}^{(0)}]_{q^{-3}}\\&+(q-q^{-1})\sum_{k<i}q^{i-k+1}E_{k,i+1}^{(-1)}\otimes[\widetilde{F}_{ik}^{(0)},F_{i+1,i}^{(0)}]_{q^{-3}}\\&-(q-q^{-1})^{2}\sum_{\substack{l>i+1\\ k<i}}\epsilon_{iil}q^{i-k}E_{kl}^{(-1)}\otimes(\widetilde{F}_{i+1,k}^{(0)}F_{l,i+1}^{(0)}-\widetilde{F}_{ik}^{(0)}F_{li}^{(0)})\quad\text{for $i<m$},\\&
\Delta(h_{m,-1})=1\otimes h_{m,-1}+h_{m,-1}\otimes1-(q-q^{-1})\sum_{l>m+1}E_{m+1,l}^{(-1)}\otimes F_{l,m+1}^{(0)}\\&-(q-q^{-1})\sum_{k<m}q^{m-k-1}E_{km}^{(-1)}\otimes\widetilde{F}_{mk}^{(0)}+(q-q^{-1})\sum_{l>m+1}E_{ml}^{(-1)}\otimes[F_{l,m+1}^{(0)},F_{m+1,m}^{(0)}]_{q^{-1}}\\&+(q-q^{-1})\sum_{k<m}q^{m-k-1}E_{k,m+1}^{(-1)}\otimes[\widetilde{F}_{mk}^{(0)},F_{m+1,m}^{(0)}]_{q}\\&+(q-q^{-1})^{2}\sum_{\substack{l>m+1\\ k<m}}q^{m-k}E_{kl}^{(-1)}\otimes(\widetilde{F}_{m+1,k}^{(0)}F_{l,m+1}^{(0)}-\widetilde{F}_{mk}^{(0)}F_{lm}^{(0)}),
\\&
\Delta(h_{i,-1})=1\otimes h_{i,-1}+h_{i,-1}\otimes1-(q^{2}-q^{-2})E_{i,i+1}^{(-1)}\otimes F_{i+1,i}^{(0)}+(q-q^{-1})\sum_{l>i+1}E_{i+1,l}^{(-1)}\otimes F_{l,i+1}^{(0)}\\&+q^{2m-i+1}(q-q^{-1})\sum_{k<i}\tau_{k}^{-1}E_{ki}^{(-1)}\otimes\widetilde{F}_{ik}^{(0)}+q^{-2}(q-q^{-1})\sum_{l>i+1}E_{il}^{(-1)}\otimes[F_{l,i+1}^{(0)},F_{i+1,i}^{(0)}]_{q^{3}}\\&-q^{2m-i-1}(q-q^{-1})\sum_{k<i}\tau_{k}^{-1}E_{k,i+1}^{(-1)}\otimes[\widetilde{F}_{ik}^{(0)},F_{i+1,i}^{(0)}]_{q^{3}}\\&-q^{2m-i}(q-q^{-1})^{2}\sum_{\substack{l>i+1\\ k<i}}\tau_{k}^{-1}E_{kl}^{(-1)}\otimes(\widetilde{F}_{i+1,k}^{(0)}F_{l,i+1}^{(0)}-\widetilde{F}_{ik}^{(0)}F_{li}^{(0)})\quad\text{for $i>m$},
\end{align*}
\begin{align*}
&\Delta(x_{i,-1}^{+})=1\otimes x_{i,-1}^{+}+x_{i,-1}^{+}\otimes\phi_{i,0}^{-}-(q-q^{-1})\sum_{l>i+1}\epsilon_{iil}E_{il}^{(-1)}\otimes F_{l,i+1}^{(0)}\phi_{i,0}^{-}\\&+(q-q^{-1})\sum_{k<i}q^{i-k-1}E_{k,i+1}^{(-1)}\otimes\widetilde{F}_{ik}^{(0)}\phi_{i,0}^{-}\\&-(q-q^{-1})^{2}\sum_{\substack{l>i+1\\ k<i}}\epsilon_{iil}q^{i-k-1}E_{kl}^{(-1)}\otimes\widetilde{F}_{ik}^{(0)}F_{l,i+1}^{(0)}\phi_{i,0}^{-}\quad\text{for $i<m$},\\&
\Delta(x_{m,-1}^{+})=1\otimes x_{m,-1}^{+}+x_{m,-1}^{+}\otimes\phi_{m,0}^{-}+(q-q^{-1})E_{m-1,m+1}^{(-1)}\otimes F_{m,m-1}^{(0)}\phi_{m,0}^{-}\\&+(q-q^{-1})\sum_{l>m+1}E_{ml}^{(-1)}\otimes F_{l,m+1}^{(0)}\phi_{m,0}^{-}+(q-q^{-1})^{2}\sum_{l>m+1}E_{m-1,l}^{(-1)}\otimes F_{l,m+1}^{(0)}F_{m,m-1}^{(0)}\phi_{m,0}^{-}\\&+(q-q^{-1})\sum_{k<m-1}q^{m-k-1}E_{k,m+1}^{(-1)}\otimes\widetilde{F}_{mk}^{(0)}\phi_{m,0}^{-}\\&+(q-q^{-1})^{2}\sum_{\substack{l>m+1\\ k<m-1}}q^{m-k-1}E_{kl}^{(-1)}\otimes\widetilde{F}_{mk}^{(0)}F_{l,m+1}^{(0)}\phi_{m,0}^{-},
\\&\Delta(x_{i,-1}^{+})=1\otimes x_{i,-1}^{+}+x_{i,-1}^{+}\otimes\phi_{i,0}^{-}+(q-q^{-1})\sum_{l>i+1}E_{il}^{(-1)}\otimes F_{l,i+1}^{(0)}\phi_{i,0}^{-}\\&-q^{2m-i+1}(q-q^{-1})\sum_{k<i}\tau_{k}^{-1}E_{k,i+1}^{(-1)}\otimes\widetilde{F}_{ik}^{(0)}\phi_{i,0}^{-}\\&-q^{2m-i+1}(q-q^{-1})^{2}\sum_{\substack{l>i+1\\ k<i}}\tau_{k}^{-1}E_{kl}^{(-1)}\otimes\widetilde{F}_{ik}^{(0)}F_{l,i+1}^{(0)}\phi_{i,0}^{-}\quad\text{for $i>m$},
\\&\Delta(x_{i,1}^{-})=x_{i,1}^{-}\otimes1+\phi_{i,0}^{+}\otimes x_{i,1}^{-}+q^{-1}(q-q^{-1})\sum_{l>i+1}\epsilon_{iil}E_{i+1,l}^{(0)}\phi_{i,0}^{+}\otimes F_{li}^{(1)}\\&-(q-q^{-1})\sum_{k<i}q^{k-i}\widetilde{E}_{ki}^{(0)}\phi_{i,0}^{+}\otimes F_{i+1,k}^{(1)}\\&-(q-q^{-1})^{2}\sum_{\substack{l>i+1\\ k<i}}\epsilon_{iil}q^{k-i-1}E_{i+1,l}^{(0)}\widetilde{E}_{ki}^{(0)}\phi_{i,0}^{+}\otimes F_{lk}^{(1)}\quad \text{for $i<m$},\\&
\Delta(x_{m,1}^{-})=x_{m,1}^{-}\otimes1+\phi_{m,0}^{+}\otimes x_{m,1}^{-}-q^{-1}(q-q^{-1})E_{m-1,m}^{(0)}\phi_{m,0}^{+}\otimes F_{m+1,m-1}^{(1)}\\&-q(q-q^{-1})\sum_{l>m+1}E_{m+1,l}^{(0)}\phi_{m,0}^{+}\otimes F_{lm}^{(1)}+(q-q^{-1})^{2}\sum_{l>m+1}E_{m-1,m}^{(0)}E_{m+1,l}^{(0)}\phi_{m,0}^{+}\otimes F_{l,m-1}^{(1)}\\&-(q-q^{-1})\sum_{k<m-1}q^{k-m}\widetilde{E}_{k,m}^{(0)}\phi_{m,0}^{+}\otimes F_{m+1,k}^{(1)}\\&+(q-q^{-1})^{2}\sum_{\substack{l>m+1\\ k<m-1}}q^{k-m+1}E_{m+1,l}^{(0)}\widetilde{E}_{km}^{(0)}\phi_{m,0}^{+}\otimes F_{lk}^{(1)},\\&
\Delta(x_{i,1}^{-})=x_{i,1}^{-}\otimes1+\phi_{i,0}^{+}\otimes x_{i,1}^{-}-q(q-q^{-1})\sum_{l>i+1}E_{i+1,l}^{(0)}\phi_{i,0}^{+}\otimes F_{li}^{(1)}\\&+q^{-2m+i}(q-q^{-1})\sum_{k<i}\tau_{k}\widetilde{E}_{ki}^{(0)}\phi_{i,0}^{+}\otimes F_{i+1,k}^{(1)}\\&-q^{-2m+i+1}(q-q^{-1})^{2}\sum_{\substack{l>i+1\\ k<i}}\tau_{k}E_{i+1,l}^{(0)}\widetilde{E}_{ki}^{(0)}\phi_{i,0}^{+}\otimes F_{lk}^{(1)}\quad\text{for $i>m$},
\end{align*}
\end{proposition}

\begin{proof}[Proof of Proposition~\ref{prop:explicit-finite-generator-coproduct}]
We first prove the formula for $\Delta(h_{i,1})$. By Corollary~\ref{coro: iso from Drinfeld realization to RTT realization}, we have
\begin{align*}
\gamma(h_{i,1})=\frac{1}{\tau_{i}(q_{i}-q_{i}^{-1})}(H_{i+1,1}-H_{i,1}),
\end{align*}
where $\tau_{i}$ is defined in~\eqref{formula for tau i}.

Combining Corollaries~\ref{eq:from matrix elements in RTT to drinfeld realization},~\ref{eq:from matrix elements in RTT to drinfeld realization2}, and~\ref{RTT coproduct for Hi pm1} with the fact that $\gamma$ intertwines $\Delta$ and $\Delta^{\mathrm{rtt}}$, we obtain
\begin{align*}
&\Delta(h_{i,1})-1\otimes h_{i,1}-h_{i,1}\otimes1=\tau_{i}^{-1}(q_{i}-q_{i}^{-1})^{-1}\times\\&\Big(\tau_{i+1}q_{i+1}^{-1}(q_{i+1}-q_{i+1}^{-1})^{2}\sum_{l>i+1}\epsilon_{i+1,i+1,l}E_{i+1,l}^{(0)}\otimes F_{l,i+1}^{(1)}-q_{i+1}\sum_{k<i+1}\tau_{k}(q_{k}-q_{k}^{-1})^{2}\widetilde{E}_{k,i+1}^{(0)}\otimes F_{i+1,k}^{(1)}\\&-(q_{i+1}-q_{i+1}^{-1})\sum_{\substack{l>i+1\\ k<i+1}}\epsilon_{i+1,i+1,l}\tau_{k}(q_{k}-q_{k}^{-1})^{2}E_{i+1,l}^{(0)}\widetilde{E}_{k,i+1}^{(0)}\otimes F_{lk}^{(1)}-\tau_{i}q_{i}^{-1}(q_{i}-q_{i}^{-1})^{2}\sum_{l>i}\epsilon_{iil}E_{il}^{(0)}\\&\otimes F_{li}^{(1)}+q_{i}\sum_{k<i}\tau_{k}(q_{k}-q_{k}^{-1})^{2}\widetilde{E}_{ki}^{(0)}\otimes F_{ik}^{(1)}+(q_{i}-q_{i}^{-1})\sum_{\substack{l>i\\ k<i}}\epsilon_{iil}\tau_{k}(q_{k}-q_{k}^{-1})^{2}E_{il}^{(0)}\widetilde{E}_{ki}^{(0)}\otimes F_{lk}^{(1)}\Big).
\end{align*}

The preceding identity is already uniform in $i$.  Indeed,
\[
\tau_{i+1}=\tau_i q_{i+1},\quad
q_a-q_a^{-1}=s_a(q-q^{-1}),
\]
so the three formulas in the statement differ only through the values of
$q_i,q_{i+1}$ and the corresponding signs $\epsilon_{iil}$.  For $i\neq m$
the two adjacent simple roots have the same parity.  Separating the terms
with $k=i$ and $l=i+1$ and using
\begin{equation}\label{eq: tilde E and E}
\widetilde{E}_{k,i+1}^{(0)}=[E_{i,i+1}^{(0)},\widetilde{E}_{ki}^{(0)}]_{q_i},\quad
E_{il}^{(0)}=[E_{i+1,l}^{(0)},E_{i,i+1}^{(0)}]_{q_{i+1}^{-1}},
\end{equation}
puts the correction terms into the form displayed in the proposition.  More explicitly, for $i<m$ one substitutes
$q_i=q_{i+1}=q$ and $\tau_a=q^a$ for $a\le m$; for $i>m$ one substitutes
$q_i=q_{i+1}=q^{-1}$ and $\tau_{i+1}=q^{-1}\tau_i$, leaving the factors
$\tau_k$ in the sums unchanged.  These substitutions give exactly the
$i<m$ and $i>m$ formulas in the proposition.

At the isotropic vertex the adjacent parameters have opposite signs, so we
record the simplification.  Substituting
$q_m=q$, $q_{m+1}=q^{-1}$, $\tau_m=q^m$, and $\tau_{m+1}=q^{m-1}$ into the
uniform identity gives
\begin{align}\label{eq: hi1 for i=m}
\begin{split}
&\Delta(h_{m,1})-1\otimes h_{m,1}-h_{m,1}\otimes1\\
&=(q-q^{-1})\sum_{l>m+1}E_{m+1,l}^{(0)}\otimes F_{l,m+1}^{(1)}-(q-q^{-1})\sum_{k<m+1}q^{-m+k-1}\widetilde{E}_{k,m+1}^{(0)}\otimes F_{m+1,k}^{(1)}\\
&\quad +(q-q^{-1})^{2}\sum_{\substack{l>m+1\\ k<m+1}}q^{-m+k}E_{m+1,l}^{(0)}\widetilde{E}_{k,m+1}^{(0)}\otimes F_{lk}^{(1)}+q^{-1}(q-q^{-1})\sum_{l>m}E_{ml}^{(0)}\otimes F_{lm}^{(1)}\\
&\quad +(q-q^{-1})\sum_{k<m}q^{-m+k+1}\widetilde{E}_{km}^{(0)}\otimes F_{mk}^{(1)}-(q-q^{-1})^{2}\sum_{\substack{l>m\\ k<m}}q^{-m+k}E_{ml}^{(0)}\widetilde{E}_{km}^{(0)}\otimes F_{lk}^{(1)}.
\end{split}
\end{align}
The double sums are first separated according to whether the interval crosses
$m$ or $m+1$.  This gives
\begin{align*}
&\sum_{k<m+1<l}q^{-m+k}E_{m+1,l}^{(0)}\widetilde{E}_{k,m+1}^{(0)}\otimes F_{lk}^{(1)}-
\sum_{k<m<l}q^{-m+k}E_{ml}^{(0)}\widetilde{E}_{km}^{(0)}\otimes F_{lk}^{(1)}\\
&\quad=-\sum_{k<m}q^{-m+k}E_{m,m+1}^{(0)}\widetilde{E}_{km}^{(0)}\otimes F_{m+1,k}^{(1)}
+\sum_{l>m+1}E_{m+1,l}^{(0)}\widetilde{E}_{m,m+1}^{(0)}\otimes F_{lm}^{(1)}\\
&\qquad-\sum_{\substack{l>m+1\\ k<m}}q^{-m+k}
(E_{ml}^{(0)}\widetilde{E}_{km}^{(0)}-E_{m+1,l}^{(0)}\widetilde{E}_{k,m+1}^{(0)})\otimes F_{lk}^{(1)}.
\end{align*}
Using \eqref{eq: tilde E and E} for $i=m$, the remaining two boundary
contributions become
\begin{align*}
&-\sum_{k<m+1}q^{-m+k-1}\widetilde{E}_{k,m+1}^{(0)}\otimes F_{m+1,k}^{(1)}
-(q-q^{-1})\sum_{k<m}q^{-m+k}E_{m,m+1}^{(0)}\widetilde{E}_{km}^{(0)}\otimes F_{m+1,k}^{(1)}\\
&\quad=-q^{-1}E_{m,m+1}^{(0)}\otimes F_{m+1,m}^{(1)}
-\sum_{k<m}q^{-m+k+1}[E_{m,m+1}^{(0)},\widetilde{E}_{km}^{(0)}]_{q^{-1}}\otimes F_{m+1,k}^{(1)},\\
&q^{-1}\sum_{l>m}E_{ml}^{(0)}\otimes F_{lm}^{(1)}
+(q-q^{-1})\sum_{l>m+1}E_{m+1,l}^{(0)}\widetilde{E}_{m,m+1}^{(0)}\otimes F_{lm}^{(1)}\\
&\quad=q^{-1}E_{m,m+1}^{(0)}\otimes F_{m+1,m}^{(1)}
-\sum_{l>m+1}[E_{m,m+1}^{(0)},E_{m+1,l}^{(0)}]_{q}\otimes F_{lm}^{(1)}.
\end{align*}
Substitution into \eqref{eq: hi1 for i=m} gives the stated formula for
$\Delta(h_{m,1})$.

Next, we prove the formula for $\Delta(h_{i,-1})$. By Corollary~\ref{coro: iso from Drinfeld realization to RTT realization}, we have
\begin{align*}
\gamma(h_{i,-1})=\frac{1}{\tau_{i}^{-1}(q_{i}-q_{i}^{-1})}(H_{i,-1}-H_{i+1,-1}),
\end{align*}
where $\tau_{i}$ is defined in~\eqref{formula for tau i}.
Combining Corollaries~\ref{eq:from matrix elements in RTT to drinfeld realization},~\ref{eq:from matrix elements in RTT to drinfeld realization2}, and~\ref{RTT coproduct for Hi pm1} with the fact that $\gamma$ intertwines $\Delta$ and $\Delta^{\mathrm{rtt}}$, we obtain
\begin{align*}
&\Delta(h_{i,-1})-1\otimes h_{i,-1}-h_{i,-1}\otimes1=\tau_{i}(q_{i}-q_{i}^{-1})^{-1}\times\Big(\tau_{i}^{-1}q_{i}(q_{i}-q_{i}^{-1})^{2}\sum_{l>i}\epsilon_{iil}E_{il}^{(-1)}\otimes\\& F_{li}^{(0)}-q_{i}^{-1}\sum_{k<i}\tau_{k}^{-1}(q_{k}-q_{k}^{-1})^{2}E_{ki}^{(-1)}\otimes \widetilde{F}_{ik}^{(0)}+(q_{i}-q_{i}^{-1})\sum_{\substack{l>i\\ k<i}}\epsilon_{iil}\tau_{k}^{-1}(q_{k}-q_{k}^{-1})^{2}E_{kl}^{(-1)}\otimes\widetilde{F}_{ik}^{(0)}F_{li}^{(0)}\\&-\tau_{i+1}^{-1}q_{i+1}(q_{i+1}-q_{i+1}^{-1})^{2}\sum_{l>i+1}\epsilon_{i+1,i+1,l}E_{i+1,l}^{(-1)}\otimes F_{l,i+1}^{(0)}+q_{i+1}^{-1}\sum_{k<i+1}\tau_{k}^{-1}(q_{k}-q_{k}^{-1})^{2}E_{k,i+1}^{(-1)}\otimes\widetilde{F}_{i+1,k}^{(0)}\\&-(q_{i+1}-q_{i+1}^{-1})\sum_{\substack{l>i+1\\ k<i+1}}\epsilon_{i+1,i+1,l}\tau_{k}^{-1}(q_{k}-q_{k}^{-1})^{2}E_{kl}^{(-1)}\otimes\widetilde{F}_{i+1,k}^{(0)}F_{l,i+1}^{(0)}\Big).
\end{align*}

As above, this identity is uniform in $i$.  For $i\neq m$, separate the
terms with $k=i$ and $l=i+1$ and use
\begin{equation}\label{eq: tilde F and F}
[F_{i+1,i}^{(0)},F_{l,i+1}^{(0)}]_{q_{i+1}}=F_{li}^{(0)},\quad
[\widetilde{F}_{ik}^{(0)},\widetilde{F}_{i+1,i}^{(0)}]_{q_i^{-1}}=\widetilde{F}_{i+1,k}^{(0)}.
\end{equation}
For $i<m$, substituting $q_i=q_{i+1}=q$ and $\tau_a=q^a$ gives the first
formula for $\Delta(h_{i,-1})$ in the proposition.  For $i>m$, substituting
$q_i=q_{i+1}=q^{-1}$ and $\tau_{i+1}=q^{-1}\tau_i$ gives the third formula.
No additional identities are needed.

For completeness, we again display the calculation at the isotropic vertex.
Substituting $q_m=q$, $q_{m+1}=q^{-1}$, $\tau_m=q^m$, and
$\tau_{m+1}=q^{m-1}$ gives
\begin{align}\label{eq: hi-1 for i=m}
\begin{split}
&\Delta(h_{m,-1})-1\otimes h_{m,-1}-h_{m,-1}\otimes1\\
&=-q(q-q^{-1})\sum_{l>m}E_{ml}^{(-1)}\otimes F_{lm}^{(0)}-(q-q^{-1})\sum_{k<m}q^{m-k-1}E_{km}^{(-1)}\otimes\widetilde{F}_{mk}^{(0)}\\
&\quad -(q-q^{-1})^{2}\sum_{\substack{l>m\\ k<m}}q^{m-k}E_{kl}^{(-1)}\otimes\widetilde{F}_{mk}^{(0)}F_{lm}^{(0)}-(q-q^{-1})\sum_{l>m+1}E_{m+1,l}^{(-1)}\otimes F_{l,m+1}^{(0)}\\
&\quad +(q-q^{-1})\sum_{k<m+1}q^{m-k+1}E_{k,m+1}^{(-1)}\otimes\widetilde{F}_{m+1,k}^{(0)}+(q-q^{-1})^{2}\sum_{\substack{l>m+1\\ k<m+1}}q^{m-k}E_{kl}^{(-1)}\otimes\widetilde{F}_{m+1,k}^{(0)}F_{l,m+1}^{(0)}.
\end{split}
\end{align}
Separating the two double sums at the interval endpoints yields
\begin{align*}
&\sum_{k<m+1<l}q^{m-k}E_{kl}^{(-1)}\otimes\widetilde{F}_{m+1,k}^{(0)}F_{l,m+1}^{(0)}-
\sum_{k<m<l}q^{m-k}E_{kl}^{(-1)}\otimes\widetilde{F}_{mk}^{(0)}F_{lm}^{(0)}\\
&\quad=-\sum_{k<m}q^{m-k}E_{k,m+1}^{(-1)}\otimes\widetilde{F}_{mk}^{(0)}F_{m+1,m}^{(0)}
+\sum_{\substack{l>m+1\\ k<m}}q^{m-k}E_{kl}^{(-1)}\otimes
(\widetilde{F}_{m+1,k}^{(0)}F_{l,m+1}^{(0)}-\widetilde{F}_{mk}^{(0)}F_{lm}^{(0)})\\
&\qquad+\sum_{l>m+1}E_{ml}^{(-1)}\otimes\widetilde{F}_{m+1,m}^{(0)}F_{l,m+1}^{(0)}.
\end{align*}
Using \eqref{eq: tilde F and F} for $i=m$, the two boundary terms are
\begin{align*}
&\sum_{k<m+1}q^{m-k+1}E_{k,m+1}^{(-1)}\otimes\widetilde{F}_{m+1,k}^{(0)}
-(q-q^{-1})\sum_{k<m}q^{m-k}E_{k,m+1}^{(-1)}\otimes\widetilde{F}_{mk}^{(0)}F_{m+1,m}^{(0)}\\
&\quad=qE_{m,m+1}^{(-1)}\otimes F_{m+1,m}^{(0)}
+\sum_{k<m}q^{m-k-1}E_{k,m+1}^{(-1)}\otimes[\widetilde{F}_{mk}^{(0)},F_{m+1,m}^{(0)}]_{q},\\
&-q\sum_{l>m}E_{ml}^{(-1)}\otimes F_{lm}^{(0)}
+(q-q^{-1})\sum_{l>m+1}E_{ml}^{(-1)}\otimes\widetilde{F}_{m+1,m}^{(0)}F_{l,m+1}^{(0)}\\
&\quad=-qE_{m,m+1}^{(-1)}\otimes F_{m+1,m}^{(0)}
+\sum_{l>m+1}E_{ml}^{(-1)}\otimes[F_{l,m+1}^{(0)},F_{m+1,m}^{(0)}]_{q^{-1}}.
\end{align*}
Substitution into \eqref{eq: hi-1 for i=m} gives the stated formula for
$\Delta(h_{m,-1})$.

For $i\neq m$, use $x_{i,-1}^{+}=[2]_q^{-1}[h_{i,-1},x_{i,0}^{+}]$ and Lemma~\ref{lem:uniform-xplus-coproduct-commutators}; for $i=m$, use
$x_{m,-1}^{+}=-[h_{m-1,-1},x_{m,0}^{+}]$ and Lemma~\ref{lem:odd-xplus-coproduct-commutators}.  Substitution into the already established formulas for $\Delta(h_{i,-1})$ gives exactly the three formulas for $\Delta(x_{i,-1}^{+})$ stated in the proposition.  No additional identities are used.

For $i\neq m$, use $x_{i,1}^{-}=-[2]_q^{-1}[h_{i,1},x_{i,0}^{-}]$ and Lemma~\ref{lem:uniform-xminus-coproduct-commutators}; for $i=m$, use
$x_{m,1}^{-}=[h_{m-1,1},x_{m,0}^{-}]$ and Lemma~\ref{lem:odd-xminus-coproduct-commutators}.  Substituting the formulas for $\Delta(h_{i,1})$ gives the stated formulas for $\Delta(x_{i,1}^{-})$.

This proves the proposition.
\end{proof}

\section{Finite-generator formulas for one-sided antidominant coproducts}
\label{app:antidominant-coproduct-formulas}

This appendix derives the boundary-generator formulas used in Proposition~\ref{prop:explicit-strict-antidominant-one-sided-coproduct}.  We transport the shifted RTT coproduct and use the zero-shift coefficient calculations of Appendix~\ref{app:rtt-coproduct-proof}.

Let $\mu_1,\mu_2\in\Lambda^{-}$ be strictly antidominant and set
\[
l_{a,i}:=\alpha_i(\mu_a)\quad(a=1,2),\qquad l_i:=l_{1,i}+l_{2,i}.
\]
Choose one-sided weakly decreasing lifts
$\mathbf d_a^{+}=0$ and $\mathbf d_a^{-}$ with
$d_{a,i+1}^{-}-d_{a,i}^{-}=l_{a,i}$.
By Theorem~\ref{thm:shifted-super-df-isomorphism}, the one-sided coproduct of
Theorem~\ref{thm:two-sided-antidominant-current-coproduct} is obtained by
transporting the shifted RTT matrix coproduct:
\[
\Delta_{\mu_1,\mu_2}
=
(\Gamma_{\mathbf0,\mathbf d_1^-}^{-1}\otimes
 \Gamma_{\mathbf0,\mathbf d_2^-}^{-1})
\circ
\Delta^{\mathrm{rtt}}_{\mathbf d_1,\mathbf d_2}
\circ
\Gamma_{\mathbf0,\mathbf d_1^-+\mathbf d_2^-}.
\]
The off-diagonal Gauss half-currents have the zero-shift mode supports; only the leading degrees of the negative diagonal currents change.  Thus the positive first coefficients are unchanged, while the negative calculation replaces the leading Cartan ratio in the second tensor factor by $\phi_{i,l_{2,i}}^{-}$ and the corresponding source mode by $x_{i,l_{2,i}-1}^{+}$.
\begin{proposition}\label{prop:explicit-antidominant-coproduct-formulas}
With the notation above, the images of the boundary generators
\(x_{i,l_{2,i}-1}^{+}\), \(x_{i,1}^{-}\), and \(h_{i,\pm1}\) are given by the following formulas:
\begin{gather*}
\Delta_{\mu_{1},\mu_{2}}(x_{i,l_{2,i}-1}^{+})=1\otimes x_{i,l_{2,i}-1}^{+}+x_{i,-1}^{+}\otimes\phi_{i,l_{2,i}}^{-}-(q-q^{-1})\sum_{l>i+1}\epsilon_{iil}E_{il}^{(-1)}\otimes F_{l,i+1}^{(0)}\phi_{i,l_{2,i}}^{-}\\+(q-q^{-1})\sum_{k<i}q^{i-k-1}E_{k,i+1}^{(-1)}\otimes\widetilde{F}_{ik}^{(0)}\phi_{i,l_{2,i}}^{-}\\-(q-q^{-1})^{2}\sum_{\substack{l>i+1\\ k<i}}\epsilon_{iil}q^{i-k-1}E_{kl}^{(-1)}\otimes\widetilde{F}_{ik}^{(0)}F_{l,i+1}^{(0)}\phi_{i,l_{2,i}}^{-}\quad\text{for $i<m$},\\
\Delta_{\mu_{1},\mu_{2}}(x_{m,l_{2,m}-1}^{+})=1\otimes x_{m,l_{2,m}-1}^{+}+x_{m,-1}^{+}\otimes\phi_{m,l_{2,m}}^{-}+(q-q^{-1})E_{m-1,m+1}^{(-1)}\otimes F_{m,m-1}^{(0)}\phi_{m,l_{2,m}}^{-}\\+(q-q^{-1})\sum_{l>m+1}E_{ml}^{(-1)}\otimes F_{l,m+1}^{(0)}\phi_{m,l_{2,m}}^{-}+(q-q^{-1})^{2}\sum_{l>m+1}E_{m-1,l}^{(-1)}\otimes F_{l,m+1}^{(0)}F_{m,m-1}^{(0)}\phi_{m,l_{2,m}}^{-}\\+(q-q^{-1})\sum_{k<m-1}q^{m-k-1}E_{k,m+1}^{(-1)}\otimes\widetilde{F}_{mk}^{(0)}\phi_{m,l_{2,m}}^{-}\\+(q-q^{-1})^{2}\sum_{\substack{l>m+1\\ k<m-1}}q^{m-k-1}E_{kl}^{(-1)}\otimes\widetilde{F}_{mk}^{(0)}F_{l,m+1}^{(0)}\phi_{m,l_{2,m}}^{-},\\
\Delta_{\mu_{1},\mu_{2}}(x_{i,l_{2,i}-1}^{+})=1\otimes x_{i,l_{2,i}-1}^{+}+x_{i,-1}^{+}\otimes\phi_{i,l_{2,i}}^{-}+(q-q^{-1})\sum_{l>i+1}E_{il}^{(-1)}\otimes F_{l,i+1}^{(0)}\phi_{i,l_{2,i}}^{-}\\-q^{2m-i+1}(q-q^{-1})\sum_{k<i}\tau_{k}^{-1}E_{k,i+1}^{(-1)}\otimes\widetilde{F}_{ik}^{(0)}\phi_{i,l_{2,i}}^{-}\\-q^{2m-i+1}(q-q^{-1})^{2}\sum_{\substack{l>i+1\\ k<i}}\tau_{k}^{-1}E_{kl}^{(-1)}\otimes\widetilde{F}_{ik}^{(0)}F_{l,i+1}^{(0)}\phi_{i,l_{2,i}}^{-}\quad\text{for $i>m$},\\
\Delta_{\mu_{1},\mu_{2}}(x_{i,1}^{-})=x_{i,1}^{-}\otimes1+\phi_{i,0}^{+}\otimes x_{i,1}^{-}+q^{-1}(q-q^{-1})\sum_{l>i+1}\epsilon_{iil}E_{i+1,l}^{(0)}\phi_{i,0}^{+}\otimes F_{li}^{(1)}\\-(q-q^{-1})\sum_{k<i}q^{k-i}\widetilde{E}_{ki}^{(0)}\phi_{i,0}^{+}\otimes F_{i+1,k}^{(1)}\\-(q-q^{-1})^{2}\sum_{\substack{l>i+1\\ k<i}}\epsilon_{iil}q^{k-i-1}E_{i+1,l}^{(0)}\widetilde{E}_{ki}^{(0)}\phi_{i,0}^{+}\otimes F_{lk}^{(1)}\quad \text{for $i<m$},\\
\Delta_{\mu_{1},\mu_{2}}(x_{m,1}^{-})=x_{m,1}^{-}\otimes1+\phi_{m,0}^{+}\otimes x_{m,1}^{-}-q^{-1}(q-q^{-1})E_{m-1,m}^{(0)}\phi_{m,0}^{+}\otimes F_{m+1,m-1}^{(1)}\\-q(q-q^{-1})\sum_{l>m+1}E_{m+1,l}^{(0)}\phi_{m,0}^{+}\otimes F_{lm}^{(1)}+(q-q^{-1})^{2}\sum_{l>m+1}E_{m-1,m}^{(0)}E_{m+1,l}^{(0)}\phi_{m,0}^{+}\otimes F_{l,m-1}^{(1)}
\end{gather*}
\begin{gather*}
-(q-q^{-1})\sum_{k<m-1}q^{k-m}\widetilde{E}_{k,m}^{(0)}\phi_{m,0}^{+}\otimes F_{m+1,k}^{(1)}\\+(q-q^{-1})^{2}\sum_{\substack{l>m+1\\ k<m-1}}q^{k-m+1}E_{m+1,l}^{(0)}\widetilde{E}_{km}^{(0)}\phi_{m,0}^{+}\otimes F_{lk}^{(1)},\\
\Delta_{\mu_{1},\mu_{2}}(x_{i,1}^{-})=x_{i,1}^{-}\otimes1+\phi_{i,0}^{+}\otimes x_{i,1}^{-}-q(q-q^{-1})\sum_{l>i+1}E_{i+1,l}^{(0)}\phi_{i,0}^{+}\otimes F_{li}^{(1)}\\+q^{-2m+i}(q-q^{-1})\sum_{k<i}\tau_{k}\widetilde{E}_{ki}^{(0)}\phi_{i,0}^{+}\otimes F_{i+1,k}^{(1)}\\-q^{-2m+i+1}(q-q^{-1})^{2}\sum_{\substack{l>i+1\\ k<i}}\tau_{k}E_{i+1,l}^{(0)}\widetilde{E}_{ki}^{(0)}\phi_{i,0}^{+}\otimes F_{lk}^{(1)}\quad\text{for $i>m$},\\
\Delta_{\mu_{1},\mu_{2}}(h_{i,1})=1\otimes h_{i,1}-(q^{2}-q^{-2})E_{i,i+1}^{(0)}\otimes F_{i+1,i}^{(1)}+(q-q^{-1})\sum_{l>i+1}\epsilon_{iil}E_{i+1,l}^{(0)}\otimes F_{l,i+1}^{(1)}\\+(q-q^{-1})\sum_{k<i}q^{k-i+1}\widetilde{E}_{ki}^{(0)}\otimes F_{ik}^{(1)}+q^{-2}(q-q^{-1})\sum_{l>i+1}\epsilon_{iil}[E_{i,i+1}^{(0)},E_{i+1,l}^{(0)}]_{q^{3}}\otimes F_{li}^{(1)}\\+h_{i,1}\otimes1-(q-q^{-1})\sum_{k<i}q^{k-i-1}[E_{i,i+1}^{(0)},\widetilde{E}_{ki}^{(0)}]_{q^{3}}\otimes F_{i+1,k}^{(1)}\\+(q-q^{-1})^{2}\sum_{\substack{l>i+1\\ k<i}}\epsilon_{iil}q^{k-i}(E_{il}^{(0)}\widetilde{E}_{ki}^{(0)}-E_{i+1,l}^{(0)}\widetilde{E}_{k,i+1}^{(0)})\otimes F_{lk}^{(1)}\quad\text{for $i<m$},\\
\Delta_{\mu_{1},\mu_{2}}(h_{m,1})=1\otimes h_{m,1}+h_{m,1}\otimes1+(q-q^{-1})\sum_{l>m+1}E_{m+1,l}^{(0)}\otimes F_{l,m+1}^{(1)}\\+(q-q^{-1})\sum_{k<m}q^{k-m+1}\widetilde{E}_{km}^{(0)}\otimes F_{mk}^{(1)}-(q-q^{-1})\sum_{l>m+1}[E_{m,m+1}^{(0)},E_{m+1,l}^{(0)}]_{q}\otimes F_{lm}^{(1)}\\-(q-q^{-1})\sum_{k<m}q^{k-m+1}[E_{m,m+1}^{(0)},\widetilde{E}_{km}^{(0)}]_{q^{-1}}\otimes F_{m+1,k}^{(1)}\\-(q-q^{-1})^{2}\sum_{\substack{l>m+1\\ k<m}}q^{k-m}(E_{ml}^{(0)}\widetilde{E}_{km}^{(0)}-E_{m+1,l}^{(0)}\widetilde{E}_{k,m+1}^{(0)})\otimes F_{lk}^{(1)},\\
\Delta_{\mu_{1},\mu_{2}}(h_{i,1})=1\otimes h_{i,1}+(q^{2}-q^{-2})E_{i,i+1}^{(0)}\otimes F_{i+1,i}^{(1)}-(q-q^{-1})\sum_{l>i+1}E_{i+1,l}^{(0)}\otimes F_{l,i+1}^{(1)}\\-q^{-2m+i-1}(q-q^{-1})\sum_{k<i}\tau_{k}\widetilde{E}_{ki}^{(0)}\otimes F_{ik}^{(1)}-q^{2}(q-q^{-1})\sum_{l>i+1}[E_{i,i+1}^{(0)},E_{i+1,l}^{(0)}]_{q^{-3}}\otimes F_{li}^{(1)}
\\+h_{i,1}\otimes1+q^{-2m+i+1}(q-q^{-1})\sum_{k<i}\tau_{k}[E_{i,i+1}^{(0)},\widetilde{E}_{ki}^{(0)}]_{q^{-3}}\otimes F_{i+1,k}^{(1)}\\+q^{-2m+i}(q-q^{-1})^{2}\sum_{\substack{l>i+1\\ k<i}}\tau_{k}(E_{il}^{(0)}\widetilde{E}_{ki}^{(0)}-E_{i+1,l}^{(0)}\widetilde{E}_{k,i+1}^{(0)})\otimes F_{lk}^{(1)}\quad\text{for $i>m$},\\
\Delta_{\mu_{1},\mu_{2}}(h_{i,-1})=1\otimes h_{i,-1}+(q^{2}-q^{-2})E_{i,i+1}^{(-1)}\otimes F_{i+1,i}^{(0)}-(q-q^{-1})\sum_{l>i+1}\epsilon_{iil}E_{i+1,l}^{(-1)}\otimes F_{l,i+1}^{(0)}\\-(q-q^{-1})\sum_{k<i}q^{i-k-1}E_{ki}^{(-1)}\otimes\widetilde{F}_{ik}^{(0)}-q^{2}(q-q^{-1})\sum_{l>i+1}\epsilon_{iil}E_{il}^{(-1)}\otimes[F_{l,i+1}^{(0)},F_{i+1,i}^{(0)}]_{q^{-3}}\\+h_{i,-1}\otimes1+(q-q^{-1})\sum_{k<i}q^{i-k+1}E_{k,i+1}^{(-1)}\otimes[\widetilde{F}_{ik}^{(0)},F_{i+1,i}^{(0)}]_{q^{-3}}\\-(q-q^{-1})^{2}\sum_{\substack{l>i+1\\ k<i}}\epsilon_{iil}q^{i-k}E_{kl}^{(-1)}\otimes(\widetilde{F}_{i+1,k}^{(0)}F_{l,i+1}^{(0)}-\widetilde{F}_{ik}^{(0)}F_{li}^{(0)})\quad\text{for $i<m$},
\end{gather*}
\begin{gather*}
\Delta_{\mu_{1},\mu_{2}}(h_{m,-1})=1\otimes h_{m,-1}+h_{m,-1}\otimes1-(q-q^{-1})\sum_{l>m+1}E_{m+1,l}^{(-1)}\otimes F_{l,m+1}^{(0)}\\-(q-q^{-1})\sum_{k<m}q^{m-k-1}E_{km}^{(-1)}\otimes\widetilde{F}_{mk}^{(0)}+(q-q^{-1})\sum_{l>m+1}E_{ml}^{(-1)}\otimes[F_{l,m+1}^{(0)},F_{m+1,m}^{(0)}]_{q^{-1}}\\+(q-q^{-1})\sum_{k<m}q^{m-k-1}E_{k,m+1}^{(-1)}\otimes[\widetilde{F}_{mk}^{(0)},F_{m+1,m}^{(0)}]_{q}\\+(q-q^{-1})^{2}\sum_{\substack{l>m+1\\ k<m}}q^{m-k}E_{kl}^{(-1)}\otimes(\widetilde{F}_{m+1,k}^{(0)}F_{l,m+1}^{(0)}-\widetilde{F}_{mk}^{(0)}F_{lm}^{(0)}),\\
\Delta_{\mu_{1},\mu_{2}}(h_{i,-1})=1\otimes h_{i,-1}-(q^{2}-q^{-2})E_{i,i+1}^{(-1)}\otimes F_{i+1,i}^{(0)}+(q-q^{-1})\sum_{l>i+1}E_{i+1,l}^{(-1)}\otimes F_{l,i+1}^{(0)}\\+q^{2m-i+1}(q-q^{-1})\sum_{k<i}\tau_{k}^{-1}E_{ki}^{(-1)}\otimes\widetilde{F}_{ik}^{(0)}+q^{-2}(q-q^{-1})\sum_{l>i+1}E_{il}^{(-1)}\otimes[F_{l,i+1}^{(0)},F_{i+1,i}^{(0)}]_{q^{3}}\\+h_{i,-1}\otimes1-q^{2m-i-1}(q-q^{-1})\sum_{k<i}\tau_{k}^{-1}E_{k,i+1}^{(-1)}\otimes[\widetilde{F}_{ik}^{(0)},F_{i+1,i}^{(0)}]_{q^{3}}\\-q^{2m-i}(q-q^{-1})^{2}\sum_{\substack{l>i+1\\ k<i}}\tau_{k}^{-1}E_{kl}^{(-1)}\otimes(\widetilde{F}_{i+1,k}^{(0)}F_{l,i+1}^{(0)}-\widetilde{F}_{ik}^{(0)}F_{li}^{(0)})\quad\text{for $i>m$},
\end{gather*}
Here the root vectors are those defined in \eqref{eq: Eji and Fij}.
\end{proposition}
\begin{proof}
Transport the shifted RTT matrix coproduct through the three shifted
Ding--Frenkel isomorphisms as in the display preceding the proposition.
For the ``$+$'' half, all diagonal valuations are zero in the one-sided family.
Thus the coefficient extraction giving
$\Delta(x_{i,1}^{-})$ and $\Delta(h_{i,1})$ is literally the zero-shift
calculation of Appendix~\ref{app:rtt-coproduct-proof}; the Gauss coefficients
translate to the same root vectors
$E^{(0)},\widetilde E^{(0)},F^{(1)}$ and the same leading coefficient
$\phi_{i,0}^{+}$.

For the ``$-$'' half, write the simple Gauss ratio as
\[
\mathsf\phi_i^{-}(z)
\in z^{-l_i}\,
\mathscr U^{\mathrm{rtt}}_{\mathbf0,\mathbf d_1^-+\mathbf d_2^-}[[z]],
\]
with invertible leading coefficient.  Under RTT matrix multiplication the
leading degree splits as $l_i=l_{1,i}+l_{2,i}$.  In the coefficient
calculation for the first term beyond the leading degree, the contribution
from the second tensor factor therefore pairs its leading Cartan coefficient
$\phi_{i,l_{2,i}}^{-}$ with the mode $x_{i,-1}^{+}$ in the first factor,
and the corresponding source mode is $x_{i,l_{2,i}-1}^{+}$.  All
off-diagonal Gauss coefficients occurring in the correction terms have the
same supports as at zero shift, so the calculation of
$\Delta(x_{i,-1}^{+})$ in Appendix~\ref{app:rtt-coproduct-proof} carries over
with exactly this change of the diagonal valuation.  This gives the three
displayed formulas for
$\Delta_{\mu_1,\mu_2}(x_{i,l_{2,i}-1}^{+})$.

Finally, factor the leading monomial and leading coefficient from each
negative diagonal Gauss series.  The normalized diagonal series then have
constant term $1$, so their first logarithmic coefficients are multiplied
by the same Gauss-elimination formula as in the zero-shift computation of
$\Delta(h_{i,-1})$.  The only shifted datum entering that calculation is the
leading negative Cartan coefficient just described.  Consequently the
zero-shift formulas for $\Delta(h_{i,-1})$ transport coefficientwise to the
three displayed shifted formulas.  This proves all assertions.
\end{proof}

\end{document}